\documentclass[draft,oneside, reqno]{amsart}
\usepackage{amssymb,latexsym, eucal}
\usepackage{verbatim}

\newcommand \pref[1]  {\textbf{\ref{#1}}}
\newcommand \ptitle[1] { (\textbf{#1}) \ }

\newcommand\ttJ  {{\tt J}}
\newcommand\ttI {{\tt I}}
\newcommand\ttK {{\tt K}}
\newcommand\ttb  {{\tt b}}
\newcommand\ttp  {{\tt p}}

\newcommand \subperm[1]{{\bar #1}} 
\newcommand \rtid[1]{{\widetilde #1}}  
\newcommand \invd[1]{{\widehat #1}}       

\newcommand \prmi{\subperm{\imath}}
\newcommand \prmj{\subperm{\jmath}}
\newcommand \prmk{\subperm{k}}
\newcommand \prml{\subperm{\ell}}

\theoremstyle{plain}
\newtheorem{theorem}{Theorem}[subsection]
\newtheorem{lemma}[theorem]{Lemma}
\newtheorem{corollary}[theorem]{Corollary}
\newtheorem{proposition}[theorem]{Proposition}

\theoremstyle{definition}
\newtheorem{definition}[theorem]{Definition}

\newtheorem{example}[theorem]{Example}
\newtheorem{remark}[theorem]{Remark}
\newtheorem{assumptions}[theorem]{Assumptions}
\newtheorem{pgraph}[theorem]{{}}

\DeclareMathOperator\ad{ad}
\DeclareMathOperator\ATI{\mathsf{ATI}}
\DeclareMathOperator\AL{A}

\DeclareMathOperator\BC{BC}
\DeclareMathOperator{\Cent}{Cent}
\DeclareMathOperator\characteristic{char}
\DeclareMathOperator\CDer{CDer}
\DeclareMathOperator\Der{Der}
\DeclareMathOperator\diag{diag}
\DeclareMathOperator\End{End}

\DeclareMathOperator\GL{GL}

\DeclareMathOperator\Hom{Hom}

\DeclareMathOperator\iso{iso}
\DeclareMathOperator\Mat{Mat}
\DeclareMathOperator\Orth{O}
\DeclareMathOperator\rank{rank}
\DeclareMathOperator\rad{rad}
\DeclareMathOperator\SCDer{SCDer}
\DeclareMathOperator\speciallinear{\mathfrak{sl}}
\DeclareMathOperator\supp{supp}
\DeclareMathOperator\tr{tr}
\DeclareMathOperator\spann{span}

\newcommand \andd{\quad\text{and}\quad}
\newcommand \distu {\uplus} 
\newcommand \id {\text{id}}
\newcommand\order[1]{\left\vert #1 \right\vert}
\newcommand \set[1]{\{#1 \}}
\newcommand \suchthat { \mid }
\newcommand \zero {\{0\}}

\newcommand \F {\mathbb F} 
\newcommand \bbQ {\mathbb Q}
\newcommand \bbZ {\mathbb Z}

\newcommand \cA {\mathcal A}
\newcommand \cB {\mathcal B}
\newcommand \cD {\mathcal D}
\newcommand \cE {\mathcal E}
\newcommand \cF {\mathcal F}
\newcommand \cL {\mathcal L}
\newcommand \cM {\mathcal M}
\newcommand \cS {\mathcal S}
\newcommand \cU {\mathcal U}

\newcommand \fbu {{\mathfrak b}{\mathfrak u}}
\newcommand \fe {{\mathfrak e}}
\newcommand \ffu {{\mathfrak f\mathfrak u}}

\newcommand \fsu {{\mathfrak s \mathfrak u}}
\newcommand \fu {{\mathfrak u}}
\newcommand{\fz}{{\mathfrak z}}

\newcommand \al {\alpha}

\newcommand \ep {\varepsilon}
\newcommand \Dl {\Delta}

\newcommand \Dlind {\Delta_{\mathrm {ind}}}

\newcommand \gm {\gamma}
\newcommand\Gm{\Gamma}
\newcommand \lm {\lambda}
\newcommand\ph{\varphi}
\newcommand \sg {\sigma}

\newcommand\boldr{{\mathbf r}}
\newcommand\bolds{{\mathbf s}}

\newcommand\tM{\widetilde{M}}
\newcommand\tL{{\widetilde \Llat}}
\newcommand\tsg{{\tilde\sigma}}
\newcommand\ttau{{\tilde\tau}}

\newcommand \an{\text{an}}
\newcommand\egr{\text{ex}}
\newcommand\gr{\text{gr}}
\newcommand \hyp{\text{hyp}}
\newcommand\rgr{\text{rt}}

\newcommand \Br {\hbox{\rm B}_r}
\newcommand \BCr {\hbox{\rm BC}_r}
\newcommand \centre {Z} 
\newcommand \gd {{\mathfrak g}} 
\newcommand \gl {\mathfrak {gl}}
\newcommand \half{\frac 12}
\newcommand \G {\Lambda} 
\newcommand \hd {{\mathfrak h}} 

\newcommand \Llat {L}  
\newcommand \pmd {\negthickspace\negthickspace  \pmod{[\mathcal A,\mathcal A]}}
\newcommand \rk{r} 

\newcommand \vz{{v_0}}  
\newcommand \twor{{2r}} 
\newcommand \Slat{S} 
\newcommand \U{\mathsf{U}}
\newcommand \Mlat{M}
\newcommand\QL{\bbQ\Llat}
\newcommand\Ztwo{\mathbb Z_2}

\newcommand\tildemap{\widetilde{\phantom{a}}}

\newcommand\form{(\,\, \vert\,\, )}
\newcommand\Lmult{\mathsf{M}}
\newcommand\staraction{{}_{{}^\ast}}

\newcommand\rC{{\mathsf C}}
\newcommand\rD{{\mathsf D}}
\newcommand\rE{{\mathsf E}}
\newcommand\rH{{\mathsf H}}

\newcommand\rc{\textrm{c}}
\newcommand\grd{{\text{gr}*}} 

\font \cm = cmr8

\begin{document}

\title[Unitary Lie Algebras and Lie Tori]{Unitary Lie Algebras and Lie Tori
of Type BC$_{\boldsymbol{\large r}}$, ${\boldsymbol r\ge 3}$}
\author{Bruce Allison}
{\thanks{The authors gratefully acknowledge support from the following
grants: \ NSERC grant RGPIN 8465 and NSF grant \#{}DMS--0245082.} } \address{Department of Mathematics and Statistics \\ University
of Victoria \\  PO BOX 3060 STN CSC\\ Victoria, B.C.~Canada  V8W 3R4}
\email{ballison@uvic.ca}

\author{Georgia Benkart}\address{Department of Mathematics \\ University
of Wisconsin - Madison \\  Madison, Wisconsin 53706, USA}
\email{benkart@math.wisc.edu}

\date{November 18, 2008}


\begin{abstract}  A Lie $\G$-torus of type  $\hbox{\rm X}_r$ is a Lie algebra with two gradings -- one
 by an abelian group $\G$ and the other by
the root lattice of a finite irreducible root system of type  $\hbox{\rm X}_r$.
In this paper we  construct a centreless Lie $\G$-torus of type $\BCr$, which we call a \emph{unitary Lie $\G$-torus},
as it is  a special unitary  Lie algebra of a nondegenerate $\G$-graded hermitian form
of Witt index $r$ over an associative torus with involution.   We prove
a structure theorem for centreless Lie $\G$-tori of type $\BCr$, $\rk \ge 3$,
that states that any such Lie torus is bi-isomorphic to a unitary Lie $\G$-torus,
and we determine necessary and sufficient conditions for two unitary Lie $\G$-tori to be bi-isomorphic.
The motivation to investigate Lie $\G$-tori came from the theory of extended affine Lie
algebras, which are natural generalizations of the affine and toroidal Lie algebras.
Every extended affine Lie algebra possesses an ideal which is a Lie $n$-torus of
type $\hbox{\rm X}_r$ for  some irreducible root system $\hbox{\rm X}_r$, where by an $n$-torus
we mean that the group $\G$ is  a free abelian group of rank $n$ for some $n \geq 0$.   The structure theorem above enables us to
classify centreless Lie $n$-tori of type $\BCr$, $r\geq 3$.
We show that they are determined by pairs consisting of a quadratic form $\kappa$  on an $n$-dimensional
$\mathbb Z_2$-vector space and of an orbit of the orthogonal group of $\kappa$.    We use that
result  to construct extended affine Lie algebras of type $\BCr$, $r \geq 3$.
Our article completes a large project involving many earlier papers and many authors to
determine the centreless Lie $n$-tori of all types.
\end{abstract}
\maketitle
\vspace{-.2 truein}

\tableofcontents

\section[Introduction] {{\cm INTRODUCTION}}
\label{sec:intro}
In Lie theory it is known
that a finite-dimensional simple Lie
algebra over a (not necessarily algebraically closed)
field of characteristic 0 having a root system of type
$\BCr$  or $\Br$, $r\ge 3$,  relative
to a maximal split toral subalgebra
is isomorphic to the special unitary Lie algebra of a nondegenerate
hermitian form of Witt index~$r$ (\cite[Chap.~V]{S}, \cite{T}). In
this paper, we prove an infinite-dimensional graded analogue of that
result.  More specifically, we show that any centreless Lie torus of
type $\BCr$ for $\rk\ge 3$ is  bi-isomorphic to the special
unitary Lie algebra of a nondegenerate graded hermitian form
of Witt index
$r$.\footnote{For Lie tori,  and more generally for root graded Lie algebras,
type $\BCr$ contains type $\Br$ as a special case (see Remark \ref{rem:Btype}).  In contrast, for   finite-dimensional simple  Lie algebras and
for extended affine Lie algebras the convention
 is that type $\Br$ and type $\BCr$ are disjoint.
Consequently, extended affine Lie algebras of type $\Br$ \emph{and} $\BCr$ are constructed from Lie tori of type $\BCr$ (see
Section~\ref{subsec:EALA}).}

The  notion of a Lie torus was  first introduced by
Y.~Yoshii (\cite{Y3}, \cite{Y4}).   An equivalent definition was later
formulated by E.~Neher  in \cite{N1}.
By   definition, a \emph{Lie $\G$-torus} $\cL$ \emph {of type} $\Dl$
has two compatible gradings, one a \emph{root grading} by the root lattice of a finite irreducible
(not necessarily reduced) root system $\Dl$ and the other an \emph{external grading}
by an arbitrary abelian group $\G$.  Because of the double grading, there is a natural notion of equivalence for Lie tori called
\emph{bi-isomorphism}.

The motivation for the study of
Lie tori came from extended affine Lie algebras (EALAs), which are natural
generalizations of the affine and toroidal Lie algebras.  There is a construction due to Neher
of a family of EALAs starting from
a centreless Lie $\G$-torus with $\G$ isomorphic to $\bbZ^n$  for some $n\ge 0$ (or what is referred to as a centreless Lie $n$-torus).  Moreover,
any EALA occurs in one such family \cite{N2}.
So an understanding of the structure of centreless Lie $n$-tori  yields
a corresponding understanding of the structure of EALAs.  More generally,
for an arbitrary torsion-free abelian group   $\G$, centreless Lie $\G$-tori that possess an invariant form
can be used to construct what are called  \emph{invariant affine reflection algebras}.
(See  \cite[\S 6.7]{N3} and Remark \ref{rem:IARA} below.)

This paper is devoted to describing the structure of Lie $\G$-tori of type $\BCr$, $\rk \ge 3$.
In  Chapters \ref{sec:prelim}--\ref{sec:Herm} we present  the necessary background on root graded Lie algebras,
Lie tori, and associative tori with involution, as well as hermitian forms and unitary Lie algebras
over associative tori with involution.

Chapter  \ref{sec:main} contains the main results
of the paper.  We first construct a centreless Lie $\G$-torus of type $\BCr$, which we call a \emph{unitary Lie $\G$-torus},
as it is  a special unitary  Lie algebra $\cS$ of a nondegenerate $\G$-graded hermitian form
of Witt index $r$ over an associative torus with involution.  The root grading of $\cS$ is the root space decomposition relative to an ad-diagonalizable subalgebra of $\cS$, and the external grading is induced from the  $\G$-grading of the hermitian form.  We then prove
a structure theorem (Theorem \ref{thm:structure}) for centreless Lie $\G$-tori of type $\BCr$, $\rk \ge 3$,
that states that any such Lie torus is bi-isomorphic to a unitary Lie $\G$-torus.  Our  next
main result (Theorem \ref{thm:biisomorphism})
provides necessary and sufficient conditions for two unitary Lie $\G$-tori to be bi-isomorphic.

In Chapter \ref{sec:EALA} we specialize  to the case when
$\G$ is isomorphic to $\bbZ^n$.  We obtain a classification of  centreless Lie $n$-tori of type
$\BCr$, $r \geq 3$,  in Theorem \ref{thm:classn}.   These Lie tori are determined  by
pairs consisting of a quadratic form $\kappa$ on a $n$-dimensional vector space over
$\bbZ_2$ and an orbit under a certain action of the orthogonal group of~$\kappa$.  We
apply Neher's method
and our results to construct maximal EALAs of type $\Br$ and $\BCr$, $r \geq 3$.
Using coordinates in the grading group $\G$ ($\simeq \bbZ^n$),
we give explicit expressions
for the product and invariant bilinear form on the EALA, in the spirit of   \cite{BGK}.
Chapter \ref{sec:conclude} contains some concluding remarks and some possible directions
for future investigations.

By treating the case of type $\BCr$, $r \geq 3$, our structure theorem completes a program
to describe the centreless Lie $n$-tori of all types.
This effort, which has involved many authors, began in 1993 with the seminal paper on EALAs of type
$\AL_\rk$, $\rk \ge 3$, by Berman, Gao and Krylyuk \cite{BGK}. (For
an overview of the program and relevant references  see \cite[\S 7--11]{AF}.)

Our work on Lie algebras graded by the root system $\BCr$ was started together with
Yun Gao.     Due to other commitments, he felt he could not devote time to this present project
and urged us to proceed without him.
We value his contributions to our monograph  \cite{ABG}
and to the initial investigations that ultimately led to our present paper;
and we thank him for his enthusiastic support of our efforts.

\section[Preliminaries]{\cm PRELIMINARIES}
\label{sec:prelim}

\subsection{Notational conventions}\

We begin with some conventions and definitions that will be  used throughout the paper.

All algebras and vector spaces  are over $\F$,  \emph {a field of characteristic different from 2}.
In Sections \ref{subsec:BCunitary}, \ref{subsec:Lientori}, \ref{subsec:EALA} and
Chapters \ref{sec:RGLT}, \ref{sec:main},  \ref{sec:conclude}  we assume that  $\F$ has characteristic 0,
and we prove the main  results  of the paper under that hypothesis.
This additional assumption on the field will always be stated explicitly.

Unless indicated to the contrary, all associative algebras are  unital,
and by a module for an associative algebra $\cA$,  we mean a right module for $\cA$.
If $X$ is an $\cA$-module,  then $\gl_\cA(X)$ is the Lie algebra with underlying space
$\End_\cA(X)$  under the commutator product.
The  \emph{centre}  of an associative or Lie algebra $\cA$ is
denoted  by $\centre(\cA)$.  A Lie algebra $\cL$ is said to be
\emph{centreless} if  $\centre(\cL) = 0$.
The \emph{centroid} of any algebra $\cA$ is the associative algebra $\Cent(\cA)$
consisting of all endomorphisms of $\cA$ that commute with all left and right multiplications.
If $\cA$ is a unital  associative algebra, then $Z(\cA)$ and $\Cent(\cA)$ are isomorphic under the map
which sends $\al$ to left multiplication by~$\al$.

If $S$ is any subset of a group $\G$, then $\langle S \rangle$ stands for
the subgroup generated by~$S$.

\subsection{Associative algebras with involution and hermitian forms}\
\label{subsec:assocherm}

An \emph{associative algebra with involution} is a pair $(\cA,-)$
consisting of an associative algebra  $\cA$ and a period 2 anti-automorphism ``$-$''  of $\cA$.   We adopt  the notation
\begin{equation*}
\label{eq:Apm}
\cA_+ = \set{\al\in \cA \suchthat \overline \al  = \al} \andd
\cA_-  = \set{\al\in \cA \suchthat \overline \al  = -\al},
\end{equation*}
for the symmetric and skew-symmetric elements relative to
the involution, so that  $\cA = \cA_+ \oplus \cA_-$.
The \emph{centre} of $(\cA,-)$ is defined as
\[Z(\cA,-) = Z(\cA)\cap \cA_+.\]

If $(\cA,-)$ is an associative algebra with involution,
a map $\xi : X \times X \to \cA$ is called a
\emph{hermitian form} over $(\cA, -)$ if
$X$ is a (right)  $\cA$-module and
$\xi : X \times X \to \cA$
is a bi-additive map such that
\begin{equation*}
\label{eq:skewherm}
\xi(x.\al,\,y) = \overline\al \xi(x,y), \quad \xi(x, \,y. \al) = \xi(x,y)\al \andd \xi(y,x) = \overline{\xi(x,y)}
\end{equation*}
for $\al\in \cA$ and $x,y\in X$.  If $Y$  is an $\cA$-submodule of $X$, then
\[Y^\perp := \set{x\in X \suchthat \xi(x,y)  = 0 \text{ for all } y \in Y}\]
is an $\cA$-submodule of $X$.   The form $\xi$ is \emph{nondegenerate} if $X^\perp = 0$.
If $\al\in \cA$, we say that $\xi$ \emph{represents}
$\al$ if $\xi(x,x) = \al$ for some $x\in X$.

\subsection{Graded structures}\

Let $\G$ be an additive abelian group.
\begin{pgraph}
\label{pgraph:graded} We have the following basic terminology:

(a) A vector space $X$ over $\F$ is \emph{$\G$-graded}
if $X$ has a decomposition
$X = \bigoplus_{\sg\in \G} X^\sg$  into subspaces indexed
by $\G$.%
\footnote{We write the degrees of the graded spaces
as superscripts except in the case of root gradings (see Section \ref{subsec:rootgraded}),
where it  is more customary to use subscripts.}
If $x\in X$,  by  $\deg_\G(x) = \sg$ we mean that $x\in X^\sg$.
The $\G$-\emph{support} of  $X$ is
\[\supp_\G(X) = \set{\sg\in \G \suchthat X^\sg \ne 0}.\]
If the subgroup $\langle \supp_\G (X) \rangle$ of $\G$ generated by  $\supp_\G (X)$
equals $\G$,  then
$X$   is said to have \emph{full support} in
$\G$.%
\footnote{Often we assume a graded space has full support,
since if this condition is not satisfied, we can always replace $\G$ by the group
$\langle \supp_\G (X) \rangle$. }
When $\dim_\F(X^\sg)$ is finite for all $\sg\in \G$, then $X$ is said to have \emph{finite graded $\F$-dimension},  and
when  $\dim_\F X^\sg \le 1$ for all $\sg\in  \G$,  then  $X$ is called
\emph{finely $\G$-graded}.
If $L$ is a subgroup of an abelian group $\G$ and $X$ is  an $L$-graded vector space, we
regard $X$ as an $\G$-graded  vector space by setting $X^\sg = 0$  for $\sg\in \G \setminus L$.

(b)
An \emph{algebra $\cA$ is  $\G$-graded}   if
$\cA = \bigoplus_{\sg\in \G} \cA^\sg$  is graded as a vector space
and $\cA^\sg \cA^\tau \subseteq \cA^{\sg+\tau}$ for $\sg,\tau\in \cA$.

(c)
An \emph{associative algebra with involution
$(\cA,-)$ is
$\G$-graded} if $\cA$
is $\G$-graded as an algebra,  and the involution preserves the grading.
Then  $\cA_-$ and $\cA_+$  are graded subspaces of $\cA$,
and we set
\[\G_- = \G_-(\cA,-)= \supp_\G(\cA_-) \andd \G_+ = \G_+(\cA,-)= \supp_\G(\cA_+)\]
so that
\[\supp_\G(\cA)  = \G_- \cup \G_+.\]
Note that in general $\G_-$ and $\G_+$
may not be subgroups of $\G$.

(d)
If $\cA$ is a $\G$-graded algebra and $\cA'$ is a $\G'$-graded algebra,
an \emph{isograded-isomorphism} of  $\cA$ onto $\cA'$ is a pair
$(\ph,\ph_\gr)$,   where $\ph: \cA \to \cA'$ is an algebra isomorphism,
$\ph_\gr : \G \to \G'$ is a group isomorphism and
$\ph(\cA^\sg) =
\cA'^{\ph_\gr(\sg)}$ for $\sg\in \G$.
If such a pair exists, we say that
$\cA$ and $\cA'$ are  \emph{isograded-isomorphic}.
If $\cA$ has full support in $\G$,
then $\ph_\gr$ is determined by $\ph$ and we can
abbreviate the pair $(\ph,\ph_\gr)$  as $\ph$.
When  $\G = \G'$ and $\ph_\gr = {\rm id}$,
then $\cA$ and $\cA'$ are said to be  \emph{graded-isomorphic}.  The notions of
isograded-isomorphic and graded-isomorphic for graded associative algebras  with involution
are defined similarly (by insisting that the map $\ph$ respects the involutions).

(e)
If $\cA$ is a $\G$-graded associative algebra,
 \emph{an $\cA$-module $X$ is $\G$-graded}
if $X= \bigoplus_{\sg\in\G} X^\sg$ is graded  as a vector space and
$X^\sg . \cA^\tau \subseteq X^{\sg+\tau}$ for $\sg,\tau \in  \G$.

(f)
If $(\cA,-)$ is a $\G$-graded associative algebra with involution,
\emph{a hermitian form $\xi : X \times X \to \cA$ over $(\cA,-)$
is said to be $\G$-graded}  if the  $\cA$-module $X$ is $\G$-graded
and $\xi(X^\sg,X^\tau) \subseteq  \cA^{\sg+\tau}$ for $\sg,\tau\in  \G$.
We then say  that $\xi$ is \emph{of finite graded $\F$-dimension} (resp.~\emph{finely $\G$-graded})
if $X$ is \emph{of finite graded $\F$-dimension} (resp.~\emph{finely $\G$-graded}).

(g)
If $X$  is a $\G$-graded $\cA$-module,  we set
\[\End_\cA(X)^\sg = \set{ T\in \End_\cA(X) \suchthat T(X^\tau) \subseteq X^{\sg+\tau} \text{ for } \tau\in \G },\]
for $\sg\in \G$, and we let $\End^\gr_\cA(X) = \bigoplus_{\sg\in\G} \End_\cA(X)^\sg$.
Then $\End^\gr_\cA(X)$ is a $\G$-graded associative algebra under composition.
We further let $\gl^\gr_\cA(X)$ be the $\G$-graded Lie algebra with underlying graded space
$\End^\gr_\cA(X)$ under the commutator product.
We say that the gradings on $\End^\gr_\cA(X)$  and $\gl^\gr_\cA(X)$
 are \emph{induced} by the grading on $X$.
If  $\supp_\G(X)$ is finite or if $X$ is a finitely generated $\cA$-module,
then $\End_\cA(X)= \End^\gr_\cA(X)$ is $\G$-graded \cite[Cor.~2.4.4 and 2.4.5]{NvO},
and hence $\gl_\cA(X) = \gl^\gr_\cA(X)$ is  $\G$-graded.
\end{pgraph}

\section[Root graded Lie algebras and Lie tori]{\cm ROOT GRADED LIE ALGEBRAS AND LIE TORI}
\label{sec:RGLT}

Lie tori are root graded
Lie algebras with additional structure.
In this  chapter, we recall the notions of root graded Lie algebras
and Lie tori.

We suppose throughout the  chapter that   \emph{$\F$ is a field of characteristic 0 and that
$\Dl$ is a finite irreducible (not necessarily reduced)
root system  in a finite-dimensional vector space
$\F \Dl$ (defined for example as in
\cite[Chap. VI, \S 1, Def.~1]{Bo})}.\footnote{In \cite{N1} and in other papers on
Lie tori and extended affine Lie algebras, it has been convenient to adopt the convention that
0 is a root.  However,
for compatibility with \cite{ABG}, we do not do that here.}

\subsection{Root systems} \quad
\label{subsec:rootsystem} \smallskip

Our notation for root systems is standard.  Let
\[Q = Q(\Dl) := \spann_{\mathbb Z}(\Dl)\]
be the \emph{root lattice} of $\Dl$, and
\[\Dlind := \left\{\mu\in \Dl \,\Big | \, \textstyle \frac 12 \mu \notin \Dl\right\}\]
be the set of \emph{indivisible} roots in $\Dl$.
For
$\mu\in \Dl$, $\mu^\vee$ will denote the \emph{coroot} of
$\mu$. That is, $\mu^\vee$ is the element of the dual space
of the vector space $\F \Dl$ so that $\nu \mapsto \nu - \langle
\nu\mid \mu^\vee \rangle\mu$ is the reflection corresponding
to $\mu$ in the Weyl group of $\Dl$, where $\langle\,\, \mid \,\,
\rangle$ is the natural pairing of $\F \Dl$ with its
dual space.

The root system $\Dl$  has type $\text X_\rk$, where
$\text X_\rk = \text A_\rk$, $\text B_\rk$, $\text C_\rk$, $\text D_\rk$, $\text E_6$, $\text E_7$,
$\text E_8$, $\text F_4$, $\text G_2$ or $\BC_\rk$.  If X$_\rk \ne \BCr$, then
$\Dl$
is reduced  (that is $2\mu\notin \Dl$ for $\mu\in \Dl$) and $\Dl = \Dlind$.
On the other hand,  if $\text X_\rk = \BCr$, then $\Dlind$ is an irreducible root system
of type $\Br$ (see \pref{pgraph:BCr} below).

\begin{pgraph}
\label{pgraph:BCr}
If $\Dl$ has type $\BCr$,
we may choose a $\bbZ$-basis
$\ep_1,\dots,\ep_r$ for $Q$ so that
\begin{equation*}
\Dl = \set{\pm \ep_i\mid 1  \le i \le \rk}\cup\{\pm (\ep_i\pm\ep_j)\mid 1\le i < j \le \rk\}
\cup\set{\pm 2\ep_i\mid 1\le i \le \rk}\ \  \hbox{\rm and }
\end{equation*}
\[\Dlind =\set{\pm \ep_i\mid 1  \le i \le \rk}
\cup\{\pm (\ep_i\pm\ep_j)\mid 1\le i < j \le \rk\}.\]
For this basis,   we define a permutation $i\mapsto \prmi$
of $\set{1,\dots,\twor}$ by $\prmi = \twor+1-i$ and set
$\ep_{\bar\imath} = -\ep_i$ for $1\le i \le \rk$,
so that
\[\ep_{\bar \imath} =   -\ep_i \qquad \hbox{\rm for all} \ \ 1 \leq i \leq \twor, \ \ \hbox{\rm and} \]
\begin{equation*}
\label{eq:BCr}
\begin{aligned}
\Dl &= \set{\ep_i\mid 1  \le i \le 2\rk}\cup\{\ep_i+\ep_j\mid 1\le i,j\le 2\rk,\ j\ne \prmi\},\\
&= \set{\ep_i\mid 1  \le i \le 2\rk}\cup\{\ep_i+\ep_j\mid 1\le i \le j\le 2\rk,\ j\ne \prmi\},
\end{aligned}
\end{equation*}
where the expressions on the last line are unique.
\end{pgraph}
\begin{pgraph}
When $\Dl$ is of type $\BCr$ in  various examples
(such as in Sections  \ref{subsec:BCunitary} and \ref{subsec:BCconstruct} below), we do not
assume
a priori that a choice of
basis for $Q$ as in \pref{pgraph:BCr} has been made, but rather instead
use a basis that arises naturally.
\end{pgraph}

\subsection{Root graded Lie algebras} \quad
\label{subsec:rootgraded} \smallskip

\begin{definition} (\cite[Chap.~1]{ABG})
\label{def:BCgraded}
A \emph{$\Dl$-graded (or root graded) Lie algebra} (with grading subalgebra of type $\Dlind$)
is a $Q$-graded Lie algebra $\cL = \bigoplus_{\mu\in Q} \cL_{\mu}$ over $\F$ satisfying
\begin{description}
\item[\textbf{(RG1)}] $\cL$ has a split simple subalgebra $\gd$ with splitting Cartan subalgebra
$\hd$,  and there exists an $\F$-linear isomorphism $\mu\mapsto \rtid{\mu}$ of $\F\Dl$ onto $\hd^*$
such that the root system of $\gd$ with respect to $\hd$ is
$\rtid{\Dlind}$,  and such that

\noindent $\cL_{\mu} = \set{x\in \cL \suchthat [h,x] = \rtid{\mu}(h)x \text{ for } h\in \hd}$  for $\mu\in Q$;
\item[\textbf{(RG2)}] $\supp_Q(\cL) \subseteq \Dl\cup\set{0}$;
\item[\textbf{(RG3)}] $\cL$ is generated as a Lie algebra by the spaces $\cL_\mu$, $\mu\in \Dl$.
\end{description}
In that case, we say that $(\gd,\hd)$ is the \emph{grading pair} for $\cL$,
$\gd$ is the  \emph{grading (simple) subalgebra}
of $\cL$, and $\hd$ is the \emph{grading ad-diagonalizable subalgebra}
of $\cL$.
Also, when $\Delta$ has type X$_\rk$,   we often refer to a  $\Delta$-graded Lie algebra  as
an \emph{$\text X_\rk$-graded Lie algebra}.
\end{definition}

\begin{pgraph}
\label{pgraph:modproperties}
If $\cL$ is a $\Dl$-graded Lie algebra with grading pair $(\gd,\hd)$, then $Z(\cL)\subseteq \cL_0$
and $\cL/Z(\cL)$ is $\Dl$-graded Lie  algebra with the induced $Q$-grading
and with grading pair $(\pi(\gd),\pi(\hd))$, where $\pi : \cL \to \cL/Z(\cL)$ is the canonical map.
\end{pgraph}  \smallskip

In this paper we are primarily  interested in the case that  $\Dl$ is of type $\BCr$ for  $\rk\ge 3$.
Such $\Dl$-graded Lie algebras are described in \cite[Chap.~2 and 3]{ABG},
and we will recall that description in Section
\ref{subsec:BCunitary}.

\subsection{Lie tori} \label{subsec:LTandDG}   \qquad

Let \emph{$\G$ be an arbitrary additive abelian group and $Q$ be the root lattice of
a root system $\Dl$.}   If $\cL$ is a $(Q\times \G)$-graded Lie algebra,  we write
$\cL_\mu^\sg$ for the $(\mu,\sg)$-component of $\cL$
(rather than $\cL^{(\mu,\sg)}$ or $\cL_{(\mu,\sg)}$) and
adopt the notation
\[\cL^\sg = \bigoplus_{\mu\in Q} \cL_\mu^\sg \ \  \text{ for } \sg\in \G
\quad \andd \quad \cL_\mu = \bigoplus_{\sg\in \G} \cL_\mu^\sg \ \ \text{ for }\mu\in Q.\]
In this way,  $\cL =  \bigoplus_{\mu \in Q} \cL_\mu= \bigoplus_{\sg \in \G} \cL^\sg$ is both a $Q$-graded algebra and
a $\G$-graded algebra, and these gradings are compatible  in the sense that
each $\cL_\mu$ is $\G$-graded (or equivalently,  each $\cL^\sg$ is
$Q$-graded).  Conversely, compatible
gradings by  $Q$ and $\G$ on an algebra $\cL$ determine a $(Q\times \G)$-grading on $\cL$.

Next we present the definition of a Lie torus following \cite{N1}.
\begin{definition}
\label{def:Lietorus}
A \emph{Lie $\G$-torus of type $\Dl$} is a $(Q\times\G)$-graded Lie algebra
$\cL$ over $\F$
satisfying:
\begin{description}
\item[(LT1)]  $\supp_Q(\cL)\subseteq \Dl\cup \zero$.
\item[(LT2)]
\begin{itemize}
\item[(i)] $\cL_\mu^0 \ne 0$ for $\mu\in \Dlind$.
\item[(ii)] If $\mu\in \Dl$, $\sg\in \G$, and $\cL_\mu^\sg \ne 0$, then
$\cL_\mu^\sg = \F e_\mu^\sg$ and $\cL_{-\mu}^{-\sg} = \F f_\mu^\sg,$
where
\[[[e_\mu^\sg,f_\mu^\sg],x_\nu^\tau] =  \langle\nu\mid\mu^\vee\rangle x_\nu^\tau\]
for all $x_\nu^\tau\in \cL_\nu^\tau$, $\nu\in Q$, $\tau\in \G$.
\end{itemize}
\item[(LT3)] $\cL$ is generated as a Lie algebra by the spaces $\cL_\mu$, $\mu\in \Dl$.
\item[(LT4)] $\cL$ has full support in $\G$.
\end{description}
The $Q$-grading (resp.~the $\G$-grading) on $\cL$ is called
the \emph{root grading} (resp.~the \emph{external grading}) of $\cL$.
If $\Dl$ has type  $\text X_\rk$, we often refer to a Lie $\G$-torus
of type $\Dl$ as a \emph{Lie $\G$-torus of type $\text X_\rk$}.  We
use the term \emph{Lie torus} when it is not necessary to specify either $\G$ or $\Dl$.
\end{definition}

\begin{remark}
\label{rem:LTsupport} Suppose that $\cL$ is a Lie $\G$-torus of type $\Dl$.
\begin{itemize}
\item[(a)] It is known (see \cite{N1}) that $\cL$ is a $\Dl$-graded Lie algebra, as defined in Section \ref{subsec:rootgraded}.
(See Proposition \ref{prop:LTbasic}~(g) below for the case when $\cL$  is centreless.)
\item[(b)] By \cite[Lem.~1.1.10]{ABFP}, either
\[\supp_Q(\cL) = \Dl\cup\zero \ \ \hbox{\rm or} \ \ \supp_Q(\cL) = \Dlind \cup \zero. \]
\item[(c)]  If $\cL$ is centreless, then
the centroid $\Cent(\cL)$ of $\cL$ is a $\G$-graded subalgebra of $\End^\gr_\F(\cL)$,
and $\supp_\G(\Cent(\cL))$ is a subgroup of $\G$,
called the \emph{centroidal grading group} of $\cL$ \cite[Prop.~3.13]{BN}
\end{itemize}
\end{remark}

\begin{remark} \label{rem:Btype} Suppose that $\Dl$ is a root system of type
$\BCr$.
The Lie $\G$-tori of type $\Dlind$ are precisely the Lie $\G$-tori of type $\Dl$
whose $Q$-support equals   $\Dlind\cup\zero$. So the class of Lie $\G$-tori of type $\Br$
is contained in the class of Lie $\G$-tori of type~$\BCr$.
\end{remark}

We will use the following natural notion of equivalence for $(Q\times \G)$-graded
algebras and hence in particular for Lie tori  \cite[Sec.~2.1]{ABFP}.

\begin{definition}
\label{def:equivLT} Let $\cL$ be a $(Q\times \G)$-graded Lie algebra
and let $\cL'$ be a $(Q'\times \G')$-graded Lie algebra
(where here $\G$, $Q$, $\G'$,  and $Q'$ can be arbitrary abelian groups).
A \emph{bi-isograded-isomorphism}, or a \emph{bi-isomorphism} for short, of $\cL$ onto $\cL'$
is a triple $(\psi,\psi_\rgr,\psi_\egr)$, where
$\psi : \cL \to \cL'$ is an algebra isomorphism,
$\psi_\rgr: Q \to Q'$ and $\psi_\egr : \G \to \G'$ are group isomorphisms,
and $\psi(\cL_\mu^\sigma) = {\cL'}_{\psi_\rgr(\mu)}^{\psi_\egr(\sigma)}$
for $\mu\in Q$ and $\sigma \in \G$.  If such a triple exists,
we say that $\cL$ and $\cL'$ are bi-isomorphic.
If $\cL$ has full $(Q\times\G)$-support (for example when $\cL$ is a Lie torus),
then $\psi_\rgr$ and $\psi_\egr$ are determined by $\psi$,  and in that case we
abbreviate the triple $(\psi,\psi_\rgr,\psi_\egr)$  as $\psi$ .
\end{definition}

\begin{remark}
\label{rem:rootiso}
Suppose that $\psi$  is a bi-isomorphism of a Lie $\G$-torus $\cL$
of type $\Dl$ onto a Lie $\G'$-torus of type $\Dl'$.  Then $\psi_\rgr(\supp_Q(\cL)) = \supp_{Q'}(\cL')$.
Hence, by Remark \ref{rem:LTsupport}\,(b), if  $\Dl$ and $\Dl'$ are either both reduced or both non-reduced,
we have $\psi_\rgr(\Dl) = \Dl'$, so $\psi_\rgr$ is an isomorphism of the root system $\Dl$ onto the
root system $\Dl'$.
\end{remark}

\subsection{Basics on centreless Lie tori}
\label{subsec:basics}  \qquad

 \smallskip

\emph{Throughout this section,
we assume  that $\cL$ is a centreless Lie $\G$-torus of type $\Dl$.}

Following \cite{N1}, we let  $\gd$ denote the subalgebra of $\cL$ generated by
$\set{\cL_{\mu}^0}_{\mu\in\Dl}$ and set  $\hd = \sum_{\mu\in\Dl} [\cL_{\mu}^0,\cL_{-\mu}^0]$.

\begin{proposition}
\label{prop:LTbasic}\  Assume  that
$\cL = \bigoplus_{(\mu,\sg) \in Q \times \G} \cL_\mu^\sg$ is a centreless Lie $\G$-torus of type $\Dl$.
\begin{itemize}
\item[(a)] If  $\mu\in\Dlind$, then $\cL_{2\mu}^0 = 0$.
\item[(b)]  $\gd$ is a finite-dimensional split simple
Lie algebra with splitting Cartan subalgebra~$\hd$.
\item[(c)]  There is a unique linear isomorphism
$\mu \to \rtid{\mu}$ of $\F \Dl$ onto $\hd^*$
such that $\rtid{\Dlind}$
is the set of roots of $\gd$ relative to
$\hd$ and
$[e_\mu^0,f_\mu^0] = \rtid{\mu}^\vee$
for $\mu\in\Dlind$.
(Here $\rtid{\mu}^\vee\in (\hd^*)^* = \hd$.)
\item[(d)] If $\mu\in Q$, then
$\cL_\mu = \set{x\in \cL \suchthat [h,x] = \rtid{\mu}(h) x \text{ for } h\in \hd}$.
\item[(e)] If $\mu\in \Dl$, $\sg\in \G$,  and $\cL_\mu^\sg \ne 0$,
then
$[e_\mu^\sg,f_\mu^\sg] = \rtid{\mu}^\vee$.
\item[(f)]
$\gd = \cL^0$
and
$\hd = \cL^0_0.$
\item[(g)] As a $Q$-graded Lie  algebra,
$\cL$ is a $\Dl$-graded Lie algebra  with grading pair $(\gd,\hd)$.
\end{itemize}
\end{proposition}

\begin{proof}  Parts (a)--(e) and  (g) were announced
in \cite[Sec.~3]{N1} under the hypothesis that $\G$  is a finitely generated free abelian group.
A proof of (a)--(f) for arbitrary $\G$ can be found in \cite[Prop.~6.3]{ABFP}.  Part
(g) follows from (b), (c), (d) and (LT3).
\end{proof}

Henceforth, we will use the map $\mu \to \rtid{\mu}$   in
Proposition \ref{prop:LTbasic}\,(c) to identify $\F \Dl$ and $\hd^*$ and will
omit the tildes.
Thus, $\Dl$ is a root system in $\hd^*$ and $\Dlind$ is the set of roots of
$\gd$ relative to $\hd$.  Moreover, by Proposition \ref{prop:LTbasic}\,(d),
we have
\begin{equation}
\label{eq:Lroot}
\cL_\mu = \set{x\in \cL \suchthat [h,x] = \mu(h) x \text{ for } h\in \hd}
\end{equation}
for $\mu\in Q$.
Also, if $\cL_\mu^\sg \ne 0$ for $\mu\in \Dl$, $\sg\in \G$,
then by Proposition \ref{prop:LTbasic}\,(e),
\begin{equation*}
\label{eq:ef}
[e_\mu^\sg,f_\mu^\sg] = \mu^\vee,
\end{equation*}
where $\mu^\vee \in (\hd^*)^* = \hd$.  Thus,
$\set{e_\mu^\sg,\mu^\vee,f_{\mu}^{\sg}}$ is an $\speciallinear_2$-triple.

Following \cite[Sec.~II.2]{AABGP} and \cite{Y3},
we set
\[\G_\mu := \supp_\G(\cL_\mu) = \set{\sg\in \G \suchthat \cL_\mu^\sg \ne 0}\]
for  $\mu\in \Dl$.  The following facts are proved in \cite[Sec.~3]{Y3} (see
also \cite[Lem.~1.1.12]{ABFP})):

\begin{lemma}
\label{lem:suppfact}  Suppose that $\cL$ is a centreless Lie $\G$-torus of type $\Dl$ and  $\mu\in \Dl$.
Then
\begin{itemize}
\item[(a)] $\G_\mu$ depends only on the length of $\mu$.
\item[(b)] If $\mu\in\Dlind$, then $0\in \G_\mu$ and $-\G_\mu = \G_\mu$.
\item[(c)] If $\mu, \nu\in \Dl$ with $\G_\nu$ and $\G_\mu$ nonempty, then
$\G_\nu - \langle \nu \vert \mu^\vee\rangle \G_\mu \subseteq \G_{\nu -\langle \nu \vert \mu^\vee\rangle \mu}$.
\item[(d)] If $\mu$ has minimum length in $\Dlind$,
then  $\langle \G_\mu\rangle = \G$.
\end{itemize}
\end{lemma}

The next lemma is a consequence of  $\speciallinear_2$-theory.

\begin{lemma}
\label{lem:sl2}
If
$\mu,\nu,\mu+\nu\in \Dl$, $\sg\in \G_\mu$, $\tau\in \G_\nu$,  and
$\sg+\tau\in \G_{\mu + \nu}$, then
\begin{equation}
[\cL_\mu^\sg,\cL_\nu^\tau] = \cL_{\mu+\nu}^{\sg+\tau}.
\end{equation}
\end{lemma}

\begin{proof}  By assumption we have
\[\cL_\mu^\sg \ne 0,\quad \cL_\nu^\tau\ne 0 \andd \cL_{\mu+\nu}^{\sg+\tau} \ne 0.\]

Suppose first that $\nu\notin\bbZ \mu$. Let
$\cM = \sum_{k\in \bbZ} \cL_{\nu + k\mu}^{\tau+k\sg}$.
Then, by (LT1) and (LT2)(ii), $\cM$ is a finite-dimensional $\speciallinear_2(\mu,\sg)$-module,
where $\speciallinear_2(\mu,\sg)$ is the Lie algebra spanned by $\set{e_\mu^\sg,\mu^\vee,f_{\mu}^{\sg}}$.
Moreover, by (LT2)(ii),  this module has one-dimensional weight spaces,  and the eigenvalues
of $\ad(\mu^\vee)|_\cM$ are integers of the same parity.  Hence by $\speciallinear_2$-theory,
$\cM$ is irreducible and $\ad(e_\mu^\sg)  \cL_\nu^\tau = \cL_{\mu+\nu}^{\sg+\tau}$, proving the desired fact.

So we can assume that $\nu\in\bbZ \mu$ and similarly that  $\mu\in\bbZ \nu$.  Thus, $\nu = \pm \mu$, and, since
$\mu+\nu \in \Dl$, we have $\nu = \mu$.  Now $\ad(e_\mu^\sg) e_{2\mu}^{\sg +\tau} = 0$ by (LT1), so
$\ad(e_\mu^\sg) \ad(f_\mu^\sg) e_{2\mu}^{\sg +\tau} = \ad(\mu^\vee) e_{2\mu}^{\sg +\tau} = 4e_{2\mu}^{\sg +\tau}$.
Therefore $\ad(e_\mu^\sg) \cL_\mu^\tau \ne 0$,  again proving the conclusion.
\end{proof}

\begin{remark}
\label{rem:Weyl}  If $\omega : Q \to Q$ is in the Weyl group of $\Dl$, then there
exists an isograded isomorphism $\psi$ from $\cL$ to $\cL$
such that $\psi_\rgr = \omega$. (To see this,  one extends an inner
automorphism of $\gd$.  See the argument in the proof
of Lemma 3.8 of \cite{AF}.)
\end{remark}

\section[Unitary Lie algebras]{\cm UNITARY LIE ALGEBRAS}
\label{sec:unitarydef}

Throughout this  chapter   we assume that  \emph{$(\cA,-)$ is
an associative algebra with  involution}.

\subsection{The Lie algebras $\fu(X,\xi)$, $\ffu(X,\xi)$, and $
\fsu(X,\xi)$}
\label{subsec:special}

\begin{definition}
\label{def:unitary}
Suppose that $\xi : X \times X \to \cA$  is a hermitian form over
$(\cA,-)$.  To  construct unitary Lie algebras from $\xi$ we will use
an associative algebra with involution $(\cE,*)$ that is determined by $\xi$.
We recall the definition of   $(\cE,*)$, following \cite[Ex.~2.3]{A}, in (a) and (d) below.
\begin{itemize}
\item[(a)]  For $x,y\in X$,  define $E(x,y)\in \End_\cA(X)$ by
\[
E(x,y)z = x.\xi(y,z).
\]
Then   \begin{eqnarray}\xi(E(x,y)z,w) &=& \xi(z,E(y,x)w),
\label{eq:Eadj}
\\
E(x.\al,\,y) &=&  E(x,y.\overline \al), \andd \label{eq:Eprod}\\
E(x,y)E(z,w) &=& E(x. \xi(y,\,z),\,w)\label{eq:Eprod2}
\end{eqnarray}
hold for all $\al \in \cA$ and  $x,y,z,w\in X$.
We set
\[\cE = \fe(X,\xi) := \spann_\F\set{E(x,y) \suchthat x,y\in X},\]
and note that  by \eqref{eq:Eprod2},
$\cE$ is an associative subalgebra of $\End_\cA(X)$; however $\cE$
may not be unital.

\item[(b)]  Let
\begin{equation}  \cU = \fu(X,\xi) := \set{ T \in \End_\cA(X) \suchthat \xi(Tx,y) +
\xi(x,Ty)=0 \ \forall \ x,y\in X}.
\end{equation}
Then $\cU$ is a Lie subalgebra of $\gl_\cA(X)$, and we
say that $\cU$ is the \emph{unitary Lie algebra} of $\xi$.

\item[(c)]  For $x,y\in X$,  set
\begin{equation*}
\label{eq:Udef}
U(x,y) := E(x,y) - E(y,x),
\end{equation*}
and let
 \[\cF = \ffu(X,\xi) := U(X,X),\]
where $U(X,X) = \spann_\F\set{U(x,y) \suchthat x,y\in X}.$
It follows that
\begin{equation}
\label{eq:Uident1}
U(x.\al,\,y) = U(x,\,y.\overline \al) \andd U(x,y) = -U(y,x)
\end{equation}
for $\al\in \cA$ and $x,y\in X$, and, by \eqref{eq:Eadj},  that  $U(X,X)  \subseteq \cU$.
In particular, if $x\in X$ and $a\in \cA_+$, then
\begin{equation}
\label{eq:Uident2}
U(x.a, \, x) =
U(x, \,x.a) = 0.
\end{equation}
Moreover,
\begin{equation}
\label{eq:Uident3}
[T,U(x,y)] = U(Tx,y) + U(x,Ty)
\end{equation}
for $x,y\in X$ and $T\in \cU$,   so that $\cF$ is  an   ideal of the Lie algebra $\cU$,  referred to as
the \emph{finite unitary Lie algebra} of
$\xi$.

\item[(d)] Suppose $\xi$ is nondegenerate. It follows from \eqref{eq:Eadj} that there
is a well-defined linear map $*:\cE \to \cE$ of period 2 such that $E(x,y)^* = E(y,x)$ for $x,y\in X$.
Using \eqref{eq:Eprod} and  \eqref{eq:Eprod2},    it is easy to check that $*$ is an involution of $\cE$. Further,
by  \eqref{eq:Eadj}, we have
\begin{equation}
\label{eq:Eadj2}
\xi(Tx,y) = \xi(x,T^*y)
\end{equation}
for $T\in \cE$ and $x,y\in X$.  It is then clear that
\begin{equation}
\label{eq:Fchar1}
\cF = \set{T\in \cE \suchthat T^* = -T},
\end{equation}
and hence, using \eqref{eq:Eadj2} and nondegeneracy, that
\begin{equation}
\label{eq:Fchar}
\cF = \cU \cap \cE.
\end{equation}

\item[(e)]  Finally,  we let \[\cS = \fsu(X,\xi) = \cF^{(1)},\]
where $\cF^{(1)} = [\cF,\cF]$ denotes the derived algebra of $\cF$.  The ideal  $\cS$ of  $\cU$ is
called  the \emph{special unitary Lie algebra} of $\xi$.
\end{itemize}
\end{definition}

\begin{example}  Suppose that $\xi: X\times X \to \cA$  is a nondegenerate hermitian form
over an associative division algebra with involution $(\cA,-)$.
Then $\cE$ is the algebra of all finite rank endomorphisms in $\End_\cA(X)$ that have an adjoint
relative to $\xi$ \cite[Prop.~1, \S IV.8]{J}.
Thus, by \eqref{eq:Fchar}, $\cF$ is the Lie algebra of all finite rank endomorphisms in $\cU$.
\end{example}

\begin{pgraph}
\label{pgraph:Zaction}
If $\xi : X\times X \to \cA$ is a hermitian form
over $(\cA,-)$, then $\End_\cA(X)$ is a left module for the center  $Z(\cA,-) = Z(\cA) \cap \cA_+$
with the action given by
$(\fz T)(x) := T(x.\fz) = T(x).\fz$ for $\fz\in Z(\cA,-)$,
$T\in \End_\cA(X)$  and $x\in X$.  Then,
\[\fz E(x,y) = E(x.\fz,y) = E(x,y.\fz) \andd
\fz U(x,y) = U(x.\fz,y) = U(x,y.\fz)\]
for $x,y\in \cA$,
$\fz\in Z(\cA,-)$.
\end{pgraph}

\begin{pgraph}
\label{pgraph:gradedunitary} \ptitle{Gradings on $\cE$, $\cF$, $\cS$ and $\cU$}
Suppose that $\xi : X \times X \to \cA$
is a $\G$-graded hermitian form  over $(\cA,-)$.
Since $\cE$ is spanned by homogeneous elements of $\End_\cA^\gr(X)$,
$\cE$ is a $\G$-graded subalgebra of $\End^\gr_\cA(X)$  with
\[\cE^\tau = \textstyle \sum_{\rho + \sg = \tau} E(X^\rho,X^\sg)\]
for $\tau \in \G$.
Similarly,
$\cF$ is a $\G$-graded subalgebra of $\gl^\gr_\cA(X)$  with
\begin{equation}
\label{eq:gradF}
\cF^\tau =\textstyle \sum_{\rho + \sg = \tau} U(X^\rho,X^\sg)
\end{equation}
for $\tau\in \G$; and $\cS$ is a $\G$-graded subalgebra
of~$\cF$.   Moreover,  if $\supp_\G(X)$ is finite or if  $X$ is a finitely generated $\cA$-module,
then $\cU$ is a $\G$-graded subalgebra of
$\gl_\cA(X) = \gl^\gr_\cA(X)$.
\end{pgraph}

\smallskip

\subsection{The BC$_{\boldsymbol r}$-graded unitary Lie algebra $\fbu(X,\xi)$} \quad
\label{subsec:BCunitary} \smallskip

We have introduced three Lie algebras
$\fu(X,\xi)$, $\ffu(X,\xi)$ and $\fsu(X,\xi)$ associated with a hermitian
form $\xi$.  A fourth Lie algebra $\fbu(X,\xi)$ will play a key role
in the proof (but not the statement) of our  structure theorem about centreless Lie tori of type $\BCr$.

In defining $\fbu(X,\xi)$ we make the following assumptions:
\begin{itemize}
\item[{\rm(i)}] $\F$
has characteristic $0$;
\item[{\rm(ii)}]  $\rk\ge 1$;
\item[{\rm(iii)}]  $(\cA,-)$ is an associative algebra with involution over $\F$;
\item[{\rm(iv)}] $\xi : X \times X \to \cA$ is a hermitian form over $ (\cA,-)$
such that $X = X_\hyp \perp X_\an$,
where $X_\hyp$ and $X_\an$ are $\cA$-submodules
of $X$;\footnote{Later when we discuss unitary
Lie tori in Chapter \ref{sec:main}, the
decomposition $X = X_\hyp \perp X_\an$ will be a  Witt decomposition of $\xi$.}
\item[{\rm(v)}]  $X_\hyp$ has an $\cA$-basis
$\set{x_i}_{i=1}^\twor$  such that
\[\xi(x_i,x_j) = \delta_{i \prmj}\]
for $1\le i,j\le \twor$, where
\begin{equation*}
\label{eq:iopp}  \prmi = \twor+1-i;   \ \ \andd
\end{equation*}
\item[{\rm(vi)}]  $X_\an$ contains an element $\vz$ such that $\xi(\vz,\vz) = 1$.
\end{itemize} \smallskip

Let
\[\cU = \fu(X,\xi),\quad \cF = \ffu(X,\xi),  \andd \cS = \fsu(X,\xi)\]
as in Definition \ref{def:unitary}, and set
\[h_i = U(x_i,x_\prmi) = E(x_i,x_\prmi) - E(x_\prmi,x_i)\]
in $\cF$ for $1\le i \le \rk$.     It is straightforward to verify that
\begin{equation}
\label{eq:hi}
[U(x_i,\vz),U(\vz,x_\prmi)] = h_i,
\end{equation}
and hence $h_i\in \cS$ for \  $1\le i \le \rk$.  Also,
if $1\le i \le \rk$, then
\begin{equation}
\label{eq:hiaction}
h_i v= 0 \text{ for } v\in X_\an
\andd
h_i x_j = \left\{
\begin{array}{ll}
   \delta_{i j}x_j & \  \hbox{if \  $1\le j \le \rk$} \\
   -\delta_{i\prmj}x_j & \  \hbox{if  \  $\rk+1\le j \le \twor$} \\
\end{array}. \right.
\end{equation}

Now let
\[\hd = \textstyle\bigoplus_{i=1}^\rk\F h_i = \bigoplus_{i=1}^\rk \F U(x_i,x_\prmi).\]
By \eqref{eq:hiaction}, $\hd$ is an abelian subalgebra of $\cS$ with basis $\set{h_1,\dots,h_\rk}$,
and $\hd$ acts diagonally on $X$ under the natural action.
 Indeed, for $\mu\in\hd^*$,   suppose
\[X_\mu = \set{x\in X \suchthat hx = \mu(h)x \text{ for } h\in \hd}\]
is  the $\mu$-\emph{weight space} of $\hd$ in $X$.
Let $\set{\ep_1,\dots,\ep_\rk}$  be the dual basis in $\hd^*$ of $\set{h_1,\dots,h_\rk}$, and,
as in \pref{pgraph:BCr}, let $\ep_{\bar\imath} = -\ep_i$ for $1\le i \le \rk$.
Then
\begin{equation}
\label{eq:Xweightdecomp}
X =   \textstyle \big(\bigoplus_{i=1}^{2\rk} X_{\ep_i}\big)\oplus X_0,
\end{equation}
\begin{equation}
\label{eq:Xweightspace}
X_{\ep_i} = x_i.\cA \ \ \hbox{\rm for } \ \ 1\le i \le \twor,   \andd   X_0 =  X_\an.
\end{equation}

Let
\[Q = \bbZ \ep_1\oplus \dots \oplus \bbZ \ep_\rk\]
in $\hd^*$.  By \eqref{eq:Xweightdecomp}, $X$ is $Q$-graded as an $\cA$-module
(if we assign the trivial $Q$-grading $\cA = \cA_0$ to $\cA$), and we have
\[\supp_Q(X) = \set{0}\cup \set{\ep_i \suchthat 1\le i \le 2\rk}.\]
Since $\supp_Q(X)$ is finite,  it follows that the Lie algebra
$\gl_\cA(X) = \bigoplus_{\mu\in Q}\gl_\cA(X)_\mu$
is $Q$-graded with
\begin{equation*}
\label{eq:QgradeEnd}
\gl_\cA(X)_\mu = \set{ T\in \gl_\cA(X) \suchthat TX_\nu \subseteq X_{\mu+\nu} \text{ for } \nu\in Q}
\end{equation*}
(see \pref{pgraph:graded}~(g)).
Moreover, $\xi$ is $Q$-graded (again assigning the trivial $Q$-grading to $\cA$), so
\[\text{$\cU$, $\cF$,  and $\cS$ are $Q$-graded subalgebras of $\gl_\cA(X)$}\]
(see \pref{pgraph:gradedunitary}).
Notice that  $\supp_Q(\gl_\cA(X)) \subseteq \set{\mu-\nu \suchthat \mu, \nu \in \supp_Q(X)}$, so we have
$\supp_Q(\gl_\cA(X)) \subseteq \Dl \cup \set{0}$, where
\[
\Dl = \set{\ep_i\mid 1  \le i \le 2\rk}\cup\{\ep_i+\ep_j\mid 1\le i,j\le 2\rk,\ j\ne \prmi\}
\]
Thus the supports of $\cS$, $\cF$, and $\cU$ are also contained in $\Dl\cup\set{0}$.
Observe also  that $\Dl$ is a root system of type $\BCr$ in $\hd^*$, and $Q = Q(\Dl)$ is the root lattice of $\Dl$.

We are now ready to  introduce  the Lie algebra $\fbu(X,\xi)$.

\begin{definition}
\label{def:UBC} Let
\[\cB = \fbu(X,\xi) := \langle \cU_\mu : \mu\in \Dl\rangle_{\text{alg}}
= \sum_{\mu\in \Dl}\cU_\mu +  \sum_{\mu\in \Dl}[\cU_\mu,\cU_{-\mu}].
\]
Then $\cB$ is a $Q$-graded ideal of $\cS$  which we call
the \emph{$\BCr$-graded unitary Lie algebra} determined by $\xi$.
\end{definition}

\begin{pgraph}
\label{pgraph:Qgrade}
The Lie algebras
$\cB$, $\cS$   and $\cF$ are $Q$-graded ideals of $\cU$ and
\[\hd \subseteq \cB \subseteq \cS \subseteq \cF \subseteq \cU\]
by \eqref{eq:hi}.  So
the $Q$-gradings on the Lie algebras $\cB$, $\cS$, $\cF$, and $\cU$
are the \emph{root gradings relative to the adjoint action of $\hd$}.
\end{pgraph}

\begin{pgraph}  It should be noted that, unlike $\cF$ and $\cS$,
the algebra $\cB$   depends not only on $(\cA,-)$, $X$ and $\xi$,
but also on a decomposition $X = X_\hyp \perp X_\an$, an
$\cA$-basis $\set{x_i}_{i=1}^{\twor}$ for $X_\hyp$, and a distinguished element $\vz\in X_\an$.
However, for simplicity we have suppressed this in the notation
$\fbu(X,\xi)$.
\end{pgraph}

The terminology in Definition \ref{def:UBC} is justified by the following
theorem that comes from  \cite{ABG}.

\begin{theorem}
\label{thm:BCgraded}  Let $\F$ be a field of characteristic 0.
\begin{itemize}
\item[ {\rm (a)}] Suppose   $\rk\ge 1$;
$\xi : X \times X \to \cA$ is a hermitian form
over an associative algebra with involution $(\cA,-)$ such
that $X = X_\hyp \perp X_\an$ where $X_\hyp$ and $X_\an$ are $\cA$-submodules of $X$;
$X_\hyp$ has an
$\cA$-basis $\{x_i\}_{i=1}^{\twor}$  such that  $\xi(x_i,x_j) =
\delta_{i \prmj}$  for all $i,j =1,\dots,\twor$, where $\bar \imath = 2r+1-i$;  and
there is an element  $\vz \in X_\an$ with $\xi(\vz,\vz) = 1$.
Then $\cB = \fbu(X,\xi)$  is a $\BCr$-graded Lie algebra with grading pair $(\gd,\hd)$,
where
\begin{gather}\label{eq:UBCgradesub}
\gd =
\Bigg(\bigoplus_{1 \le i < j \le \twor} \F\, U(x_i,x_j)\Bigg)
\bigoplus
\Bigg(\bigoplus_{i=1}^\twor \F\, U(x_i,\vz)\Bigg),\\
\label{eq:UBChsub}  \hd =
\bigoplus_{i=1}^\rk \F\, U(x_i,x_\prmi).
\end{gather}

\item[{\rm(b)}]  Conversely, if $\rk\ge 3$ and $\cL$ is a $\BCr$-graded Lie algebra,
then $\cL/Z(\cL)$ is isomorphic to $\cB/Z(\cB)$ for some
$\BCr$-graded unitary Lie algebra  $\cB$ as in~\hbox{\rm(a)}.
More precisely, if $\cL$ is a $\BCr$-graded Lie algebra for $\rk \geq 3$
with grading pair $(\gd_\cL,\hd_\cL)$ then there exists
a $\BCr$-graded unitary Lie algebra $\cB$ (as in \hbox{\rm(a)})  and a
Lie algebra isomorphism $\psi : \cL/Z(\cL) \to
\cB/Z(\cB)$ so that $\psi$ maps the canonical image
of $(\gd_\cL,\hd_\cL)$ in $\cL/Z(\cL)$ onto the canonical image of
$(\gd,\hd)$ in $\cB/Z(\cB)$, where $\gd$ and $\hd$  are defined by \eqref{eq:UBCgradesub}
and \eqref{eq:UBChsub}.
\end{itemize}
\end{theorem}

\begin{proof} The results of \cite{ABG} provide the proof of this theorem, but a translation of
notation is required to apply them.
Let $C = (\vz.\cA)^\perp$ in $X_\an$, and identify
$X_\hyp\oplus \vz.\cA$ with $\cA^{2\rk +1}$ using the ordered $\cA$-basis
$\set{x_1,\dots,x_\rk,\vz ,x_{\rk+1},\dots,x_{2\rk}}$ for $X_\hyp\oplus \vz.\cA$.
Then we have $X =  \cA^{2\rk +1} \oplus C$ and $\xi = \omega_{2\rk+1}\perp(-\chi)$,
where $\omega_{2\rk+1}$ (resp.~$-\chi$) is the restriction of $\xi$
to $\cA^{2\rk +1}$ (resp.~$C$).
With this translation,
(a) and the first statement in (b) comprise \cite[Thm.~3.10]{ABG}.
The last (more explicit) statement in (b) follows from the proof of
the same theorem. (See Paragraph 2.5, Theorem 2.48, Paragraph 2.50,  and
the proof of Proposition 3.9 in \cite{ABG}.)
\end{proof}

\begin{remark}
\label{rem:rootmod}
Suppose that $\cB = \fbu(X,\xi)$ and $\hd$ are as in Theorem
\ref{thm:BCgraded}~(a), and let $\pi : \cB \to \cB/Z(\cB)$
denote the canonical projection.  Now $\hd\cap Z(\cB) = 0$
(since
$0\ne U(\vz,x_i) \in U(X_0,X_{\ep_i}) \subseteq
\cU_{\ep_i} \subseteq \cB$,   and hence any element of
$\hd\cap Z(\cB)$ lies in the kernel of $\ep_i$  for all $i=1,\dots,\rk$).
Thus, $\pi|_\hd$
is a linear isomorphism of $\hd$ onto $\pi(\hd)$.  Consequently, the inverse dual of this map
is a linear isomorphism  of $\hd^*$ onto $\pi(\hd)^*$.  Later we will use
this map to identify $\hd^*$ and $\pi(\hd)^*$.  In this way,
the induced $Q$-grading on $\pi(\cB)$ can be regarded as the root space
decomposition relative to the adjoint action of $\pi(\hd)$.
\end{remark}

The next two results provide us with  detailed information about the root spaces of
$\cB$, $\cS$, $\cF$,  and $\cU$ relative to the adjoint action of $\hd$.

\begin{proposition}
\label{prop:Qgrade}\
\begin{itemize}
\item[(a)]  If $0 \ne \mu\in Q$, then $\cB_\mu = \cS_\mu = \cF_\mu = \cU_\mu$.
\item[(b)] If $\nu\in Q$, then $\cF_\nu = \sum_{\lambda,\mu \in Q,\ \lambda+\mu = \nu}U(X_\lambda,X_\mu)$.
\end{itemize}
\end{proposition}

\begin{proof}
Part  (a) is clear since $\cU_\mu \subseteq \cB \subseteq \cU$
for $0\ne \mu\in Q$.  Part (b) follows from~\eqref{eq:gradF}.
\end{proof}

\begin{proposition}
\label{prop:Qgradeexp}  We have
\begin{equation}
\label{eq:FU0}
\cF_{0}= \sum_{i=1}^\rk U(x_i.\cA,x_\prmi) + U(X_\an,X_\an).
\end{equation}
Moreover,  for each  $\mu\in \Dl$,
a general element of $\cF_\mu\ (= \cU_\mu =\cS_\mu = \cB_\mu)$
can be expressed uniquely in the form indicated below.
\begin{tabbing}
\hspace{.1truein}
\=\hspace{.25truein}
\=\hspace{2.25 truein}
\=\kill
\> \hbox{\rm (a)}
\> $\mu=\ep_i+\ep_j$, \ $1\le i \ne j \le 2\rk$, $i\ne \bar \jmath$:
\> \quad $U(x_i.\al,x_j)$, \ $\al\in\cA$\\
\> \hbox{\rm (b)}
\> $\mu=2\ep_i$, \ $1\le i  \le 2\rk$:
\> \quad $U(x_i.b,x_i)$, \ $b\in\cA_-$\\
\> \hbox{\rm (c)}
\> $\mu=\ep_i$, \ $1\le i  \le 2\rk$:
\> \quad $U(v,x_i)$, \ $v\in X_\an$.
\end{tabbing}
\end{proposition}

\begin{proof}  Applying   Proposition \ref{prop:Qgrade}\,(b), \eqref{eq:Xweightspace},
\eqref{eq:Uident1} and \eqref{eq:Uident2}, we obtain \eqref{eq:FU0} and
the existence of the expressions in (a)--(c).
Also, the uniqueness is easily checked.  For example,
in (c), suppose that $U(v,x_i) = 0$, where $v\in X_\an$.  Then we have
$0 = U(v,x_i)x_\prmi = v$.
\end{proof}

The
following identities, together with Proposition \ref{prop:Qgradeexp},
allow us to recover $(\cA,-)$, $X$ and $\xi$ from the Lie algebra
$\cB$ and its root space decomposition.  Each of these identities
can be checked directly using  \eqref{eq:Uident1} and \eqref{eq:Uident3}.

\begin{proposition}
\label{prop:identities} Suppose that $\al,\beta\in\cA$, $v,w\in X_\an$ and
$1\le i,j,k \le \twor$.
Then,
\begin{alignat}{2}
\label{eq:Uskew}
U(x_i.\al,x_j) &= - U(x_j.\overline \al,x_i)
\\
\label{eq:Uident4}
[U(x_i.\alpha ,x_j) , U(x_\prmj.\beta,x_k)]  &=\phantom{-}
U(x_i.\alpha\beta,x_k)
&&\quad\text{if $k \ne \prmi, \prmj$},
\\
\label{eq:Uident5}
[U(x_i.\alpha,x_j) , U(v,x_\prmi)] &=-U(v.\al,x_j),
&&\quad\text{if $j\ne i$},
\\
\label{eq:Uident6}
[U(v,x_i), U(w,x_j)] &=  -U(x_i.\xi(v,w),x_j)
&&\quad\text{if $j \ne \prmi$}.
\end{alignat}
\end{proposition}

\begin{remark}
\label{rem:idB} Suppose that $\rk\ge 2$.
It follows
from Proposition 4.2.7, identity
\eqref{eq:Uident5} (with $\al = 1$) and identity \eqref{eq:Uident6} (with $v= \vz$,
$w = \vz.\al$ and $j=i$)
that the ideal of $\cB$ generated by
$\set{\cB_{\ep_i+\ep_j} \suchthat 1\le i,j\le 2\rk,\ i,j,\prmi,\prmj \text{ distinct}}$
contains $\cB_\mu$ for all $\mu\in \Dl$ and hence equals $\cB$.
\end{remark}

For use in the construction of extended affine Lie algebras (see  Section \ref{subsec:EALA}),
we now  determine the centroid of the $\BCr$-graded unitary
Lie algebra $\cB$.
This has been done previously in \cite[Thm.~5.8]{BN}, where $\cB$ is presented in a different way
(using the module decomposition of $\cB$ relative to the adjoint action of the grading subalgebra), and
we could transport that result to our setting.  However, for the convenience of the reader,
instead we  include a direct proof in the present setup.

\begin{proposition}
\label{prop:centroid}
For
$\fz\in Z(\cA,-)$,  define $\Lmult_\fz :  \End_\cA(X) \to \End_\cA(X)$ by
$\Lmult_\fz(T) = \fz T$   for $T\in \End_\cA(X)$ (see \pref{pgraph:Zaction}).
Then the  map
$\fz\to \Lmult_\fz|_\cB$ is an algebra monomorphism of $Z(\cA,-)$ into
$\Cent(\cB)$.
Moreover, when $r\ge 3$,  this map is an isomorphism.
\end{proposition}

\begin{proof}
One can verify directly that the map $\fz\to \Lmult_\fz$
is an algebra homomorphism of $Z(\cA,-)$ into the
centroid of the associative algebra $\End_\cA(X)$.  Further, by
\pref{pgraph:Zaction}, $\Lmult_\fz$ stabilizes $\cF$ for
$\fz\in Z(\cA,-)$.
Therefore,
$\Lmult_\fz|_\cF$ lies in the centroid of $\cF$, so it stabilizes the root spaces of $\cF$ relative
to $\hd$.  Thus, $\Lmult_\fz$ stabilizes $\cB$, and so   $\Lmult_\fz|_{\cB}$  lies in the centroid
of $\cB$. To see that the map $\fz\to \Lmult_\fz|_\cB$ is injective, suppose that  $\Lmult_\fz|_\cB = 0$
where $\fz\in Z(\cA,-)$.  Then $U(v_0.\fz,x_i) = \Lmult_\fz U(v_0,x_i) = 0$ for $1\le i \le 2\rk$,
so $v_0.\fz = 0$ by the uniqueness in Proposition \ref{prop:Qgradeexp}~(c). Hence,
$\fz = \xi(v_0,v_0.\fz) = 0$.

Now let $\rk\ge 3$ and $T\in \Cent(\cB)$.  We show
that $T = \Lmult_\fz$ for some $\fz\in Z(\cA,-)$.  Set
$\ttJ = \set{1,\dots,2\rk}$.    By Proposition \ref{prop:Qgradeexp},
if  $i,\prmi, j$ are distinct in $\ttJ$, then
any element of the root space of $\cB$
corresponding to $\ep_i + \ep_j$
can be written uniquely in the form $U(x_i.\al,x_j)$,
where $\al\in\cA$.  Since $T$ must stabilize root spaces, when  $i,\prmi, j$ are distinct in $\ttJ$,
there exists a unique $\tau_{ij}\in \End_\F(\cA)$ such that
\begin{equation*}
\label{eq:C1}
T(U(x_i.\al,x_j)) = U(x_i.\tau_{ij}(\al),x_j)
\end{equation*}
for $\al\in \cA$.

Applying $T$ to equation
\eqref{eq:Uident4}, we see that
\begin{equation}
\label{eq:C2}
\tau_{i k}(\al\beta) = \al\tau_{\prmj k}(\beta) = \tau_{ij}(\al)\beta
\end{equation}
for $i,\prmi,j,\prmj,k, \prmk $ distinct in $\ttJ$.  The special cases of $\al = 1$ and $\beta = 1$ then tell us
that
\begin{equation}
\label{eq:C3}
\tau_{i k} = \tau_{\prmj k} = \tau_{ij}
\end{equation}
when  $i,\prmi,j,\prmj,k, \prmk$ are distinct in $\ttJ$.
Now if $i,\prmi, j$ are distinct in $\ttJ$, then since $r\ge 3$, we can choose $k\in \ttJ$
with $i,\prmi,j,\prmj,k, \prmk $ distinct,  in which case
$\tau_{ij} =\tau_{ik} = \tau_{i\prmj}$,  and similarly
$\tau_{ij} = \tau_{\prmi j}$.   So any $\tau_{ij}$ with $i,\prmi, j$ distinct in $\ttJ$
equals one with $i,j$ distinct in $\set{1,\dots,r}$, and in that event,  it follows easily from \eqref{eq:C3}
that $\tau_{ij}$ is independent of the choice of $i,j$.  We let
$\tau$  denote this common map.   From  \eqref{eq:C2} we know that
$\tau \in \Cent(\cA)$, and therefore
\[\tau(\al) = \fz \al = \al \fz\]
for $\al\in\cA$, where $\fz = \tau(1) \in Z(\cA)$.
Now applying $T$ to equation   \eqref{eq:Uskew},  we see that
$U(x_i.\tau(\al),x_j) = -U(x_j.\tau(\overline \al),x_i) = U(x_i.\overline{\tau(\overline \al)},x_j)$ for
$\al\in\cA$, $i,\prmi,j$ distinct in $\ttJ$.  So $\tau$ commutes with $-$. Hence
$\fz\in  Z(\cA,-)$.

Finally, $T' :=T - \Lmult_\fz$ sends
the root spaces $\cB_{\ep_i+\ep_j}$, with $i,j,\prmi, \prmj \text{ distinct in \ttJ}$, to 0
so $T' = 0$ by Remark \ref{rem:idB}.
\end{proof}
\smallskip

\section[Unitary Lie algebras over associative tori]{\cm HERMITIAN  FORMS AND UNITARY LIE ALGEBRAS OVER ASSOCIATIVE TORI}
\label{sec:Herm}\

The  results we develop in this  chapter on
graded hermitian forms and unitary Lie algebras
over associative tori with involution will be used in the  next chapter  to
construct unitary Lie tori of type $\BCr$.

\subsection{Associative tori with involution}\
\label{subsec:associativetori}

Suppose that $\Llat$ is an additive abelian group.

\begin{definition}
\label{def:associativetorus}
An \emph{associative $\Llat$-torus} is an
$\Llat$-graded associative algebra $\cA$
such that $\cA^\sg$ is spanned by an invertible
element of $\cA$ for each $\sg\in \Llat$.  (In other language,
an associative  $\Llat$-torus is a \emph{twisted group algebra} of $\Llat$ \cite[\S 1.1.2]{P}.)
If ``$-$'' is a graded involution on such an
$\cA$, then $\cA$ is an \emph{associative $\Llat$-torus with involution}.
\end{definition}

\begin{remark}
If $\Llat$ is free of finite rank, the associative $\Llat$-tori
with involution have been classified by  Yoshii in \cite{Y2}.  We will return  to this setting
in Section \ref{subsec:associativentori}.
\end{remark}

The following fact is stated in \cite[Prop.~2.44\,(iii)]{BGK} when $\Llat$  is free of finite  rank,
and for arbitrary $\Llat$   in \cite[7.7.1]{N3}.  For the convenience of the reader, we
supply a proof.

\begin{lemma}
\label{lem:torusfact}
If $\cA$ is an associative $\Llat$-torus, then $\cA = Z(\cA)\oplus [\cA,\cA]$.
\end{lemma}

\begin{proof} To show that $Z(\cA)\cap [\cA,\cA] = 0$, suppose the contrary.
Then since $\cA$ is an associative $\Llat$-torus, we know that
$[\cA,\cA]^\sg = Z(\cA)^\sg$ is one-dimensional for some $\sg\in \Llat$.
So  $0\ne [\al,\beta] \in Z(\cA)^\sg$ for some nonzero homogeneous $\al,\beta\in\cA$.
Since  $\cA^{\sg}$ is one-dimensional, $[\al,\beta] = \vartheta \al \beta$ for some
$\vartheta\in \F^\times$.  So $0 = [\al,[\al,\beta]] = \vartheta[\al,\al \beta] = \vartheta \al[\al,\beta] = \vartheta^2\al \al \beta$,
and since $\al,\beta$ are invertible $\vartheta = 0$, a contradiction.

To show that $\cA = Z(\cA)+ [\cA,\cA]$,  it is enough to prove that  any homogeneous element $\al$
of $\cA \setminus Z(\cA)$  is in $[\cA,\cA]$.  Now $\al\beta \ne \beta \al$ for some nonzero homogeneous $\beta$.
So $\beta\al = \delta \al\beta$, where $\delta \in \F^\times$ and $\delta \ne 1$.
Then $[\al\beta^{-1},\beta] = (1-\delta)\al$, so   $\al\in [\cA,\cA]$.  \end{proof}

\begin{pgraph}
\label{pgraph:L+-}
If $\cA$ is an associative
$\Llat$-torus with involution, then since $\cA$ is finely graded with support $\Llat$, we have
$\Llat  = \Llat_+ \distu \Llat_-$,
where $\distu$ denotes disjoint union.  (See  \pref{pgraph:graded}~(c).)
\end{pgraph}

\begin{pgraph}
\label{pgraph:rank}   Now suppose that  $L$ is a subgroup of  an abelian group $\G$ and
$\cA$ is an associative
$\Llat$-torus.     Then
$\cA$ is a graded division algebra (in the sense that each nonzero homogeneous element
in invertible), so $\G$-graded $\cA$-modules
have properties analogous to modules over a division algebra.
In particular,  if $X$ is such a $\G$-graded $\cA$-module,   every homogeneous $\cA$-spanning set of $X$ contains a homogeneous $\cA$-basis of $X$;
every homogeneous $\cA$-independent set in $X$ is contained in a homogeneous $\cA$-basis of $X$;
$X$ is a free $\cA$-module with a homogeneous $\cA$-basis;
and any two homogeneous $\cA$-bases for $\cA$ have the same cardinality  \cite[p.~100]{RTW}.
We define the \emph{rank} of $X$ over $\cA$, denoted by $\rank_\cA(X)$, to be
the cardinality  of a homogeneous $\cA$-basis for~$X$.

If $\sg\in \G$ and $B^\sg$ is an $\F$-basis for $X^\sg$,
then $B^\sg$ is also an $\cA$-basis for $X^\sg.\cA$, so
\[\rank_\cA(X^\sg.\cA) = \dim_\F(X^\sg).\]

Let $S = \supp_\G(X)$.  Then $\Llat+S \subseteq S$, so
$S$ is the union of cosets of $\Llat$ in $\G$. Let $S/\Llat = \{ \sigma + L \mid \sigma \in S\}$
in  $\G/\Llat$.
Then $X = \bigoplus_{\sg\in\Theta} X^\sg.\cA$, where
$\Theta$ is a set of representatives of the cosets in  $S/\Llat$, and
\begin{equation}
\label{eq:rank}
\rank_\cA(X) = \textstyle\sum_{\sg\in \Theta} \dim_\F(X^\sg),
\end{equation}
where the right-hand side uses the arithmetic of cardinal numbers.
\end{pgraph}

\smallskip
\subsection{Hermitian forms over associative tori with involution}\
\label{subsec:herm}

Throughout Section \ref{subsec:herm},  \emph{we assume $\Llat$ is a subgroup
 of an abelian group $\G$ and
$(\cA,-)$ is an associative $\Llat$-torus with involution.}  Recall that we are regarding
$(\cA,-)$ as a $\G$-graded associative algebra with involution (as in \pref{pgraph:graded}\,(a)).

\begin{definition}
\label{def:isotropic}
Suppose that
$\xi : X \times X \to \cA$ is a $\G$-graded hermitian form over $(\cA,-)$.
We say that $\xi$ (or $X$)  is \emph{anisotropic}%
\footnote{If $\G$ is torsion-free,  then $\G$ can be ordered, so it is easy to check that our graded definition
of anisotropic is equivalent to the usual
ungraded definition of anisotropic, namely that $\xi(x,x) \ne 0$ for all
$0\ne x\in X$ \cite[p.~101]{RTW}.}
if  $\xi(x,x)\ne 0$ for all nonzero homogeneous $x\in X$.
An $\cA$-submodule $Y$ of $X$  is
\emph{totally isotropic} if
$\xi(y,y')=  0$ for all  $y,y' \in Y$.   When the form $\xi$ on $X$  is nondegenerate
and $X$ is the orthogonal direct sum
of two totally isotropic graded $\cA$-submodules, then we say that $\xi$ (or $X$)
is \emph{hyperbolic}.   When  $X$ has
a  homogeneous $\cA$-basis
$\set{x,y}$ such that $\xi(x,y) = 1$ and $\xi(x,x) = 0= \xi(y,y)$,
then $\xi$ (or $X$) is called a \emph{hyperbolic plane}.
\end{definition}

\begin{pgraph}
\label{rem:perp}
Suppose that
$\xi : X \times X \to \cA$ is a nondegenerate $\G$-graded hermitian form over $(\cA,-)$, and
$X$ has finite rank over $\cA$.
If $Y$ is a graded $\cA$-submodule
of $X$,  then $Y^\perp$ is a graded $\cA$-submodule of $X$ and
$\rank_\cA(Y^\perp) + \rank_\cA(Y) = \rank_\cA(X)$.  Hence if
$\xi|_{Y\times Y}$ is nondegenerate,  then
$X = Y \oplus Y^\perp$.
\end{pgraph}

Using the results of \cite[Sec.~1]{RTW}, we can prove the following analogues  of classical facts
about finite-dimensional  hermitian forms over division rings:

\begin{theorem}
\label{thm:herm}
Let  $\xi : X \times X \to \cA$ be a nondegenerate $\G$-graded
hermitian form  of finite graded $\F$-dimension over an associative $\Llat$-torus with involution~$(\cA,-)$.
\begin{itemize}
\item[(a)]  If $\xi$ is hyperbolic,  and hence $X = W_1\oplus W_2$, where $W_1$ and $W_2$
are totally isotropic graded $\cA$-submodules of $X$, then $\rank_\cA(W_1) = \rank_\cA(W_2)$,
and $X$ is the graded  orthogonal
sum of hyperbolic planes,  where the sum runs over an index set of cardinality $\rank_\cA(W_1)$.
\item[(b)] If $\xi$ is anisotropic, then $X$ has an
orthogonal homogeneous $\cA$-basis.  In fact, any orthogonal set of  nonzero homogeneous  elements
of $X$ is contained in an orthogonal homogeneous basis of $X$ over $\cA$.
\item[(c)] There exist $\G$-graded $\cA$-submodules $X_\hyp$ and $X_\an$ of $X$ such
that
\[X = X_\hyp \perp X_\an,\]
$\xi_\hyp = \xi|_{X_\hyp \times X_\hyp}$
is hyperbolic, and
$\xi_\an = \xi|_{X_\an \times X_\an}$
is anisotropic.  Moreover, if $X = X'_\hyp \perp X'_\an$ is another such orthogonal decomposition,
then there exists an $\cA$-linear  graded isometry from $X$ to $X$ that
maps $X_\hyp$ to $X'_\hyp$ and $X_\an$ to $X'_\an$.
\end{itemize}
\end{theorem}

\begin{proof}  We may assume that $X\ne 0$.  We define an  equivalence relation $\sim$ on $\G$ by saying
$\tau \sim \sg$ if and only if $\tau \equiv \pm \sg \pmod \Llat$. If
$\sg \in \G$, we set $X(\sg) = \sum_{\tau \sim \sg}X^\tau$.
Let $R$ denote a set of representatives
of the equivalence classes for $\sim$.
Then, any homogeneous element of $X$ is in $X(\sg)$ for some $\sg\in R$;
$X = \perp_{\sg\in R} X(\sg)$; and, if $Y$ is any graded $\cA$-submodule of $X$,
$Y = \perp_{\sg\in R} Y\cap X(\sg)$.
Thus to prove the theorem, we can assume that $X = X(\sg)$ for some
$\sg\in \G$.  Hence, we have $X = X^\sg.\cA + X^{-\sg}.\cA$.
Thus $\rank_\cA(X) = \dim_\F(X^\sg)$ if $2\sg\in \Llat$, whereas $\rank_\cA(X) = 2 \dim_\F(X^\sg)$
if $2\sg\notin \Llat$. So $X$ has finite rank over $\cA$.
Then (a) and (b) are easily checked (using  \ref{rem:perp}).
Moreover, (c) follows from  \cite[Prop.~1.4\,(iv)]{RTW}.
However, a few words are needed to justify the use of this proposition.  Indeed, it is assumed in
\cite{RTW} that the grading group $\G$ is torsion-free and divisible.  But
the proof of \cite[Prop.~1.4\,(iv)]{RTW}
does not use that assumption, provided we interpret
$\frac 12 \Llat$ as $\set{\sg\in \G \suchthat 2\sg \in \Llat}$.
\end{proof}

\begin{definition}
\label{def:Witt}    Suppose the assumptions of Theorem \ref{thm:herm} hold.
\begin{itemize}
\item[(a)]  A decomposition $X = X_\hyp \perp X_\an$ as in
Theorem \ref{thm:herm}\,(c) is said to be   a \emph{Witt decomposition} of $\xi$ (or $X$).
In that case we write
\[\G_\hyp = \supp_\G(X_\hyp) \andd \G_\an = \supp_\G(X_\an).\]
Note that by Theorem \ref{thm:herm}\,(c),
these subsets of $\G$ are independent of the choice of Witt decomposition.
\item[(b)] By assumption
$X_\hyp = W_1 \oplus W_2$, where $W_1$ and $W_2$ are totally isotropic
$\G$-graded $\cA$-submodules of $X$.
The \emph{Witt index of $\xi$} (or $X$) is defined as
$\rank_\cA(W_1) = \rank_\cA(W_2)$.   It follows from
Theorem \ref{thm:herm}~(c)  that this  index is
independent of the choice of Witt decomposition, and it is clear that it does not depend on the
decomposition  $X_\hyp = W_1 \oplus W_2$.
\end{itemize}
\end{definition}

In our setting, anisotropic forms have the following characterization:

\begin{proposition}
\label{prop:anisotropic}  Suppose the assumptions of Theorem \ref{thm:herm} hold.
If  $\xi$ is finely graded and $2\supp_\G(X) \subseteq \Llat$, then $\xi$ is anisotropic.
Moreover,  the converse holds if $\F$ is algebraically closed.
\end{proposition}

\begin{proof} ``$\Rightarrow$'' Suppose the contrary that $\xi(x,x) = 0$ for some $0\ne x\in X^\sg$, $\sg\in \G$.
As $2\sg\in \Llat$, it follows that $x.\cA^{-2\sg} \ne 0$.    But since $X$ is finely-graded, $X^{-\sg} = x.\cA^{-2\sg}$ and hence
$\xi(x, X^{-\sg}) = 0$, a contradiction.

``$\Leftarrow$'' By the argument in the proof of Theorem \ref{thm:herm}, we can assume that $X = X^\sg .\cA + x^{-\sg}.\cA$
for some $\sg\in \supp_\G(X)$.  But
$0\ne \xi(X^\sg,X^{\sg}) \subseteq \cA_+^{2\sg}$, so $2\sg \in \Llat$.  Hence,
\[X = X^\sg . \cA,\]
so $2\supp_\G(X) = 2(\sg+\Llat) \subseteq \Llat$. It remains to show that $\dim_\F(X^\sg) = 1$.
For this fix $0\ne a \in \cA_+^{2\sg}$.  We can define   an $\F$-bilinear form $f:X^\sg \times X^\sg \to \F$
such that $\xi(x,y) = f(x,y) a$ for $x,y\in X^\sg$.  Then, $f$ is symmetric and
anisotropic (in the usual ungraded sense).  Hence, since $\F$ is algebraically closed, $\dim(X^\sg) = 1$.
\end{proof}

\subsection{Unitary Lie algebras over associative tori with involution}
\label{subsec:unitary}   \

In Section \ref{subsec:unitary} (except in the final Remark \ref{rem:ungraded}), \emph{we assume  that
$\Llat$ is a subgroup
of an abelian group $\G$, that
$(\cA,-)$ is an associative $\Llat$-torus with involution,
and that $\xi : X\times X \to \cA$ is a nondegenerate $\G$-graded hermitian form of finite graded $\F$-dimension
over $(\cA,-)$; and we fix a  Witt decomposition $X = X_\hyp \perp X_\an$ of $\xi$}.

By Theorem \ref{thm:herm} we can choose a homogeneous $\cA$-basis $\set{x_i}_{i\in \ttI}$ for $X$ and a permutation
$i\mapsto \prmi$
\footnote{We have two maps, the involution $\alpha \to \overline \alpha$ of $\cA$
and the permutation $i \to \prmi$ of $\ttI$.  The context will make it clear which
of these is intended.}
 of period 2 of $\ttI$, such that
\begin{equation}
\label{eq:ximat}
\xi(x_i,x_j) = \delta_{i \prmj}\gm_i
\end{equation}
for $i,j\in \ttI$, where each $\gm_i$ is a homogeneous invertible element of $\cA$,
and such that
\begin{equation}
\label{eq:defhyp}
\textstyle X_\hyp = \bigoplus_{i\in \ttJ} x_i.\cA \andd X_\an = \bigoplus_{i\in \ttK} x_i.\cA.
\end{equation}
where
\[\ttJ = \set{i\in \ttI \suchthat \prmi \ne i} \andd \ttK = \set{i\in \ttI \suchthat \prmi = i}.\]
In this section, \emph{we fix a choice of such a basis $\set{x_i}_{i\in \ttI}$
and we fix a total ordering ``$\le$'' of the (possibly infinite) index set $\ttI$}.
Our goal is to use the basis $\set{x_i}_{i\in \ttI}$ to deduce more information  about
$\ffu(X,\xi)$ and $\fsu(X,\xi)$.

The elements $\gm_i$ and the bijection $i\mapsto
\prmi$ are uniquely determined by the basis $\set{x_i}_{i\in \ttI}$.
Also \eqref{eq:ximat} implies that
\begin{equation}
\label{eq:gmbar}
\gm_{\prmi} = \overline{\gm_i}
\end{equation}
for $i\in \ttI$.

As in Definition \ref{def:unitary}, we let
\[\cE = \fe(X,\xi),\quad \cU = \fu(X,\xi),\quad \cF = \ffu(X,\xi),  \andd \cS = \fsu(X,\xi).\]

\begin{pgraph} \ptitle{Matrix notation and the trace map}
\label{pgraph:matrix}
We adopt  matrix notation relative to the basis $\set{x_i}_{i\in\ttI}$,
which enables us to introduce a trace map on $\cE$.

(a)    For $i,j\in \ttI$ and $\al\in \cA$, let $e_{ij}(\al)$ denote the endomorphism in $\End_\cA(X)$
given by
\begin{equation}
\label{eq:eij}
e_{ij}(\al) x_t = \delta_{jt}x_i.\al
\end{equation}
for $t\in \ttI$.
Then  one verifies using \eqref{eq:ximat} that
\begin{equation}
\label{eq:Etoe}
E(x_i.\al,\,x_j.\beta) = e_{i\prmj}(\al\overline\beta \gm_j)
\quad\text{and so}
\quad e_{ij}(\al) = E(x_i,\,x_\prmj.(\gm_j^{-1}\overline \al))
\end{equation}
for $i,j\in I$, $\al,\beta\in \cA$.  Thus,  each $e_{ij}(\al)$ is in $\cE$ and
\begin{equation}
\label{eq:Esum}
\cE = \sum_{i,j\in\ttI} e_{ij}(\cA).
\end{equation}
In fact,  this sum is direct,  and each element of $\cE$
can be written uniquely in the form $\sum_{ij} e_{ij}(\al_{ij})$, where
$\al_{ij}\in \cA$ for all $i,j$. Moreover, we have
\begin{equation}
\label{eq:eprod}
e_{ij}(\al)e_{k\ell}(\beta) = \delta_{jk}e_{i\ell}(\al\beta)
\end{equation}
for $i,j,k,\ell \in \ttI$, $\al,\beta\in \cA$,
and so $\cE$ can be regarded as the \emph{associative algebra of   $(\ttI \times \ttI)$-matrices
over $\cA$} with finitely many nonzero entries.
\smallskip

(b)   The \emph{trace map} on $\cE$ is the additive map $\tr : \cE \to \cA$
so that
\begin{equation}
\label{eq:trace0}
\tr\big(e_{ij}(\al)\big) = \delta_{ij}\al
\end{equation}
 for $i,j\in \ttI$, $\al\in \cA$. One easily checks using \eqref{eq:eprod} that
\begin{equation}
\label{eq:trace1} \tr(T_1T_2) \equiv \tr(T_2T_1) \  \   \pmd
\end{equation}
for $T_1,T_2\in \cE$.
Also if $i,j\in \ttI$ and $\al,\beta\in \cA$, then  we have
$\tr\big(E(x_i.\al,\, x_j.\beta)\big) =  \tr\big(e_{i\prmj}(\al\overline\beta\gm_j)\big)
= \delta_{i \prmj}\al\overline\beta\gm_j
\equiv \delta_{i \prmj}\overline\beta\gm_j \al \  \pmd$
and so
\[\tr\big(E(x_i.\al,\,x_j.\beta)\big) \equiv \xi(x_j.\beta,x_i.\al)\  \pmd.\]
Hence for all $x,y\in\cA$,
\begin{equation}
\label{eq:trace2}
\tr\big(E(x,y)\big) \equiv  \xi(y,x) \  \  \pmd.
\end{equation}
  Consequently, modulo $[\cA,\cA]$, the trace function is independent of
the choice of basis $\set{x_i}_{i\in \ttI}$.  (This is also easy to
see directly using \eqref{eq:trace1}.)
\smallskip

(c)   It follows from \eqref{eq:trace2} that
\begin{equation}
\label{eq:trace3} \tr\big(U(x,y)\big) \equiv \xi(y,x) - \xi(x,y) \  \  \pmd. \end{equation}
\smallskip

(d) Recall from \pref{def:unitary}\,(d) that  we have an involution $*$ on $\cE$,
which satisfies $E(x,y) = E(y,x)^*$ for $x,y\in X$.  One can check
using \eqref{eq:Eprod}, \eqref{eq:gmbar} and \eqref{eq:Etoe} that
\begin{equation}
\label{eq:eij*}
e_{ij}(\al)^* = e_{\prmj\,\prmi}(\gm_j^{-1}\overline\al\gm_i)
\end{equation}
for $i,j\in \ttI$ and $\al\in \cA$.

(e)    For $i,j\in \ttI$ and $\al\in \cA$,  set
\begin{equation}
\label{eq:udef}
u_{ij}(\al) =  e_{ij}(\al) -  e_{ij}(\al)^* =  e_{ij}(\al) - e_{\prmj\,\prmi}(\gm_j^{-1}\overline\al\gm_i)
\end{equation}
so that
\begin{equation*}
\label{eq:usymm}
u_{ij}(\al) = -u_{\prmj\, \prmi}(\gm_j^{-1}\overline \al \gm_i).
\end{equation*}
Then \eqref{eq:Etoe} can be used to show that
\begin{equation}
\label{eq:Utou}
U(x_i.\al,\,x_j.\beta) = u_{i\prmj}(\al\overline\beta \gm_j)
\quad \text{and so}\quad u_{ij}(\al) = U(x_i,\,x_\prmj.(\gm_j^{-1}\overline \al ))
\end{equation}
for $i,j\in \ttI$, $\al,\beta\in \cA$.
Thus,
\[\cF = \sum_{i,j\in\ttI} u_{ij}(\cA).\]

\end{pgraph}

\begin{remark}  If $X$ has finite rank over $\cA$,
it follows from \eqref{eq:Esum} that
$\cE = \End_\cA(X)$ and hence, by virtue of \eqref{eq:Eadj}, that $\cF = \cU$.
\end{remark}

In the next two propositions, we use $U$-operators and then matrix notation to
to express each element of $\cF$ uniquely  and to calculate its  trace.

\begin{proposition}
\label{prop:FU}   We have
\begin{equation}
\label{eq:FUsum2}
\cF = \sum_{i,j\in \ttI}U(x_i.\cA,\,x_j.\cA)
\end{equation}
with
\begin{equation}
\label{eq:trace4}
\tr\big(U(x_i.\al,x_j.\beta)\big) \equiv \delta_{i\prmj}(\al\overline  \beta \overline{\gm_i} -  \gm_i \beta\overline \al )
\pmd
\end{equation}
for $\al,\beta\in \cA$.
Each element in $\cF$ can be expressed uniquely in the form
\begin{equation}
\label{eq:Uexp}
\sum_{i,j\in\ttI,\, i<j}U(x_i.\al_{ij},\,x_j)
+ \sum_{i\in \ttI} U(x_i.b_i,\,x_i),
\end{equation}
where $\al_{ij}\in \cA$ and $b_i\in \cA_-$.
\end{proposition}

\begin{proof}  Equation \eqref{eq:FUsum2} follows from the definition of $\cF$.
By \eqref{eq:trace3}, we have
\begin{eqnarray*} \tr\big(U(x_i.\al,x_j.\beta)\big) &\equiv&  \xi(x_j.\beta,x_i.\al) - \xi(x_i.\al,x_j.\beta)
\equiv   \overline \beta\delta_{j\prmi}\gm_j \al - \overline  \al \delta_{i\prmj} \gm_i\beta \\
&\equiv&  \delta_{i\prmj}( \al\overline \beta\overline {\gm_i} - \gm_i\beta\overline \al)\ \pmd.\end{eqnarray*}
Finally, by \eqref{eq:FUsum2}, \eqref{eq:Uident1},  and \eqref{eq:Uident2}, each element of $\cF$ can be expressed in the form given in
\eqref{eq:Uexp}.  Uniqueness of the expression in \eqref{eq:Uexp} comes from
\eqref{eq:udef} and \eqref{eq:Utou}.
\end{proof}

\begin{proposition}
\label{prop:FUmatrix} We have
\begin{equation}
\label{eq:FUmatrix} 
\cF = \sum_{i,j\in \ttI}u_{ij}(\cA)= \bigg\{\sum_{i,j\in \ttI} e_{ij}(\al_{ij})
\, \bigg | \,\al_{ij}\in \cA,
\ \al_{\prmj\,\prmi} = -\gm_j^{-1}\overline{\al_{ij}}\gm_i \text{ for } i,j\in \ttI\bigg\},
\end{equation}
with
\begin{equation}
\label{eq:trace5}
\tr\big(u_{ij}(\al)\big) \equiv \delta_{ij}(\al -\overline \al)\pmd
\end{equation}
for $\al,\beta\in \cA$ and $i,j \in \ttI$.  Each element of $\cF$
can be expressed uniquely in the form
\begin{equation}
\label{eq:uexp}
\sum_{i,j\in\ttI,\, i< \prmj} u_{ij}(\al_{ij}) + \sum_{i\in\ttI} u_{i\prmi}(b_i\gm_i)
\end{equation}
where $\al_{ij}\in\cA$ and $b_i\in \cA_-$.
\end{proposition}

\begin{proof}  The equalities  in \eqref{eq:FUmatrix} follow from
\eqref{eq:Utou} and  \eqref{eq:udef}.
 Also,
since
\[\tr\big(u_{ij}(\al)\big) = \delta_{ij} \al - \delta_{\prmj\, \prmi}\gm_j^{-1}\overline\al\gm_i,\]
we have \eqref{eq:trace5}.  Finally,
the last statement in the proposition follows from the last statement
in  Proposition \ref{prop:FU} and \eqref{eq:Utou}.
\end{proof}

In view of \eqref{eq:FUmatrix},
we can regard $\cF$ as a Lie algebra
of skew-hermitian $(\ttI \times \ttI)$-matrices over
$\cA$.\footnote{If $\prmi = i$ and $\gm_i = 1$ for all $i\in \ttI$, then
$\cF$ is the usual Lie algebra of skew-hermitian matrices.}
\medskip

We now prove the main result of this  chapter  which gives a simple and convenient description
of the Lie algebra $\cS$ and some of its properties. This  theorem was shown to hold in a special case in
\cite[Prop.~III.3.14 and Lem.~III.3.21]{AABGP}.
\medskip

\begin{theorem}
\label{thm:SU}
Suppose that
$\xi : X\times X \to \cA$ is a nondegenerate $\G$-graded hermitian form of finite graded $\F$-dimension
over an associative $\Llat$-torus
with involution $(\cA,-)$,  $X = X_\hyp \perp X_\an$ is a Witt decomposition of $\xi$,
and  $X_\hyp \ne 0$ and $X_\an \ne 0$ (that is $\xi$ is not anisotropic or hyperbolic).
Let $\cS =  \fsu(X,\xi)$ and $\cF  = \ffu(X,\xi)$.
Then
\begin{itemize}
    \item[(a)] $\cS = \set{\ T\in \cF \suchthat \tr(T) \equiv 0 \ \pmd\, }$.
    \item[(b)] $\cS$ is generated as an algebra by $U(X_\hyp,X_\an)$.
    \item[(c)] Suppose that the rank of $X$ over $\cA$ is infinite or finite and not divisible
    by $\characteristic(\F)$ (which holds in particular if $\characteristic(\F) = 0$). Then
    $\centre(\cS) = 0$.
    \item[(d)] Under the assumptions of (c), let $\varpi : \cA \to \F$ be  the $\Llat$-graded projection
onto $\cA^0 = \F1 = \F$ (where $\F$ has the trivial grading). Then
\[(T_1\mid T_2) := \varpi(\tr(T_1T_2))\]
defines a $\G$-graded nondegenerate associative symmetric bilinear form on $\cE$ whose
restriction to $\cF$ and to $\cS$ is nondegenerate.
\end{itemize}
\end{theorem}

\begin{proof}  Since $X_\hyp \ne 0$ and $X_\an \ne 0$, the sets
$\ttJ$ and $\ttK$ in \eqref{eq:defhyp}  are nonempty.  Further since $X = X_\hyp\oplus X_\an$ and $\xi(X_\hyp,X_\an)= 0$,
we know
that  $U(X_\hyp,X_\hyp)$ maps $X_\hyp$ to $X_\hyp$ and $X_\an$ to $0$; \ $U(X_\an,X_\an)$ maps $X_\hyp$ to $0$ and $X_\an$ to $X_\an$;
and $U(X_\hyp,X_\an)$ exchanges $X_\hyp$ and $X_\an$.  Therefore
\begin{equation}
\label{eq:FU}
\cF = U(X_\hyp,X_\hyp)\oplus U(X_\an,X_\an) \oplus U(X_\hyp,X_\an).
\end{equation}

(a) and (b):  Let $\cS_1$ denote the subalgebra of $\cF$ generated by $U(X_\hyp,X_\an)$, and  set
\[\cS_2 = \set{\ T\in \cF \suchthat \tr(T) \equiv 0 \pmd\ }.\]
To prove (a) and (b), we must verify that
$\cS_1 = \cS = \cS_2$.     We do this by showing  that
\[\cS_1\subseteq \cS \subseteq \cS_2 \subseteq \cS_1.\]

First, if $i\in \ttJ$ and $x''\in X_\an$, then \eqref{eq:Uident3} implies that
$[U(x_i,x_\prmi),U(x_i,x'')] = U(U(x_i,x_\prmi)x_i,x'') = U(x_i.\overline{\gm_i},\,x'')$.
Thus, $U(x_i.\overline{\gm_i},x'') \in \cF^{(1)} = \cS$. Since $\overline{\gm_i}$ is invertible,
 $U(X_\hyp,X_\an)\subseteq \cS$ by \eqref{eq:Uident1}  and so $\cS_1\subseteq \cS$.

The containment $\cS \subseteq \cS_2$ is a consequence of \eqref{eq:trace1}.

Finally,  we assume that $T\in \cS_2$ and show that $T\in \cS_1$.
Since $U(X_\hyp,X_\an) \subseteq \cS_1$, we can suppose by \eqref{eq:FU} that
\[ T\in \big(U(X_\hyp,X_\hyp)\oplus U(X_\an,X_\an)\big)\cap\cS_2. \]
Now if $x',y'\in X_\hyp$ and $x'',y''\in X_\an$, we have
\begin{align*}
[U(x',x''),U(y',y'')] &= U(U(x',x'')y',y'') +  U(y',U(x',x'')y'')\\
&=-U(x''.\xi(x',y'),\,y'') +  U(y',\,x'.\xi(x'',y''))
\end{align*}
and so
\begin{equation}
\label{eq:inSU1one}
U(x''.\xi(x',y'),\,y'') -  U(y',\,x'.\xi(x'',y''))\in \cS_1.
\end{equation}
If we take $x' = x_i$ and $y' = x_\prmi.\gm_i^{-1}$, where $i\in \ttJ$, then
$\xi(x',y') = 1$ and so $U(x'',y'') - U(y',\,x'.\xi(x'',y''))\in \cS_1$.
Subtracting elements of this form from $T$,  we can assume that
\begin{equation}
\label{eq:inSU1two}
T\in U(X_\hyp,X_\hyp)\cap\cS_2.
\end{equation}
For the rest of  the proof of (b),   we fix  $j\in \ttK$
and let
\[x'' = x_j, \quad \gm = \gm_j = \xi(x'',x'')\andd y'' = x''.\gm^{-1}. \]
Since $\xi(x'',y'') = 1$,   we have
$U(x''.\xi(x',y'),\,x''. \gm^{-1}) -  U(y',x')\in \cS_1$ by \eqref{eq:inSU1one}.
Thus
\begin{equation}
\label{eq:inSU1three}
U(x',y') + U(x'',\,x''. \gm^{-1}\xi(y',x'))\in \cS_1
\end{equation}
for $x',y'\in X_\hyp$.
Subtracting such elements from $T$,  we can suppose
that $T\in U(x'',\,x''.\cA)\cap \cS_2$ (of course we are no longer assuming
\eqref{eq:inSU1two}).    Now by \eqref{eq:Uident2}, we see that
$T\in U(x'',\,x''.\cA_-)\cap \cS_2$, so we can assume
\[T = U(x'',\,x''. b),\]
where $b\in \cA_-$.  Then by \eqref{eq:trace3},
$\tr(T) = -\xi(x'',\,x''. b) + \xi(x''. b,\,x'')
= -\xi(x'',x'')b -b \xi(x'',x'') = -\gm b - b\gm$ and so $\gm b + b\gm \in [\cA,\cA]$.
Therefore $\gm b\in [\cA,\cA]$ and so
\[T \in U(x'',\,x''. \gm^{-1}[\cA,\cA]).\]
Now for $x',y'\in X_\hyp$ and $\al\in \cA$,
\[U(x'.\al,\,y') + U(x'',\,x''. \gm^{-1}\xi(y',\,x'.\al))\in \cS_1\]
and
\[U(x',\,y'.\overline\al) + U(x'',\,x''. \gm^{-1}\xi(y'.\overline \al,x'))\in \cS_1\]
by \eqref{eq:inSU1three}.
Taking the difference,  we get $U(x'',\,x''. \gm^{-1}[\xi(y',x'),\al])\in \cS_1$.
Since $\xi(X_\hyp,X_\hyp) = \cA$, we have $U(x'',\,x''. \gm^{-1}[\cA,\cA])\subseteq \cS_1$;  hence $T\in \cS_1$
and $\cS_2 \subseteq \cS_1$.

(c)  Assume $T = \sum_{i,j \in \ttI}  e_{ij}(\al_{ij}) \in \centre(\cS)$
and set $U_k : = u_{kk}(1) = e_{kk}(1) - e_{\prmk \prmk}(1)$ for $k \in \ttJ$.
By \eqref{eq:trace0},
$\tr(U_k)= 0$,  and so $U_k  \in \cS$ by (a).
The relation $U_k  T = T U_k$  implies that
\begin{equation*}
\sum_j e_{kj}(\al_{kj}) - \sum_{j} e_{\prmk j}(\al_{\prmk j})  = \sum_i e_{ik}(\al_{ik}) - \sum_i
e_{i\prmk}(\al_{i\prmk}).
\end{equation*}
Hence, it must be that $\al_{ik} = 0 = \al_{i\prmk}$ and $\al_{kj} = 0 = \al_{\prmk j}$
for $i,j \neq k,\prmk$ and  that $\al_{k \prmk} = 0 = \al_{\prmk k}$ for all $k \in \ttJ$.
Thus, we may assume
\begin{equation}
\label{eq:zrel1} T = \sum_{\substack{ k \in \ttJ \\ k < \prmk}}
\Big( e_{kk}(\al_k) - e_{\prmk \prmk}(\gm_k^{-1} \overline
{\al_k} \gm_k) \Big) + \sum_{i,j \in \ttK} e_{ij}(\al_{ij}), \end{equation}
where $\al_{ji} = -\gm_j^{-1}\overline {\al_{ij}}\gm_i$ for $i,j \in \ttK$.
Suppose $\ell \in \ttJ$, $\ell < \prml$, and $m \in \ttK$.   Then $u_{\ell m}(\beta) =
e_{\ell m}(\beta) - e_{m \prml}(\gm_m^{-1} \overline \beta \gm_\ell)$ belongs
to $\cS$,  as it has trace 0 for all $\beta \in \cA$.      {F}rom
$u_{\ell m}(\beta) T = T u_{\ell m}(\beta)$ we deduce
\begin{equation}
\label{eq:zrel2}
e_{m \prml}(\gm_m^{-1}\overline \beta \overline {\al_\ell} \gm_\ell)
+ \sum_{j \in \ttK} e_{\ell j}(\beta \al_{m j}) = e_{\ell m}(\al_\ell \beta) - \sum_{i \in \ttK}
e_{i \prml}(\al_{i m} \gm_m^{-1} \overline \beta \gm_\ell).
\end{equation}
Therefore, $\beta \al_{m j} = 0$ for all $j \in \ttK, j \neq m$, and all $\beta$,
forcing $\al_{m j} = 0$ for all $j$ different from $m$.       Moreover, it follows from \eqref{eq:zrel2} that
$\beta \al_{m m} = \al_\ell \beta$ for all $\beta \in \cA$.   In particular
when $\beta = 1$, we obtain $ \al_\ell = \al_{m m}$ for $\ell \in \ttJ$, $m \in \ttK$.
Since both $\ttJ$ and $\ttK$ are assumed to be nonempty,  there is a unique element of $\cA$,
call it $\zeta$, so that $\zeta = \al_\ell = \al_{m m} $ for all $\ell \in \ttJ$, $m \in \ttK$,
and $\zeta \in \centre(\cA)$,  as it commutes with all $\beta$.    The $(m,\prml)$-term
of \eqref{eq:zrel2} gives $\gm_m^{-1} \overline \beta \, \overline \zeta \gm_\ell
= -\zeta \gm_{m}^{-1} \overline \beta \gm_\ell$ for all $\beta \in \cA$.
Taking $\beta = 1$ and using the fact that $\zeta$ is central, we see that
$\overline \zeta = - \zeta$, so that $\zeta \in \cA_-$.   Substituting these results
back into \eqref{eq:zrel1}, we have
\begin{equation}
\label{eq:zrel3}
T = \sum_{\substack{k \in \ttJ\\ k < \prmk}}
\Big(e_{k k}(\zeta) - e_{\prmk \prmk}(\gm_k^{-1}
\overline \zeta \gm_k)\Big) + \sum_{j \in \ttK}  e_{j j}(\zeta) = \sum_{s \in \ttI} e_{s s}(\zeta).
\end{equation}
Since elements of $\cS$ are finite sums, it must be that $\zeta = 0$ when $| \ttI | = \infty$.
Thus, the centre is trivial when $\ttI$ is infinite.  So we may assume that $| \ttI |$ is finite and not
divisible by the characteristic of $\F$.  Then  since
$\tr(T) \equiv 0  \ \ \pmd$ by (a), we see from \eqref{eq:zrel3}  that $0 \equiv | \ttI | \zeta  \ \ \pmd$.
Then, $\zeta \in [\cA,\cA]$
and $T = R_\zeta$, where  $R_\zeta(x) = x.\zeta$ for all $x \in X$.  We have proved that
$\centre(\cS) \subseteq \set{R_\zeta \suchthat \zeta \in \centre(\cA)\cap\cA_- \cap[\cA,\cA]}$.
Thus, by Lemma  \ref{lem:torusfact}, $Z(\cS) = 0$.

(d) The same proof given for \cite[Lem.~III.3.21]{AABGP} works here.    That  argument uses
part (c) above, Lemma \ref{lem:torusfact} and the equation $\varpi(\tr(T^*)) = \varpi(\tr(T))$ for $T\in\cE$.
This last equation can be checked for example
using \eqref{eq:trace2}.
\end{proof}

\begin{corollary}
\label{cor:SU}
 Under the hypotheses of Theorem \ref{thm:SU},  suppose that
$T\in \cF$ has an expression
\[T = \textstyle\sum_{i<j}U(x_i.\al_{ij},\,x_j)
+ \sum_{i} U(x_i.b_i,\,x_i)\]
as in \eqref{eq:Uexp}
(resp.~$T = \sum_{i< \prmj} u_{ij}(\al_{ij}) + \sum_{i} u_{i\prmi}(b_i\gm_i)$ as in \eqref{eq:uexp}).
Then $T \in \cS$ if and only if
\[\textstyle \sum_{i< \prmi}(\al_{i\prmi}\overline{\gm_i}- \gm_i\overline{\al_{i\prmi}})
+ \sum_{i=\prmi}  (b_i\gm_i + \gm_i b_i) \equiv 0 \ \  \pmd\]
(resp.,
$\sum_{i< \prmi}(\al_{ii}- \overline{\al_{ii}}) + \sum_{i=\prmi}  (b_i\gm_i +  \gm_i b_i) \equiv 0 \ \ \ \pmd$).
\end{corollary}

\begin{proof}  This follows from Theorem \ref{thm:SU}~(a) and a calculation of $\tr(T)$
using  \eqref{eq:trace4}   (resp.~\eqref{eq:trace5}).
\end{proof}

We  now use matrix notation to give an explicit
description   of the $\G$-gradings on $\cE$ and $\cF$ induced by the $\G$-grading on $X$.

\begin{proposition}
\label{prop:gradingmatrix}  Let $\rho_i = \deg_\G(x_i)$ for $i\in \ttI$.
Then, $\deg(\gm_i) = \rho_i + \rho_\prmi$ for $i\in \ttI$.
Moreover, the $\G$-gradings on $\cE$ and $\cF$
are given respectively by
\begin{gather}
\notag
\deg_\G(e_{ij}(\al)) = \rho_i - \rho_j + \deg_\G(\al),
\\
\deg_\G(u_{ij}(\al)) = \rho_i - \rho_j + \deg_\G(\al)
\end{gather}
for $\al$ homogeneous in $\cA$, $i,j\in \ttK$.
\end{proposition}
\begin{proof} The first relation follows from the fact that $\gm_i = \xi(x_i,x_\prmi)$,
The remaining equations come  from the definition of the gradings (see \pref{pgraph:gradedunitary}),
\eqref{eq:Etoe} and \eqref{eq:Utou}.
\end{proof}

\begin{remark}
\label{rem:gradingmatrix}
If $\xi$ not hyperbolic or anisotropic, Proposition \ref{prop:gradingmatrix} also tells us
the $\G$-grading on $\cS$,
since we know the precise form of elements of $\cS$ by Corollary~\ref{cor:SU}.
\end{remark}

\begin{remark}
\label{rem:ungraded} For the sake of clarity and
ease of reference,  in this section  we have chosen to work in the
context of graded hermitian forms over associative tori with involution.
However,  most of the results of  Section \ref{subsec:unitary}
hold in a more general ungraded setting.  Indeed,  suppose
that $\xi: X\times X \to \cA$ is a hermitian form over an associative algebra with involution $(\cA,-)$ on a free $\cA$-module
$X$ with $\cA$-basis $\set{x_i}_{i\in \ttI}$ satisfying \eqref{eq:ximat},
where $\gm_i$ is an invertible element of $\cA$ for
$i\in \ttI$,  and $i\mapsto \prmi$ is a permutation of period 2 of
$\ttI$; and suppose that we \emph{define} submodules $X_\hyp$ and $X_\an$ by
\eqref{eq:defhyp}.  Then all of the results of the section
hold (with the same proofs) except for  Theorem \ref{thm:SU}~(c),
Theorem \ref{thm:SU}~(d) and Proposition \ref{prop:gradingmatrix}.
The last two of these make no sense in the ungraded setting.
In the case of Theorem \ref{thm:SU}~(c), the same proof shows that the following
statement is true:  If $\left| \ttI \right|$ is infinite, then $\centre(\cS) = 0$; and if
$\left| \ttI \right|$ is finite and  not divisible by $\characteristic(\F)$,
then
\[\centre(\cS) = \set{R_\zeta \suchthat \zeta \in \centre(\cA)\cap\cA_- \cap[\cA,\cA]},\]
where $R_\zeta:X \to X$ is defined by $R_\zeta x = x.\zeta$ for $x\in X$.
Indeed, the proof of Theorem \ref{thm:SU}\,(c) establishes  all of this except
for the inclusion ``$\supseteq$'' which is easily checked using the fact that $T = R_\zeta$
can be expressed as   in \eqref{eq:zrel3}.
\end{remark}

\section[Unitary Lie tori]{\cm UNITARY LIE TORI AND THE MAIN THEOREMS}
\label{sec:main}

Throughout this  chapter, we assume that \emph{$\F$ has characteristic 0}.

We use the special unitary Lie algebra to
give a construction of a centreless Lie $\G$-torus of type $\BC_r$ and refer to the
resulting algebra as a
unitary Lie $\G$-torus.    Then we  show that if $r\ge 3$, any  centreless Lie $\G$-torus of
type $\BC_r$
is obtained in this way, and we determine when two such unitary Lie $\G$-tori are bi-isomorphic.

\subsection{Construction of unitary Lie $\G$-tori}
\label{subsec:BCconstruct}  \

\begin{assumptions}\label{asspts:uLietorus}\
Throughout Section \ref{subsec:BCconstruct}  we assume that:
\begin{itemize}
\item[(H1)]   $\rk$ is an integer $\ge 1$ and $\G$ is an abelian group.
\item[(H2)] $\Llat$ is a subgroup of $\G$, and
$(\cA,-)$ is an associative $\Llat$-torus with involution over $\F$  (which we regard as $\G$-graded by setting  $\cA^\sg = 0$ for $\sg \in \G \setminus L$).
\item[(H3)]   $\xi : X\times X \to \cA$ is a nondegenerate
$\G$-graded hermitian form of finite graded $\F$-dimension and Witt index $r$.
\item[(H4)]
For some Witt decomposition $X = X_\hyp \perp X_\an$ (and hence for
all Witt decompositions of $\xi$ by  Theorem \ref{thm:herm}\,(c)),
\begin{itemize}
\item[(a)]  $\G_\hyp = \Llat$,
\item[(b)] the restriction of $\xi$ to $X_\an^0$  represents $1$,
\item[(c)] $X_\an$ is finely graded.
\end{itemize}
\item[(H5)] $X$ has full support in $\G$.
\end{itemize}
\end{assumptions}

\begin{remark}  It follows from Proposition
\ref{prop:anisotropic} that (H4)(c) is redundant if $\F$ is algebraically closed.
\end{remark}

\begin{pgraph} \ptitle{The construction}
\label{con:uLietorus}
Let
\[\cF = \ffu(X,\xi) \andd \cS = \fsu(X,\xi),\]
so $\cS \subseteq \cF$.

To define gradings on $\cF$ and $\cS$, and to study the properties of $\cF$ and $\cS$,
we select an $\cA$-basis for $X$.  Indeed, by Theorem \ref{thm:herm}\,(a) and (b)
and by (H4)(a) and (b), we can choose a homogeneous $\cA$-basis  $\set{x_i}_{i\in\ttI}$ for
$X$  satisfying the following properties:
$X= X_\hyp \perp X_\an$ is a Witt decomposition of $\xi$ with
\[X_\hyp = \bigoplus_{i\in \ttJ} x_i.\cA \andd X_\an = \bigoplus_{i\in \ttK} x_i.\cA,\]
where $\ttI = \ttJ \distu \ttK$ with $\ttJ = \set{1,\dots,\twor}$;
$\deg_\G(x_i) = 0$ for $i\in \ttJ$ and $\deg_\G(x_{k_0}) = 0$ for some $k_0\in \ttK$;
and
\begin{equation*}
\label{eq:hypbasis}
\xi(x_i,x_j) = \gamma_i\delta_{i \prmj}
\end{equation*}
for $i,j\in \ttI$, where  $i\mapsto \bar\imath$ is the permutation of $\ttI$ defined by
\[\prmi = \twor + i - 1 \text{ for $i\in \ttJ$} \andd \prmi = i \text{ for $i\in \ttK$,}\]
and where $\gm_i$ is nonzero and homogeneous in $\cA$ for $i\in \ttI$ with
\begin{equation}
\label{eq:gammacond}
\gm_i = 1 \qquad\text{for } i\in \ttJ\cup\set{k_0}.
\end{equation}
We note that $\ttJ$, $\ttK$, $X_\hyp$, $X_\an$, $k_0$, the permutation $i\mapsto \prmi$ of $\ttI$,
and the  set $\set{\gm_i}_{i\in \ttI}$ are completely determined by this choice of $\cA$-basis.
We call  $\set{x_i}_{i\in \ttI}$ a \emph{compatible} $\cA$-basis for $\xi$.

We let
\begin{equation}
\label{eq:Slat}
\rho_i = \deg_\G(x_i) \text{ for $i\in \ttI$},
\end{equation}
so that
\begin{equation}
\label{eq:rhogam}
\rho_i = 0 \text{ for } i\in \ttJ\cup\set{k_0} \andd \gm_i = \xi(x_i,x_\prmi)\in \cA_+^{2\rho_i} \text{ for } i\in \ttI.
\end{equation}
We set $\vz= x_{k_0}$ and observe that then $\xi(v_0,v_0) = 1$.

We are now ready to define the gradings on $\cF$ and $\cS$.  First, we have the assumptions of
Section  \ref{subsec:BCunitary}, and as in that section,  we
let
\begin{equation}
\label{eq:hdef}
\hd = \bigoplus_{i=1}^\rk\F h_i, \quad \text{where }
h_i = U(x_i,x_\prmi) \text{ for } 1\le i \le \rk,
\end{equation}
and
\begin{equation} \label{eq:BCr2}
\Dl
= \set{\ep_i\mid i\in\ttJ}\cup\{\ep_i+\ep_j\mid i,j\in\ttJ,\ j\ne \prmi\},
\end{equation}
where $\set{\ep_1,\dots,\ep_\rk}$ in $\hd^*$ is the dual basis of $\set{h_1,\dots,h_\rk}$
and $\ep_\prmi = -\ep_i$ for $1\le i \le \rk$.
Thus, $\Dl$ is a root system of type $\BCr$ in $\hd^*$ with root lattice
\[Q = \bbZ \ep_1\oplus \dots \oplus \bbZ \ep_\rk.\]
The natural action of $\hd$ on $X$ gives the
weight space decomposition,
 \[X = \textstyle\bigoplus_{\mu\in Q}X_\mu\]
with
$\supp_Q(X) = \set{0}\cup \set{\ep_i \suchthat i\in\ttJ}$ and
\[
X_{\ep_i} =  x_i.\cA \ \hbox{\rm for}\  \  i \in \ttJ \  \andd X_0 =  X_\an.
\]
It is clear that this grading is compatible with the given $\G$-grading
on $X$, so $X$ is $(Q\times \G)$-graded with
\[X_\mu^\sigma = X_\mu \cap X^\sigma\]
for $\mu\in Q$, $\sigma\in \G$.
The $(Q\times \G)$-grading on $X$ induces a $(Q\times \G)$-grading
on $\cF$ with
\begin{equation}
\label{eq:FUgrading}
\cF_\nu^\tau = \textstyle \sum_{\rho+\sg=\tau}  \sum_{\lambda+\mu=\nu} U(X_\lambda^\rho,X_\mu^\sg),
\end{equation}
for  $\nu\in Q$, $\tau \in \G$ (see \pref{pgraph:gradedunitary}).
Also, since $\cS$ is generated by homogeneous elements of $\cF$,
it follows that $\cS$ is a $(Q\times \G)$-graded subalgebra of $\cF$.
The $\G$-gradings on $\cF$ and $\cS$  are induced by the $\G$-grading
on $X$, so they do not depend on our choice of compatible basis;
 and the $Q$-gradings on these algebras are the root space decompositions
relative to the adjoint action of $\hd$ (as in Section \ref{subsec:BCunitary}).
\end{pgraph}

\begin{pgraph}
We will see in
Theorem \ref{thm:biisomorphism} that the graded Lie algebras
$\cF$ and $\cS$ just described are independent up to bi-isomorphism
of the choice of a compatible $\cA$-basis $\set{x_i}_{i\in \ttI}$ for $\xi$.
We will also see in Theorem \ref{thm:structure}~(a) that $\cS$ is a centreless
Lie $\G$-torus of type $\BCr$.
 In anticipation of these results, we adopt the following terminology:

\begin{definition}
\label{def:uLie}
The $(Q\times \G)$-graded Lie algebra
$\cS = \fsu(X,\xi)$ constructed in \pref{con:uLietorus}
is called the \emph{unitary Lie $\G$-torus} (of type $\BCr$) constructed from $\xi$.
\end{definition}
\end{pgraph}

\subsection{Properties of unitary Lie tori}\
\label{subsec:propunitaryLietorus}

Our assumptions throughout this section are those of \pref{asspts:uLietorus}.  We fix
a compatible $\cA$-basis $\set{x_i}_{i\in\ttI}$ for $\xi$ and use the notation of
\pref{con:uLietorus}.  We record some properties of the $(Q\times \G)$-graded Lie algebras
$\cF$ and $\cS$ as well as of the $\G$-graded $\cA$-module $X$.

We have the assumptions of Section \ref{subsec:BCunitary}, so we can construct $\BCr$-graded
unitary Lie algebra $\cB = \fbu(X,\xi)$ as in that section.    Also, we have
the assumptions of Section \ref{subsec:unitary}, so we can apply the results of that section.

\begin{proposition} \
\label{prop:propSU}
\begin{itemize}
\item[(a)] $\cS = \set{\ T\in \cF \suchthat \tr(T) \equiv 0\ \pmd}$.
\item[(b)]
$\cS$ is generated as an algebra by  $\sum_{i\in\ttJ} \cS_{\ep_i}$.
Consequently $\cS = \cB$.
\item[(c)] $Z(\cS) = 0$.
\end{itemize}
\end{proposition}

\begin{proof} Parts (a) and (c) are proved in Theorem \ref{thm:SU}\,(a) and (c).
For (b), we have by Proposition \ref{prop:Qgradeexp}\,(c) that  $U(X_\hyp,X_\an) = \textstyle U(\sum_{i\in\ttJ} x_i.\cA ,X_{\an})
\textstyle =  \sum_{i\in\ttJ} \cS_{\ep_i}.$
Thus, Theorem \ref{thm:SU}\,(b) gives us the first statement of (b).
The second assertion in (b) then follows from the definition of $\cB = \fbu(X,\xi)$.
\end{proof}

Next we describe some  properties of the groups $\G$ and $\Llat$, the support sets  for $X$, $X_\hyp$
and $X_\an$, and the   modules $X$ and $X_\an$.

\begin{lemma}
\label{lem:suppcompare} \
\begin{itemize}
\item[(a)] $\Llat = \G_\hyp \subseteq \G_\an$, $\supp_\G(X) = \G_\an$ and $\langle \G_\an \rangle = \G$.

\item[(b)] $2\G_\an \subseteq \Llat_+$ and $2\G\subseteq \Llat \subseteq \G$.

\item[(c)]
$\G_\an = \distu_{i\in \ttK} (\rho_i + \Llat) = \distu_{i\in \ttK} (-\rho_i + \Llat) $.

\item[(d)] If $\G$ is free of rank $n$, then
$\Llat$ is free of rank $n$.

\item[(e)] $\rank_\cA(X_\an) = \order{\G_\an/\Llat}$
and $\rank_\cA(X) = 2\rk + \rank_\cA(X_\an)$.

\item[(f)] $\rank_\cA(X)$ is finite if and only if $|\G/\Llat |$ is finite.

\item[(g)] If $\G$ is finitely generated,  then $\rank_\cA(X)$ is finite.
\end{itemize}
\end{lemma}

\begin{proof} (a) The first statement follows from (H4)(a) and (b); the second statement follows
from the first; and then the third statement follows from (H5).

(b)  If $\sg\in \G_\an$, then there exists a nonzero $x\in X_\an^\sg$, so
$0\ne \xi(x,x)\in \cA^{2\sg}_+$ and hence $2\sg\in \Llat_+$.   So
$2\G_\an \subseteq \Llat_+$; hence $2\G\subseteq \Llat$ since $\langle \G_\an \rangle = \G$.

(c) This is clear since $\set{x_i}_{i\in \ttK}$ is an $\cA$-basis for $X_\an$.
(We know that $\rho_i + \Llat = -\rho_i + \Llat$ for $i\in\ttK$  since $2\G \subseteq \Llat$.)

(d)  follows from (b).

(e) follows from \eqref{eq:rank} since $X_\an$ is finely  graded.

(f) The group $\G/\Llat$ is generated by the set $\G_\an/\Llat$ by (a).  Hence, since $\G/\Llat$ has exponent 2
by (b), we see that $\G/\Llat$ is finite if and only if $\G_\an/\Llat$ is finite.  The claim now
follows from (e).

(g) Since $\G/\Llat$ has exponent 2, (g) follows from (f).
\end{proof}

The  next example, which  is a special case of the construction in
\cite[\S 7]{Y3}, shows that $X$ may have infinite rank over $\cA$.

\begin{example}
Suppose
$\G$ is an abelian group and $\Llat $ is a subgroup of $\G$ containing $2\G$.
Let $\cA = \F[\Llat] = \bigoplus_{\sg\in \Llat}\F t^\sg$ (the group algebra of $\Llat$) with the identity involution~$-$; let
$X = \cA^{2r}\oplus \F[\G]$ with the natural $\G$-grading and the natural right $\cA$-module structure; and let
$\xi : X \times X \to \cA$ be the unique symmetric $\cA$-bilinear form such that
$\xi(e_i,e_j) = \delta_{i\prmj}$, where $e_1,\dots,e_{2\rk}$ is the standard basis for $\cA^{2\rk}$,
$\xi(\cA^{2\rk},\F[\G]) = 0$, and $\xi(t^\sg,t^\tau) = \delta_{\sg+\Llat,\tau+\Llat}t^{\sg + \tau} $
for $\sg,\tau\in \G$.
Then Assumptions \ref{asspts:uLietorus}, with $X_\hyp = \cA^{2\rk}$ and $X_\an = \F[\G]$, are easily verified.
If $\G/\Llat$ is infinite, $X$ has infinite rank over  $\cA$.
\end{example}

If $\mu\in \Dl$ and $\sg\in \G$, then $\cF_\mu^\sg = \cS_\mu^\sg$ by Proposition \ref{prop:Qgrade}~(a).
We now describe these spaces explicitly.

\begin{proposition}
\label{prop:SUgradingexp}
Let $\mu\in \Dl$ and $\sg\in\G$.
In  each of the following cases,
a general element of $\cS_\mu^\sg$ has a unique expression in the indicated form.
Moreover, if  $\mu$ and $\sg$ are not covered by any of the cases below,
then  $\cS_\mu^\sg = 0$.
\begin{tabbing}
\hspace{.1truein}
\=\hspace{.25truein}
\=\hspace{2.5truein}
\=\kill
\> {\rm (a)}
\> $\mu=\ep_i+\ep_j$, \ $1\le i < j \le 2\rk$,\ $j \ne \bar \imath$,\  $\sg\in \Llat$:
\> \quad $U(x_i.\al,x_j)$, \ $\al\in\cA^\sg$\\
\> {\rm (b)}
\> $\mu=2\ep_i$, \ $1\le i  \le 2\rk$,\  $\sg\in \Llat_-$:
\> \quad $U(x_i.b,x_i)$, \ $b\in\cA_-^\sg$\\
\> {\rm (c)}
\> $\mu=\ep_i$, \ $1\le i  \le 2\rk$,\  $\sg\in \G_\an$:
\> \quad $U(v,x_i)$, \ $v\in X_\an^\sg$.
\end{tabbing}
\end{proposition}

\begin{proof} This follows from Proposition \ref{prop:Qgradeexp} and the definition
of the $\G$-grading on $\cF$ (see \pref{pgraph:gradedunitary}).
\end{proof}

\subsection{Statement of the structure theorem} \label{subsec:mainthm}  \qquad

We now state our first main theorem,  a structure theorem for centreless Lie tori
of type $\BCr$, $\rk \ge 3$.

\begin{theorem}[\textbf{Structure theorem}]
\label{thm:structure}  \ Let $\F$ be a field of characteristic 0.
\begin{itemize}
\item [{\rm (a)}]  Suppose that  $\rk\ge 1$, $\G$ is an abelian group  and $\Llat$ is a subgroup of $\G$.   Assume that
$ (\cA,-)$ is  an associative $\Llat$-torus with involution, that $\xi : X\times X \to \cA$   is a nondegenerate
$\G$-graded hermitian form of finite graded $\F$-dimension and Witt index $r$
over $(\cA,-)$, and that (H4) and (H5) in \pref{asspts:uLietorus}  hold.
Let $\cS = \fsu(X,\xi)$ be
the special unitary Lie algebra of $\xi$ with a $(Q\times\G)$-grading defined
using a compatible $\cA$-basis for $\xi$ as in
\pref{con:uLietorus}.
Then $\cS$
is a centreless Lie $\G$-torus of type $\Dl$, where
$\Dl$ is the root system of type $\BCr$ defined by~\eqref{eq:BCr2}.

\item[{\rm (b)}] Conversely, if $r\ge 3$, then
any centreless Lie $\G$-torus of type $\BCr$ is  bi-isomorphic
to a centreless Lie $\G$-torus $\cS$ constructed  as in {\rm (a)}.
\end{itemize}
\end{theorem}

\subsection{Proof of part~(a)  of the structure theorem}
\label{subsec:parta}  \

\begin{proof} Under the assumptions of (a) of Theorem \ref{thm:structure}, we
see that \ $\cS$ is centreless by Proposition \ref{prop:propSU}\,(c).
Thus, what is required to be shown is that $\cS$ satisfies (LT1)--(LT4)
of Definition \ref{def:Lietorus}.

For (LT1), we have seen in Section   \ref{subsec:BCunitary}
that $\cS_\mu = 0$ for $\mu\in Q\setminus(\Dl\cup\zero)$.     For (LT2)(i),
we know by Proposition \ref{prop:SUgradingexp}
that $\cS_\mu^0 \ne 0$ for $\mu\in \Dl$.   From Proposition \ref{prop:propSU}~(b) we see that (LT3) holds.
For (LT4), observe that $\supp_\G(X) = \G_\an \subseteq \supp_\G(\cS)$   by
Lemma \ref{lem:suppcompare}~(a) and Proposition \ref{prop:SUgradingexp}~(c).   So since
$X$ has full support in $\G$ by (H5),  we have (LT4).

It remains to confirm  that (LT2)(ii) holds.   For this,  recall that $\mu^\vee\in (\hd^*)^* = \hd$
denotes the coroot of $\mu$ for $\mu\in\Dl$.  Therefore,
\begin{equation}
\label{eq:coroot}
\begin{gathered}
\ep_i^\vee = 2 h_i,\quad (2\ep_i)^\vee =  h_i,\quad \text{for \ $1\le i \le \rk$},\\
(\ep_i-\ep_j)^\vee =  h_i-h_j,\quad (\ep_i+\ep_j)^\vee =  h_i+h_j\quad\text{for \ $1\le i < j \le \rk$},
\end{gathered}
\end{equation}
and $(-\mu)^\vee = -\mu^\vee$ for $\mu\in \Dl$.
Now $[\mu^\vee, x_\nu^\tau] = \langle\nu \mid \mu^\vee\rangle x_\nu^\tau$ for
$x_\nu^\tau\in \cS _\nu^\tau$, $\nu\in Q$, $\tau\in \G$.  Thus, since $Z(\cS) = 0$,
(LT2)(ii)  is equivalent to the following statement:
If $\mu\in \Dl$, $\sg\in \G$,  and $\cS_\mu^\sg \ne 0$, then
$\cS_\mu^\sg = \F e_\mu^\sg$ and $\cS_{-\mu}^{-\sg} = \F f_\mu^\sg$, where
\[[e_\mu^\sg,f_\mu^\sg] = \mu^\vee.\]
To verify that this is true, suppose that
$\mu\in \Dl$, $\sg\in \G$ and $\cS_\mu^\sg \ne 0$.
Since $X_{\an}$ is finely graded by (H4)(c),
it follows from Proposition \ref{prop:SUgradingexp}
that $\cS_\mu^\sg$ and $\cS_{-\mu}^{-\sg}$ are one-dimensional.  Hence, it suffices
to prove   that
\begin{equation}
\label{eq:LT2check}
\mu^\vee\in[\cS_\mu^\sg,\cS_{-\mu}^{-\sg}].
\end{equation}
For  this, we may assume that
$\mu$ is one of the following: $\ep_i$, $2\ep_i$,  where $1\le i \le \rk$,
or $\ep_i-\ep_j$, $\ep_i+\ep_j$, where $1\le i< j \le \rk$.
In each of these cases, one can check
\eqref{eq:LT2check} directly using  Proposition \ref{prop:SUgradingexp}
and \eqref{eq:coroot}.  We present only one argument, the others being similar.
Suppose that $\mu = \ep_i$, where $1\le i \le r$.
Then, by Proposition \ref{prop:SUgradingexp}\,(c),  $\sg\in \G_\an$.
So we may choose $0\ne v \in X_{\an}^\sg$
and take $0\ne \al = \xi(v,v)\in \cA_+^{2\sg}$.  Then
$U(v,x_i)\in \cS_\mu^\sg$, $U(v.\al^{-1},x_\prmi)\in \cS_{-\mu}^{-\sg}$,
and  since $\overline \al = \al$  we have
\begin{align*}
[U(v,x_i),&U(v.\al^{-1},x_\prmi)]
= U(U(v,x_i)(v.\al^{-1}),x_\prmi) +U(v.\al^{-1},U(v,x_i)x_\prmi)\\
&= U(-x_i,x_\prmi) +U(v.\al^{-1},v)
= -U(x_i,x_\prmi) + 0 = -h_i = -\frac 12 \ep_i^\vee.\qedhere
\end{align*}
\end{proof}

\subsection{Proof of part (b) of the structure theorem} \
\label{subsec:partb}

\begin{proof}

In order to prove part (b) of Theorem \ref{thm:structure}, we assume in this section  that
\[\cL = \textstyle\bigoplus_{\mu\in Q_\cL,\, \sg\in\G} \cL_\mu^\sg\]
is a
centreless Lie $\G$-torus of type $\Dl_\cL$, where $\Dl_\cL$ is a root system of type
$\BCr$, $r\ge 3$, and $Q_\cL$ is the root lattice of $\Dl_\cL$.  For convenience  we set
$\ttJ = \set{1,\dots,2\rk}$ and write $\prmi = 2\rk + 1 - i$ for $i\in \ttJ$.

\begin{pgraph} \ptitle{Preparation}
By Proposition \ref{prop:LTbasic}\,(f) and (g),
$\cL$ is a $\Dl_\cL$-graded Lie algebra with grading pair
$(\gd_{\cL}, \hd_{\cL})$,   where
\begin{equation}
\label{eq:gL}
\gd_\cL = \cL^0 \andd \hd_\cL = \cL_0^0.
\end{equation}
As in Section  \ref{subsec:basics}, we identify $\Dl_\cL$ with a root system in $\hd_\cL^*$.

By Theorem \ref{thm:BCgraded}\, (b), there exists an associative algebra with involution   $(\cA,-)$;  a
hermitian form $\xi : X\times X \to \cA$ over $(\cA,-)$ with
$X = X_\hyp \perp X_\an$;
an
$\cA$-basis $\{x_i\}_{i=1}^{\twor}$  for $X_\hyp$
with   $\xi(x_i,x_j) =
\delta_{i \prmj}$  for all $i,j \in \ttJ$;
an element $\vz\in X_\an$ with
$\xi(\vz,\vz)=1$;
and an isomorphism
\[\psi : \cL \to \cB/Z(\cB),\]
where $\cB = \fbu(X,\xi)$.    Moreover, $\psi$ may be chosen  so that
\[\psi(\gd_\cL) = \pi(\gd) \andd \psi(\hd_\cL) = \pi(\hd),\]
where $\pi : \cB \to \cB/Z(\cB)$ is the canonical projection,
 \begin{equation}
\label{eq:gh}
\gd = \textstyle
\Big(\bigoplus_{i,j \in \ttJ, \ i < j} \F\, U(x_i,x_j)\Big)
\bigoplus
\Big(\bigoplus_{i\in \ttJ } \F\, U(x_i,\vz)\Big),
\qquad
\hd =
\bigoplus_{i=1}^\rk  \F\, h_i,
\end{equation}
and $h_i =  U(x_i,x_\prmi)$, $1\le i \le r$, in $\cB$.
\end{pgraph}

\begin{pgraph} \ptitle{The root gradings}
\label{pgraph:Qgr}
Recall that $\cL$, as a Lie torus, comes equipped with
a grading by the root lattice $Q_\cL$ of the root system $\Dl_\cL$ in
$\hd_\cL^*$.     On the other hand, in Section  \ref{subsec:BCunitary}
we gave $\cB$ a grading by the root lattice $Q = Q(\Dl) = \bbZ\ep_1\oplus \dots  \oplus \bbZ\ep_r$ of $\Dl$,
where
\[\Dl = \set{\ep_i\mid i\in\ttJ} \cup\{\ep_i+\ep_j\mid i,j\in\ttJ,\ j\ne \prmi\},\]
$\ep_1,\dots,\ep_r$ is the dual basis in $\hd^*$
of $h_1,\dots,h_r$, and $\ep_\prmi = - \ep_i$ for $1\le i \le r$.  So as noted in (\pref{pgraph:modproperties}),
$\pi(\cB)$ is also graded by $Q$.
We now check that $\psi$
is an isograded-isomorphism of the $Q_\cL$-graded Lie algebra $\cL$ onto the $Q$-graded
Lie algebra $\pi(\cB)$.

We identify $\pi(\hd)^*$ with $\hd^*$ as in
Remark \ref{rem:rootmod}, and so $\Dl$ and $Q$ lie in
$\pi(\hd)^*$.
Next let $\invd{\psi}:\hd_\cL^* \to \pi(\hd)^* $ denote the
inverse dual of $\psi|_{\hd_\cL} : \hd_\cL \to
\pi(\hd)$. Now, the $Q_\cL$-grading of $\cL$ is the root space
decomposition relative to $\hd_\cL$  by \eqref{eq:Lroot}, and the $Q$-grading of $\pi(\cB)$ is the
root space decomposition relative to $\pi(\hd)$  by Remark
\ref{rem:rootmod}.
So it follows from the definition of $\invd{\psi}$
that
\[\psi(\cL_\mu) = \pi(\cB)_{\invd{\psi}(\mu)}\]
for $\mu\in \hd_\cL^*$.  Hence, we have
\begin{equation}
\label{eq:psisupp}
\invd{\psi} (\supp_{Q_\cL}(\cL)) = \supp_{Q}(\pi(\cB)).
\end{equation}
But $\supp_{Q_\cL}(\cL)$ generates $Q_\cL$ as a group
by (LT1) and (LT2)(i),  and $\supp_{Q} (\pi(\cB))$ generates $Q$ as a group by
Proposition \ref{prop:Qgradeexp}. So
\[\invd{\psi}(Q_\cL) = Q.\]
Thus,  $\psi$ is an isograded-isomorphism of the $Q_\cL$-graded Lie algebra $\cL$ onto the $Q$-graded
Lie algebra $\pi(\cB)$, as desired.

In addition, $\supp_{Q_\cL}(\cL)$ is either $\Dl_\cL\cup\zero$ or
$(\Dl_\cL)_\text{ind}\cup\zero$ by Remark \ref{rem:LTsupport}~(b);
whereas $\supp_Q(\pi(\cB))$ is either $\Dl\cup\zero$ or
$\Dlind\cup\zero$  by Proposition \ref{prop:Qgradeexp}. Thus, since
$\Dl_\cL$ and $\Dl$ are each root systems of type $\BCr$, it follows from
\eqref{eq:psisupp} that
\[\invd{\psi}(\Dl_\cL) = \Dl.\]
\end{pgraph}

\begin{pgraph} \ptitle{Transfer of the $\G$-grading}
\label{pgraph:transfer}
Now $\cL$, as a Lie $\G$-torus, has a given $\G$-grading which is
compatible with the $Q_\cL$-grading.  We use $\psi$
to transfer the $\G$-grading from $\cL$ to a $\G$-grading on $\pi(\cB)$ which
is compatible with the $Q$-grading on $\pi(\cB)$.

Henceforth, for convenience, we use $\psi$ and $\invd{\psi}$  to make the following identifications:
\[\cL = \pi(\cB),\quad \gd_\cL = \pi(\gd), \quad  \hd_\cL  = \pi(\hd),\quad Q_\cL = Q,  \andd \Dl_\cL = \Dl. \]
(These identifications are allowed since we are working up to  bi-isomorphism.)
Note in particular that $\cL$ and $\pi(\cB)$ are identified as $(Q\times \G)$-graded Lie algebras.   Also,
since  $\cL^0 = \gd_\cL = \pi(\gd)$, it follows from \eqref{eq:gL} and \eqref{eq:gh} that
\begin{equation}
\label{eq:deg0}
\pi(U(x_i,x_j))\in \cL^0 \andd \pi(U(x_k,\vz))\in \cL^0
\end{equation}
for $i,j\in \ttJ$, $i\ne j$, $k\in \ttJ$.
\end{pgraph}

\begin{pgraph} \ptitle{Support sets}
\label{pgraph:suppfact}
Recall that  in Section \ref{subsec:basics} we defined
$\G_\mu = \supp_\G(\cL_\mu)$
for $\mu\in \Dl$.  We set
\[S = \G_{\ep_k} \text{ for $k\in \ttJ$} \andd L = \G_{\ep_i+\ep_j} \text{ for $i,\prmi,j$  distinct in $\ttJ$.}  \]
By Proposition \ref{lem:suppfact}, these sets are well defined,  and we have
\begin{equation}
\label{eq:suppfact1}
0\in S, \quad -S = S, \quad 0\in L,\quad -L = L
\end{equation}
and
\begin{equation*}
\label{eq:suppfact2}
\langle S \rangle= \G.
\end{equation*}
Furthermore, taking $\nu = \ep_1+\ep_2$ and $\mu =  \ep_1$ in
Lemma \ref{lem:suppfact}~(c), we have $\Llat - 2\Slat \subseteq \Llat$,
and so
\begin{equation}
\label{eq:suppfact3}
2S \subseteq L.
\end{equation}
\end{pgraph}

\begin{pgraph} \ptitle{Canonical forms for elements of $\cL_\mu$}
\label{pgraph:canform}
Now if $i,j \in \ttJ$ with $j \ne i, \prmi$, then,
by Proposition \ref{prop:Qgradeexp},
elements of  $\cB_{\ep_i+\ep_j}$ can be uniquely expressed in the form
$U(x_i.\al,x_j)$ where $\al\in\cA$. So
we define
\[\U_{ij}(\al) = \pi(U(x_i.\al,x_j)) \in \cL_{\ep_i+\ep_j}\]
for $\al\in\cA$ and $i,j\in \ttJ$ with $j \ne i, \prmi$
(but we  do not define $\U_{i \prmi}(\al)$ and $\U_{ii}(\al)$).
Since $\ker(\pi) = Z(\cB) \subseteq \cB_0$, Proposition \ref{prop:Qgradeexp} tells us that
a general element of $\cL_{\ep_i+\ep_j}$
has a unique expression
\[\U_{ij}(\al), \quad \al\in\cA.\]
Similarly if we define
\[\U_{i}(v) = \pi(U(v,x_i)) \in \cL_{\ep_i}\]
for $i\in \ttJ$ and $v\in X_\an$, then  a general element
of
$\cL_{\ep_i}$
has a unique expression of the form
\[\U_i(v), \quad v \in X_\an.\]
(We could also introduce unique expressions for elements of $\cL_{2\ep_i}$, but that is not
 needed here.)

If $\al,\beta\in\cA$, $v,w\in X$ and $i,\prmi,j,\prmj,k,\prmk$ are distinct in $\ttJ$,
we obtain  the following identities by applying $\pi$ to the
identities in Proposition \ref{prop:identities}:
\begin{align}
\label{eq:Uskewmain}
\U_{ij}(\al) &= - \U_{ji}(\overline \al)\\
\label{eq:Uidentmain4}
[\U_{ij}(\alpha), \U_{\bar \jmath k}(\beta)]  &=
\U_{i k}(\alpha \beta)\\
\label{eq:Uidentmain5}
[\U_{ij}(\alpha), \U_{\bar \imath}(v)] &=- \U_j(v. \alpha),\\
\label{eq:Uidentmain6}
[\U_i(v), \U_j(w)] &=  -\U_{ij}(\xi(v,w)).
\end{align}
\end{pgraph}

\begin{pgraph} \ptitle{The $\G$-grading on $\cA$}
\label{pgraph:Agr}
Now $\cL_{\ep_1+\ep_2}$ is a finely $\G$-graded
vector space with support $\Llat$, and, by \pref{pgraph:canform}, the map $\al \mapsto \U_{ij}(\al)$ is a linear bijection
of $\cA$ onto $\cL_{\ep_1+\ep_2}$.   Hence, there is a unique $\G$-grading
on the vector space $\cA$  such that
\[\U_{12}(\cA^\sg) =  \cL_{\ep_1+\ep_2}^\sg\]
for $\sg\in \G$; and with respect to this grading, $\cA$ is finely graded with support $\Llat$.

Next we argue that
\begin{equation}
\label{eq:Ugr1}
\U_{ij}(\cA^{\sg}) =
\cL_{\ep_i + \ep_j}^{\sg}
\end{equation}
for $i, \prmi, j, \prmj$ distinct and $\sg\in \G$.
Indeed, since $\cL_{\ep_i+\ep_j} = \bigoplus_{\sg\in \G} \U_{ij}(\cA^\sg)$ and
$\cL_{\ep_i+\ep_j} = \bigoplus_{\sg\in \G} \cL_{\ep_i+\ep_j}^\sg$, it is sufficient to show that
\begin{equation}
\label{eq:Ugr2}
\U_{ij}(\cA^{\sg}) \subseteq
\cL_{\ep_i + \ep_j}^{\sg}
\end{equation}
for $i, \prmi, j, \prmj$ distinct and $\sg\in \G$.
But by  \eqref{eq:deg0}, we know that
$\U_{\prmj \ell}(1) \subseteq \cL_{\ep_\prmj + \ep_\ell}^0$ for $j\ne \ell, \bar \ell$.
Thus, using \eqref{eq:Uidentmain4}, we have
\[\U_{ik}(\cA^\sg) = [\U_{ij}(\cA^\sg), \U_{\prmj k}(1)]
\subseteq [\U_{ij}(\cA^\sg), \cL_{\ep_\prmj + \ep_k}^0]\]
for $i,\prmi, j, \prmj, k, \prmk$ distinct.
Hence, if \eqref{eq:Ugr2} holds for the pair $(i,j)$,  it also holds for $i,k$,
provided that $i,\prmi, j, \prmj, k, \prmk$ are distinct.
Similarly from \eqref{eq:Uidentmain4} we see that if
\eqref{eq:Ugr2} holds for the pair $(i,j)$, it also holds for the pair $(k,j)$,
provided that $i,\prmi, j, \prmj, k, \bar k$ are distinct.
Since \eqref{eq:Ugr2} holds for the pair $(1,2)$, it now follows easily that it
holds for all pairs $(i,j)$ with $i,\bar \imath, j, \bar \jmath$ distinct.

Next, if $\sg,\tau\in\G$, we have using
\eqref{eq:Uidentmain4} and \eqref{eq:Ugr1} that
\[\U_{12}(\cA^\sg\cA^\tau) = [\U_{13}(\cA^\sg),\U_{\bar 3 2}(\cA^\tau)] =
[\cL_{\ep_1+\ep_3}^\sg, \cL_{\ep_{\bar 3}+\ep_2}^\tau] \subseteq \cL_{\ep_1+\ep_2}^{\sg+\tau}
= \U_{12}(\cA^{\sg+\tau}),\]
and hence
$\cA^\sg \cA^\tau \subseteq \cA^{\sg+\tau}$.
Thus, $\cA$ is a finely $\G$-graded associative algebra.

Using \eqref{eq:Uskewmain} and \eqref{eq:Ugr1}, we conclude   that  $\U_{12}(\overline{\cA^\sg}) =
\U_{21}(\cA^\sg) = \cL_{\ep_2+\ep_1}^\sg = \U_{12}(\cA^\sg)$ for $\sg\in \G$.
Thus, the involution ``$-$'' on $\cA$ is graded.
\end{pgraph}

\begin{pgraph} \ptitle{$(\cA,-)$ is an  associative $\Llat$-torus with involution}
We have already established that the algebra $(\cA,-)$ is a finely $\G$-graded associative
algebra with involution and that $\supp_\G(\cA) = \Llat$.  To show that
$(\cA,-)$ is an associative $\Llat$-torus with involution,  it remains to show that
$\cA^\sg$ is spanned by an invertible element for $\sg\in \Llat$, since
if  that is true,  it follows that $\Llat$ closed under addition and also
that $\Llat$ is a subgroup of $\G$ using \eqref{eq:suppfact1}.
Let $\sg\in \Llat$.  Then, $-\sg\in \Llat$
and $0\in \Llat$ by \eqref{eq:suppfact1}, and thus by  Lemma \ref{lem:sl2}, we have
$[\cL_{\ep_1+\ep_2}^\sg, \cL_{\ep_{\bar 2}+\ep_3}^{-\sg}] = \cL_{\ep_1+\ep_3}^0$
Therefore,  from  \eqref{eq:Ugr1}, we see $[\U_{12}(\cA^\sg), \U_{\bar 2 3}(\cA^{-\sg})]
= \U_{13}(\cA^0)$.  So, by \eqref{eq:Uidentmain4},  $\cA^\sg \cA^{-\sg} = \cA^0 = \F 1$.
By symmetry we also have  $\cA^{-\sg} \cA^{\sg} = \F 1$ so that $\cA^\sg$ is spanned by an invertible element.
\end{pgraph}

\begin{pgraph} \ptitle{The $\G$-grading on $X$}
\label{pgraph:Xgr}
Now
$\cL_{\ep_1}$ is a finely $\G$-graded
vector space with support $S$, and, by \pref{pgraph:canform}, the map $v \mapsto \U_{1}(v)$ is a linear bijection
of $X_\an$ onto $\cL_{\ep_1}$.   Hence, there is a unique $\G$-grading
on the vector space $X_\an$ such that
\[\U_1(X_\an^\sg) = \cL_{\ep_1}^\sg\]
for $\sg\in \G$; and with respect to this grading,  $X_\an$ is finely graded with support $S$.
Also, arguing as in
\pref{pgraph:Agr}, we see using \eqref{eq:Uidentmain5} in place of \eqref{eq:Uidentmain4} that
\begin{equation}
\label{eq:Ugr3}
\U_i(X_\an^\sg) =  \cL_{\varepsilon_i}^\sg
\end{equation}
for all $i\in \ttJ$, $\sg\in\G$, and also that
$X_\an^\sg.\cA^\tau \subseteq X_\an^{\sg+\tau}$
for $\sg,\tau\in \G$.
Hence,  $X_\an$ is a finely $\G$-graded $\cA$-module.  Note also that
$\U_1(\vz)\in \cL^0_{\ep_1}$ by \eqref{eq:deg0},
so $\vz\in X_\an^0$.

To obtain a $\G$-grading
on the $\cA$-module $X_\hyp$, let $\deg_\G(x_i.\alpha) = \deg_\G(\alpha) = \sg$ for
$i \in \ttJ$, $\alpha \in \cA^\sg$ and $\sg\in \Llat$.

Finally, we give the $\cA$-module $X$
the direct sum grading with $X^\sg = X_\hyp^\sg \oplus X_\an^\sg$ for $\sg\in \G$. It is clear that  $X$ has finite
graded $\F$-dimension since this is true for both $X_\hyp$ and $X_\an$.
\end{pgraph}

\begin{pgraph} \ptitle{$\xi$ is $\G$-graded}
\label{pgraph:Vform}
Equation  \eqref{eq:Uidentmain6} tells  us that $\xi(X_\an^\sg,X_\an^\tau) \subseteq \cA^{\sg+\tau}$ for $\sg,\tau\in \G$,
So the hermitian form $\xi$ on $X_\an$ is $\G$-graded.  Also, if  $\alpha \in \cA^\sg$
and $\beta \in \cA^\tau$, then
$\xi(x_i.\alpha,x_j.\beta) = \delta_{i \prmj}\, \overline \alpha \beta \in \cA^{\sg + \tau}$,
which shows that the form $\xi$ is $\G$-graded on $X_\hyp$.  Thus, $\xi$ is $\G$-graded on $X$.
\end{pgraph}

\begin{pgraph} \ptitle{The Witt decomposition of $\xi$}
\label{pgraph:Witt}
First we argue  that $\xi$ is anisotropic on $X_\an$.
Let $v \in X_\an^\sg$ be nonzero for some $\sg \in \G$.
Then, $\sg\in S$ and $2\sg\in L$  by \eqref{eq:suppfact3}, so
by Lemma \ref{lem:sl2} we know that
$[\cL_{\varepsilon_1}^\sg, \cL_{\varepsilon_2}^\sg] = \cL_{\varepsilon_1+\varepsilon_2}^{2\sg}$.
But by \eqref{eq:Uidentmain6} we have
$$[\U_1(v), \U_2(v)] = \U_{12}(\xi(v,v)).$$
Thus, $\xi(v,v) \neq 0$, and the form on $X_\an$ is anisotropic.

It follows that $\xi$ is nondegenerate. But $X_\hyp$ is certainly  a hyperbolic space,
so  $X = X_\hyp \perp X_\an$ is a Witt
decomposition of $\xi$.  Therefore,
the Witt index of $\xi$ is $\rk$.
\end{pgraph}

\begin{pgraph} \ptitle{Assumptions \ref{asspts:uLietorus} hold}
We have argued that $(\cA,-)$ is an associative $\Llat$-torus with involution;
$\xi$ is a nondegenerate $\G$-graded hermitian form of finite graded $\F$-dimension
with Witt decomposition $X = X_\hyp \perp X_\an$ and Witt index $r$; and that
there exists a nonzero element  $\vz\in X_\an^0 $ such that
$\xi(\vz,\vz) = 1$.
In order to see that Assumptions \ref{asspts:uLietorus} hold,
all that remains is to show that (H4)(a), (H4)(c) and (H5) hold.
But  the first two of these are clear.  Also   since $\langle \supp_\G(X_\an) \rangle = \langle S \rangle  = \G$,
it follows that  $X_\an$ has full support in $\G$, and thus (H5)  holds.
\end{pgraph}

\begin{pgraph} \ptitle{Conclusion of the proof of Theorem \ref{thm:structure}\,(b)}
In \pref{pgraph:Qgr}, we  identified $\cL$  and  $\pi(\cB)$ as $(Q\times \G)$-graded Lie algebras.
We have just demonstrated that the assumptions in  \textbf {\ref{asspts:uLietorus}} hold.   Hence,
by Proposition \ref{prop:propSU}\,(b) and (c), we have $\cB = \cS$ and $Z(\cB) = 0$.
So with the obvious identification $\pi(\cB) = \cB$, we have
\[\cL = \pi(\cB) = \cB = \cS,\]
as $(Q\times \G)$-graded Lie algebras.

Finally, we know that the $Q$-grading on $\cS$ is the root space decomposition of $\cS$
relative to the adjoint action of $\hd$ (see Section \ref{subsec:BCconstruct}).
So all that remains to be  shown is that the $\G$-grading on $\cS$ (which
was obtained by transferring the given grading on $\cL$ to $\cS$ via the identification  in \pref{pgraph:transfer})
coincides with the $\G$-grading on $\cS$ induced
by the $\G$-grading on $X$.
But both of these $\G$-gradings are compatible with the $Q$-grading; and by Proposition \ref{prop:propSU}\,(b),
the Lie algebra $\cS$ is generated by $\sum_{i=1}^{2\rk} \cS_{\ep_i}$.  Hence, it suffices
to verify that the two $\G$-gradings agree on $\cS_{\ep_i}$ for $i\in \ttJ$.  But this is a consequence
of \eqref{eq:Ugr3} and Proposition \ref{prop:SUgradingexp}\,(c).
\qedhere
\end{pgraph}  \end{proof}

\subsection{The bi-isomorphism theorem}\
\label{subsec:bi}

Our second  main theorem gives necessary and sufficient conditions for two unitary Lie tori
of type $\BCr$, $\rk \ge 3$, to be bi-isomorphic.  Since $\rk$ is a bi-isomorphism invariant,
we can for convenience fix $\rk$ in the discussion.

\begin{theorem}[\textbf{Bi-isomorphism theorem}] \
\label{thm:biisomorphism}
Assume  $\F$ has characteristic 0.  Suppose that
$r$, $\G$, $\Llat$, $(\cA,-)$
and $\xi : X\times X \to \cA$
(resp.~$\rk' = \rk$, $\G'$,  $\Llat'$, $(\cA',-)$ and $\xi' : X'\times X' \to \cA'$) satisfy (H1)--(H5)
in \pref{asspts:uLietorus} and that
$\cF$ and $\cS$ (resp.~$\cF'$ and $\cS'$) are the graded Lie algebras constructed from $\xi$
(resp.~$\xi'$) as in \pref{con:uLietorus} using a compatible $\cA$-basis for $\xi$ (respectively
$\xi'$).
\begin{itemize}
\item[{\rm (a)}]   If there exists
a triple $(\theta,\theta_1,\theta_2)$ of maps, where
$\theta: X \to X'$  is an  $\F$-linear isomorphism,
$\theta_1 : (\cA,-) \to (\cA',-)$ is an isomorphism of algebras with involution,
and $\theta_2 : \G \to \G'$ is a group isomorphism,  which satisfy  \begin{gather}
\label{eq:bi1}
\theta(x.\al) = \theta(x).\theta_1(\al)
\\
\label{eq:bi2}
\theta_1(\xi(x,y)) = \xi'(\theta(x),\theta(y))\\
\label{eq:bi3}
\theta_1(\cA^\sg) = {\cA'}^{\theta_2(\sg)} \andd \theta(X^\sg) = {X'}^{\theta_2(\sg)}
\end{gather}
for $x,y\in X$, $\al\in \cA$ and $\sg\in \G$,
then  $\theta_2(\Llat) = \Llat'$, $\cF$ is bi-isomorphic to $\cF'$,
and $\cS$ is bi-isomorphic to $\cS'$.

\item[{\rm (b)}]  Conversely,
if $\rk \ge 3$ and $\cS$ and $\cS'$ are bi-isomorphic, then
there exists a triple $(\theta,\theta_1,\theta_2)$ of maps   with the indicated properties.
\end{itemize}
\end{theorem}

\begin{proof}  We assume that $\cS$ has been constructed
using a compatible $\cA$-basis $\set{x_i}_{i\in \ttI}$ for $\xi$ and we adopt
the notation of \pref{con:uLietorus}.  We do the same for $\xi'$ with primed notation.

(a) Suppose we have a triple of maps
$(\theta,\theta_1,\theta_2)$  as described in (a).
By \eqref{eq:bi3}, we have $\theta_2(\supp_\G(\cA)) = \supp_{\G'}(\cA')$ so $\theta_2(\Llat) = \Llat'$.  Next
by \eqref{eq:bi1}--\eqref{eq:bi3},
\[X' = \theta(X_\hyp) \oplus \theta(X_\an)\]
is a $\G'$-graded orthogonal
decomposition of $\cA'$-modules, $\theta(X_\an)$ is anisotropic, and $\set{\theta(x_i)}_{i\in \ttJ}$
is an $\cA'$-basis for $\theta(X_\hyp)$ relative to which the matrix of $\xi'|_{X_\hyp}$ is the $(\twor\times\twor)$-matrix
$(\delta_{i\prmj})$.  Then, by Theorem \ref{thm:herm}\,(c),
we can compose $\theta$ with a $\G'$-graded $\cA'$-linear isometry and assume that
\[\theta(X_\hyp) = X'_\hyp, \quad \theta(X_\an) = X'_\an \andd \theta(x_i) = x_i' \text{ for } i \in \ttJ.\]

Now define $\psi : \End_\F(X) \to \End_\F(X')$ by $\psi(T) = \theta T \theta^{-1}$.  By \eqref{eq:bi1},
we have $\psi(E(x,y)) = E(\theta(x),\theta(y))$;  hence
\[\psi(U(x,y)) = U(\theta(x),\theta(y))\]
for $x,y\in X$.  So $\psi(\cF) = \cF'$. Also,
\[\psi(h_i) = \psi(U(x_i,x_\prmi)) = U(\theta(x_i),\theta(x_\prmi))= U(x_i',x_\prmi') = h_i' \]
for $1\le i \le \rk$.  So $\psi(\hd) = \hd'$; and the inverse dual $\hat\psi : \hd^* \to {\hd'}^*$
of $\psi|_\hd$ maps $\ep_i$ to $\ep'_{i}$ for $1\le i \le r$.   Thus, $\hat\psi(Q) = Q'$ and
$\theta(X_\mu) = X_{\hat\psi(\mu)}'$  for $\mu\in Q$ so that
$\theta(X^\sg_\mu) = {X'}^{\theta_2(\sg)}_{\hat\psi(\mu)}$
for $\sg\in \G$, $\mu\in Q$.  Therefore
$\psi(\cF^\sg_\mu) = {\cF'}^{\theta_2(\sg)}_{\hat\psi(\mu)}$
for $\sg\in \G$, $\mu\in Q$.  Thus, $\psi$ is a bi-isomorphism
of $\cF$ onto $\cF'$, and its restriction maps $\cS$ onto $\cS'$.

(b) For the converse, suppose that $\rk\ge 3$ and $\psi$ is
a bi-isomorphism from $\cS$ to $\cS'$.  Then,  by Remark \ref{rem:rootiso},
$\psi_\rgr : Q\to Q'$  is an isomorphism of the root system
$\Dl$ onto the root system $\Dl'$, so $\psi_\rgr(\set{\ep_1,\dots,\ep_\twor}) = \set{\ep'_1,\dots,\ep'_\twor}$.
Therefore,
there exists an element $\omega'$ of the  Weyl group of $\Dl'$ such that
$\omega'(\psi_\rgr(\ep_i)) = \ep'_i$ for $1\le i \le 2\rk$. Hence,
by Remark \ref{rem:Weyl}, we can assume that $\psi_\rgr(\ep_i) = \ep'_i$ for $1\le i \le \twor$.

By Proposition \ref{prop:SUgradingexp}\,(a), there  exist unique linear bijections $\eta_{ij} : \cA \to \cA'$ for $i,\prmi, j,\prmj$
distinct in $\ttJ$ such that
\begin{equation}
\label{eq:ph1}
\psi(U(x_i\al,x_j)) = U(x_i'.\eta_{ij}(\al),x_j') \andd \eta_{ij}(\cA^\sg) = \cA'^{\theta_2(\sg)}
\end{equation}
for $\al\in \cA$ and $\sg\in \G$, where
\[\theta_2 = \psi_\egr.\]
Similarly, there exist unique linear bijections $\eta_i : X_\an \to X_\an'$ for $i\in \ttJ$ such that
\begin{equation}
\label{eq:ph2}
\psi(U(v,x_i)) = U(\eta_{i}(v),x'_i) \andd  \eta_i(X_\an^\sg) = {X'_\an}^{\theta_2(\sg)}
\end{equation}
for $v\in X_\an$ and $\sg\in \G$. Note that it follows from the second equation   in \eqref{eq:ph1}
and in  \eqref{eq:ph2} that
\[\eta_{ij}(1) = b_{ij}1 \andd \eta_{i}(\vz) = c_i v'_0,\]
where $b_{ij}\in \F^\times$ and $c_i\in \F^\times$.

If $\al,\beta\in\cA$, $v,w\in X_\an$ and
$i,j,k\in \ttJ$, then  applying $\psi$ to the identities in
Proposition \ref{prop:identities} yields:
\begin{alignat}{2}
\notag
\overline{\eta_{ij}(\al)} &= \eta_{ji}(\overline \al)
&&\quad\text{if $i,\prmi,j,\prmj$ are distinct},
\\
\notag
\eta_{ij}(\al) \eta_{\prmj k}(\al) &=
\eta_{ik}(\al\beta)
&&\quad\text{if $i,\prmi,j,\prmj,k, \prmk $ are distinct},
\\
\notag
\eta_\prmi(v) . \eta_{ij}(\al) &= \eta_j(v.\al),
&&\quad\text{if $i,\prmi,j, \prmj$ are distinct,}\\
\notag
\xi'(\eta_i(v) ,\eta_j(w))  &=  \eta_{ij}(\xi(v,w))
&&\quad\text{if $i,\prmi,j,\prmj$ are distinct.}
\end{alignat}

It is easy to verify now that
$\theta_1 := b_{ij}^{-1}\eta_{ij}$ is independent of $i,j$;
$\theta := c_{i}^{-1}\eta_i$ is independent of $i$;   and
the triple $(\theta,\theta_1,\theta_2)$ satisfies all of the required conditions, except that
so far $\theta$ is only defined on $X_\an$.
(We omit the proof of these assertions since the argument  is similar to the one used
in the proof of  Proposition \ref{prop:centroid}.  See also the reasoning in \pref{pgraph:Agr}.)
Finally,  we extend $\theta$ to $X$, by defining $\theta: X_\hyp \to X'_\hyp$
by  $\theta(x_i.\al) = x_i.\theta_1(\al)$ for $\al\in \cA$, $i\in \ttJ$.
The conditions that $\theta$ must satisfy   are then easily checked for this extension.
\end{proof}

\begin{pgraph}
\label{pgraph:invariant}
Suppose $\cL$ is a centreless Lie $\G$-torus of type $\BCr$,
where $\rk \ge 3$.  By Theorem \ref{thm:structure} we know that $\cL$ is bi-isomorphic to
a unitary Lie $\G$-torus $\cS$ constructed from some
$\xi : X\times X \to \cA$ over $(\cA,-)$.  By
Theorem \ref{thm:biisomorphism},  the associative $\Llat$-torus with involution $(\cA,-)$ (up to isograded-isomorphism)
and the $\cA$-rank of $X_\an$ are bi-isomorphism invariants of $\cL$.
We refer to  $(\cA,-)$ as the \emph{coordinate torus} and  $\rank_\cA(X_\an)$ as
the \emph{anisotropic rank} of $\cL$.
\end{pgraph}

With some additional assumptions, we can give a simpler version of the  bi-isomorphism theorem.  For this purpose, we will
use the following lemma:

\begin{lemma}
\label{lem:group} Suppose that $\G$ is an abelian group without $2$-torsion,
$\Llat$ is a subgroup of $\G$ containing $2\G$, and $\Slat$ is a subset of
$\G$ such that $\G = \langle \Llat \cup \Slat \rangle$.  Suppose further
$\G'$, $\Llat'$ and $\Slat'$ satisfy the same assumptions.  Then any isomorphism
of $\Llat$ onto $\Llat'$ mapping $2\Slat$ onto $2\Slat'$ extends uniquely
to an isomorphism of
$\G$ onto $\G'$ which  maps $\Slat$ onto $\Slat'$.
\end{lemma}

\begin{proof}  The map $\sg \mapsto 2\sg$ is an isomorphism of $\G$ onto $2\G$,
and we write its inverse map from $2\G$ to $\G$ by $\sg \mapsto \frac 12 \sg$.
We use the same notation for $\G'$.

For uniqueness, suppose that $\eta_1$ and $\eta_2$ are homomorphisms from $\G$ onto
$\G'$ that agree on $\Llat$.  Then, if $\sg\in \G$, we have $2\eta_1(\sg) = \eta_1(2\sg)
= \eta_2(2\sg) = 2\eta_2(\sg)$, so $\eta_1(\sg) = \eta_2(\sg)$.

\newcommand\heta{\hat\eta}

For existence, suppose that $\eta : \Llat \to \Llat'$ is an isomorphism such
that $\eta(2\Slat) = 2\Slat'$.  If suffices to show that there exists a homomorphism
$\heta: \G \to \G'$ such that $\heta(\Slat) \subseteq \Slat'$ (since then
the same argument will apply to $\eta^{-1}$). To see this,  observe
first that
\begin{equation}
\label{eq:groupfact}
\eta(2\G) \subseteq 2\G'.
\end{equation}
Indeed, if $\sg\in \Llat$  we have $\eta(2\sg)\in \eta(2\Llat) = 2\eta(\Llat)
= 2\Llat' \subseteq 2\G'$;  whereas if $\sg\in\Slat$ we have
$\eta(2\sg) \in \eta(2\Slat) = 2\Slat' \subseteq 2\G'$, proving \eqref{eq:groupfact}.
Thus, we can define $\heta:\G \to \G'$ by
\[\heta(\sg) = \textstyle \frac 12 \eta(2\sg)\]
for $\sg\in \G$.  Then, if $\sg\in \Llat$, we have
$\heta(\sg) = \frac 12 \eta(2\sg) = \frac 12 2\eta(\sg)= \eta(\sg)$;
and if $\sg\in \Slat$, we have $\heta(\sg) = \frac 12 \eta(2\sg)
\in \frac 12 \eta(2\Slat) = \frac 12 2 \Slat' = \Slat'$.
\end{proof}

We can now prove the following corollary of Theorem \ref{thm:biisomorphism}. Note
that under the assumptions of that theorem we have
$2\supp_\G(X) \subseteq \Llat$ and $2\supp_{\G'}(X') \subseteq \Llat'$ (see Lemma
\ref{lem:suppcompare}\,(b)).

\begin{corollary}
\label{cor:biisomorphism}
Suppose that  the hypotheses of
Theorem \ref{thm:biisomorphism} hold.  In addition assume that
$(\F^\times)^2 = \F^\times$ and that $\G$ and $\G'$ have no 2-torsion.
\begin{itemize}
\item[{\rm (a)}]  If there exists an isograded isomorphism $\ph$ of the $\Llat$-graded
associative algebra with involution $(\cA,-)$ onto the
$\Llat'$-graded associative algebra with involution $(\cA',-)$ such that
\[\ph_\gr(2\supp_\G(X)) = 2\supp_{\G'}(X'),\]
then $\cF$ is bi-isomorphic to $\cF'$ and $\cS$ is bi-isomorphic to $\cS'$.  (Here we are using the notation
$\ph_\gr$ of \pref{pgraph:graded}\,(d).)

\item[{(b)}]  Conversely, if $\rk \ge 3$ and $\cS$ and $\cS'$ are bi-isomorphic,
then there exists an isograded isomorphism $\ph$ as in (a).  \end{itemize}
\end{corollary}

\begin{proof} We assume that $\cS$ has been constructed
using a compatible $\cA$-basis $\set{x_i}_{i\in \ttI}$ for $\xi$ and we use
the notation of \pref{con:uLietorus}.  We make similar assumptions for  $\xi'$ and $\cS'$ using primed notation.
Note that  $\ttJ = \ttJ' = \set{1,\dots,\twor}$.
Let $\Slat = \supp_\G(X)$ and $\Slat' = \supp_{\G'}(X')$.

(a) Our goal is to construct a triple of maps $(\theta,\theta_1,\theta_2)$,
where $\theta: X \to X'$, $\theta_1 : \cA \to \cA'$ and $\theta_2 : \G \to \G'$
satisfy conditions \eqref{eq:bi1}--\eqref{eq:bi3}  of Theorem \ref{thm:biisomorphism}\,(a).

By Lemma \ref{lem:suppcompare}\,(a)  and (b), we have the assumptions of
Lemma \ref{lem:group} for $\ph_\gr: \Llat \to \Llat'$.  Hence, there is  an extension
$\theta_2 : \G \to \G'$ of $\ph_\gr$
such that $\theta_2(\Slat) = \Slat'$.  But by Lemma \ref{lem:suppcompare}\,(c), it follows that
$\Slat = \distu_{i\in \ttK} (\rho_i + \Llat)$ and
$\Slat' = \distu_{i\in \ttK'} (\rho'_i + \Llat')$.  So there exists a bijection
$\pi : \ttK \to \ttK'$ such that  $\theta_2(\rho_i) = \rho_{\pi(i)}' +\sg_i'$
where $\sg'_i\in \Llat'$ for $i\in \ttK$. Moreover,
$\pi(k_0) = k'_0$ and $\sg'_{k'_0} = 0$.  We extend $\pi$ to a bijection
$\pi : \ttI \to \ttI'$ by defining $\pi(i) = i$ for $i\in \ttJ$, and we set
$\sg_i' = 0$ for $i\in \ttJ$.  Then,
\begin{equation}
\label{eq:bin0}
\theta_2(\rho_i) = \rho_{\pi(i)}' +\sg_i'
\end{equation}
for $i\in \ttI$, so
\begin{equation}
\label{eq:bin1}
\ph_\gr(2\rho_i) = 2\rho_{\pi(i)}' +2\sg_i',
\end{equation}
for $i \in \ttI$. Now we choose
\begin{equation}
\label{eq:bin2}
0\ne \beta_i'\in {\cA'}^{\sg_i'} \quad\text{for $i\in \ttI$},
\end{equation}
with $\beta'_i = 1$ for $i\in \ttJ\cup\set{k_0}$. But
by \eqref{eq:rhogam}, we have
$\gm_i \in {\cA}^{2\rho_i}$ and  $\gm'_{\pi(i)} \in {\cA'}^{2\rho'_{\pi(i)}}$ for $i\in \ttI$.
Hence,  by \eqref{eq:bin1} and \eqref{eq:bin2}, we have
$\ph(\gm_i) = \F^\times \overline{\beta_i'} \gamma'_{\pi(i)} \beta_i'$
for $i\in \ttI$.
Since by assumption ${\F^\times}^2 = \F^\times$,  we can alter our choice of  $\beta_i'$   so that
\begin{equation}
\label{eq:bin3}
\ph(\gm_i) = \overline{\beta_i'} \gamma'_{\pi(i)} \beta_i'.
\end{equation}
holds for $i\in \ttI$. Further, by \eqref{eq:gammacond},
we may take  $\beta'_i = 1$ for $i\in \ttJ\cup\set{k_0}$, and hence
 \begin{equation}
\label{eq:bin4}
\beta_\prmi'= \beta_i'
\end{equation}
for  $i\in \ttI$.

Finally, let $\theta_1 = \ph$, and define $\theta : X \to X$ by
\begin{equation*}
\label{eq:bin5}
\textstyle \theta(\sum_{i=1}^\ell  x_i.\al_i) = \sum_{i=1}^\ell  x'_{\pi(i)}.(\beta_i'\ph(\al_i)).
\end{equation*}
It is a straightforward matter using
\eqref{eq:bin0}--\eqref{eq:bin4} to check conditions \eqref{eq:bi1}--\eqref{eq:bi3}
in Theorem  \ref{thm:biisomorphism}~(b).  We leave this to the reader.

(b) For the converse, suppose that $\psi$ is a bi-isomorphism of $\cS$ onto $\cS'$.
Let $(\theta,\theta_1,\theta_2)$ be the triple promised by
Theorem \ref{thm:biisomorphism}.  Then, by \eqref{eq:bi3}, $\ph := \theta_1$
is an isograded isomorphism of $(\cA,-)$ onto $(\cA',-)$ with $\ph_\gr = \theta_2|_\Llat$.
Finally, by \eqref{eq:bi3}, $\theta_2(S) = S'$, so $\ph_\gr(2S) = 2S'$.
\end{proof}

\section[Lie $n$-tori]{\cm LIE $n$-TORI AND EXTENDED AFFINE LIE ALGEBRAS OF TYPE $\BCr$}
\label{sec:EALA}

If $\G$ is a free abelian group of finite rank $n$, a Lie $\G$-torus is  called an \emph{Lie $n$-torus}.
It is known that Lie $n$-tori are the starting point for the construction
of extended affine Lie algebras (EALAs) of nullity $n$ (see Section \ref{subsec:EALA}).
With that as motivation,   we apply our results from  Chapter \ref{sec:main} to the special case
of Lie $n$-tori to obtain a classification up to bi-isomorphism of centreless Lie $n$-tori of
type $\BCr$ for $\rk \ge 3$.  We conclude by discussing the construction
of EALAs from these Lie $n$-tori.

If $n=0$, it is well known that a centreless Lie $n$-torus $\cL$ is a finite-dimensional split simple Lie algebra
(see for example \cite[Rem.~1.2.4]{ABFP}).  Since that case is well understood,
\emph{our focus is the case  when $n$ is a positive
integer, and we make that assumption for the rest of this  chapter}.

\subsection{Associative $n$-tori with involution}\
\label{subsec:associativentori}

If $\Llat$ is a free abelian group of rank $n$, an associative $\Llat$-torus with involution
is called an \emph{associative $n$-torus with involution}.  In view of
Theorem \ref{thm:structure}, our first step must be to understand these graded algebras with involution.
Fortunately, they have been classified by Yoshii \cite[Thm.~2.7]{Y2} using elementary quantum matrices.  This classification
can be formulated using quadratic forms over $\bbZ_2$ \cite{AFY, AF}, and we recall that point of view now.

Throughout this section, \emph{we suppose that $\Llat$ is a free
abelian group of rank $n$}.

\begin{pgraph}
\label{not:tilde}
We let $\tL = \G/2\G$ with canonical map
$\sg \to \tsg$,
in which case $\tL$ is an $n$-dimensional vector space over $\Ztwo = \bbZ/2\bbZ$. If
$\eta : \G \to \G'$ is an isomorphism of free abelian groups of rank $n$, then
$\eta$ induces a vector space isomorphism $\tilde\eta : \tL \to \tL'$
defined by $\widetilde\eta(\tsg) = \widetilde{\eta(\sg)}$; and any vector
space isomorphism from $\tL$ to $\tL'$ is induced by some $\eta$ in this way.
\end{pgraph}

\begin{pgraph}
\label{pgraph:quadform1}  Suppose that  $\kappa : \tL \to \Ztwo$
is a  \emph{quadratic form} on $\tL$ over $\Ztwo$.
Recall that this means that the
map
$\kappa_\ttp : \tL \times\tL \to \Ztwo$ defined by
\[\kappa_\ttp(\tsg,\ttau) = \kappa(\tsg+\ttau) + \kappa(\tsg)+ \kappa(\ttau)\]
is $\Ztwo$-bilinear. (See for example
\cite[\S 5.2]{HO}.) Then $\kappa_\ttp$ is an alternating
form on $\tL$ ($\kappa_\ttp(\tsg,\tsg) = 0$ for $\tsg\in \tL$) called the
\emph{polar form}  of $\kappa$.  If  $\kappa_\ttb : \tL \times \tL \to \Ztwo$ is a $\Ztwo$-bilinear form,
we say that $\kappa_\ttb$ is \emph{compatible}
with $\kappa$ if
\[\kappa_\ttp(\tsg,\ttau) = \kappa_\ttb(\tsg,\ttau) + \kappa_\ttb(\ttau,\tsg)\]
for
$\tsg,\ttau\in \tL$.  It is easy to check  that such a $\kappa_\ttb$ exists (see for example
\cite[5.1.15]{HO}) and is unique
up to the addition of a symmetric bilinear form on $\tL$.
\end{pgraph}

\begin{pgraph}
\label{pgraph:quadform3}
If $\kappa : \tL \to \Ztwo$ is a quadratic form, we
adopt the
notation
\[\rad(\kappa) := \set{\tsg\in\tL \suchthat \kappa(\tsg) = 0 \text{ and } \kappa_\ttp(\tsg,\tL) = 0}\]
for the \emph{radical} of
$\kappa$;
and
\[\iso(\kappa) := \set{\tsg\in\tL \suchthat \kappa(\tsg) = 0}\]
for
the set of \emph{isotropic vectors} for $\kappa$.  The \emph{orthogonal group}
of $\kappa$ is the group
$\Orth(\kappa) := \set{\psi \in \GL(\tL) \suchthat \kappa(\psi(\tsg)) =
\kappa(\tsg) \text{ for } \tsg\in \tL}$
of isometries of $\kappa$.
\end{pgraph}

We now see that associative tori with involution are classified by quadratic forms over   $\Ztwo$.

\begin{proposition} \  \textrm{\cite[Prop.~11.4]{AFY}}
\label{prop:assocclass}
 If $(\cA,-)$ is an associative  $\Llat$-torus with involution,
then there exists a unique
quadratic form $\kappa : \tL \to \Ztwo$, called the \emph{mod-2 quadratic form} for $(\cA,-)$,
such that
\begin{equation}
\label{eq:propkappa}
\overline{\al_\sg} = (-1)^{\kappa(\tsg)} \al_\sg \andd
\al_\sg \al_\tau = (-1)^{\kappa_\ttp(\tsg,\ttau)} \al_\tau \al_\sg
\end{equation}
for $\al_\sg\in\cA^\sg$,
$\al_\tau\in\cA^\tau$, $\sg,\tau\in \Llat$.%
\footnote{As usual, we abuse notation and write an element $\mu+2\bbZ$ in $\Ztwo$ as
$\mu$, in which case $(-1)^\mu$ is well defined.}
Moreover, any quadratic form on $\tL$ arises from some associative  $\Llat$-torus with involution
in this way (see the construction in \pref{con:assocntori} below).
Finally, two associative $n$-tori are isograded isomorphic if and only if
their mod-2 quadratic forms are isometric.
\end{proposition}

\begin{pgraph}
\label{not:L+}
Suppose $(\cA,-)$ is an associative  $\Llat$-torus
with mod-2 quadratic form $\kappa$. There are
two
subsets of $\Llat$ that will play  a key role in what  follows.
First define
\[\Gm =\Gm(\cA,-) = \supp_\Llat (Z(\cA,-)).\]
Second recall that   $L_+ = L_+(\cA,-) = \supp_\Llat(A_+)$
(see \pref{pgraph:graded}).
By \eqref{eq:propkappa}  we have
\begin{equation}
\label{eq:L+}
\Gm= \set{\sg\in\Llat \suchthat \tsg\in \rad(\kappa)}
\andd
L_+ = \set{\sg\in\Llat \suchthat \tsg \in \iso(\kappa)}.
\end{equation}
\end{pgraph}

We now describe a construction of an associative $n$-torus with a given mod-2 quadratic form.

\begin{pgraph}
\label{con:assocntori} Suppose  that $\kappa : \tL \to \Ztwo$ is a quadratic form and
that $\kappa_\ttb : \tL \times \tL \to \Ztwo$ is a bilinear form that is compatible
with $\kappa$ (see \pref{pgraph:quadform1}).
Let $\cA = \oplus_{\sg\in\Llat} \F t^\sg$, the $\Llat$-graded vector space with $0\ne t^\sg\in \cA^\sg$ for $\sg\in \Llat$;
and define a product and involution on $\cA$ by
\begin{equation}
\label{eq:ATIdef}
t^\sg t^\tau = (-1)^{\kappa_\ttb(\tsg,\ttau)}t^{\sg + \tau}
\andd \overline{t^\sg} = (-1)^{\kappa(\tsg)} t^\sg.
\end{equation}
Then, one checks that $(\cA,-)$ is an associative $\Llat$-torus with involution whose
mod-2 quadratic form is $\kappa$.
We denote
$(\cA,-)$ by $\ATI_\Llat(\kappa,\kappa_\ttb)$ and call it the
\emph{associative $\Llat$-torus with involution determined by $(\kappa,\kappa_\ttb)$}.
Note that by Proposition \ref{prop:assocclass}, $\ATI_\Llat(\kappa,\kappa_\ttb)$ depends up to isograded-isomorphism
only on $\kappa$.

The construction just described takes on a familiar form  if we choose a $\bbZ$-basis
$\set{\sg_1,\dots,\sg_n}$  for $\Llat$.
Let $t_i = t^{\sg_i}$ for $1\le i \le n$.  Then,
$\cA$ is generated as an algebra over $\F$ by $t_1^{\pm 1},\dots,t_n^{\pm1}$; and we have
$\overline{t_i}  = (-1)^{b_i} t_i$ and $t_jt_i =   (-1)^{a_{ij}} t_i t_j $,
where $b_i = \kappa(\widetilde{\sigma_i})\in \Ztwo$ for
$1\le i \le n$ and  $a_{ij} = \kappa_\ttp(\widetilde{\sigma_i},\widetilde{\sigma_j})\in \Ztwo$ for
$1\le i \le n$.
Thus $\ATI_\Llat(\kappa,\kappa_\ttb)$ is the \emph{quantum torus with involution
determined by the vector $((-1)^{b_i})$ and
the matrix $((-1)^{a_{ij}})$} (\cite[\S III.3]{AABGP}, \cite[\S 2]{AG}).
\end{pgraph}

\begin{remark}
\label{rem:quadprop}
Suppose that $\kappa: \tL \to \Ztwo$ is a quadratic form. It is easy to see
that it is possible to choose a bilinear form $\kappa_\ttb : \tL \times \tL \to \Ztwo$  that is compatible
with $\kappa$ and satisfies
\begin{equation}
\label{eq:quadprop1}
\kappa_\ttb(\tL,\rad(\kappa)) = \kappa_\ttb(\rad(\kappa),\tL) = 0.
\end{equation}
In that case, if we let $(\cA,-) =  \ATI_\Llat(\kappa,\kappa_\ttb)$,
we have the convenient property
\begin{equation}
\label{eq:quadprop2}
t^\sg t^\tau = t^\tau t^\sg = t^{\sg + \tau} \qquad \text{for } \sg\in \Gamma(\cA,-), \tau\in\Llat.
\end{equation}
In particular, $Z(\cA,-) = \F[\Gamma]$, the group algebra of $\Gamma$.
\end{remark}

\begin{pgraph}
\label{pgraph:quadclass}
To classify associative $\Llat$-tori with involution, one can proceed as follows:
Fix a basis $\set{\sg_1,\dots,\sg_n}$  for $\Llat$.  Then
write down representatives of the isometry classes of quadratic forms on $\Llat$
using the well-known  classification (see \cite[Chap.~I, \S 16]{D} or \cite[Rem.~5.17]{AFY}).
Next, for each $\kappa$ in the list, choose $\kappa_\ttb$ compatible with $\kappa$ and satisfying
\eqref{eq:quadprop1},
and then construct
the associative $\Llat$-torus with involution $\ATI_\Llat(\kappa,\kappa_\ttb)$. This gives the
list of associative $\Llat$-tori with involution up to isograded-isomorphism,
and each torus in the list satisfies  \eqref{eq:quadprop2}.
\end{pgraph}

\begin{example}
\label{ex:n=3a}
To give a concrete  example,
we suppose that $n=3$ and fix a $\mathbb Z\hbox {\rm-basis}$  $\set{\sg_1,\sg_2,\sg_3}$ for $\Llat$.
The following table lists the five possible quadratic forms $\kappa$ on $\Llat$ up to isometry,
and the corresponding forms $\kappa_\ttp$ and $\kappa_\ttb$, obtained as in \pref{pgraph:quadclass}.
\begin{equation}
\label{table:quad}
\begin{tabular}
[c]{|c |   c |   c|} \hline
$\kappa$ & \quad $\kappa_\ttp$\quad  &$\quad \kappa_\ttb$\quad \\\hline\hline
$0$ & $0$ & $0$\\\hline
$\ell_3$  & $0$ & $0$\\\hline
$\quad \ell_2+\ell_3+\ell_2\ell_3$ \quad &\quad $\ell_2\ell'_3 + \ell_3\ell'_2$\quad &\quad $\ell_2\ell'_3$\quad\\\hline
$ \ell_2\ell_3$ & $\ell_2\ell'_3 + \ell_3\ell'_2$ & $\ell_2\ell'_3$\\\hline
$\ell_3 + \ell_1\ell_2$ & $\ell_1\ell'_2 + \ell_2\ell'_1$ & $\ell_1\ell'_2$ \\\hline
\end{tabular}\
.
\end{equation}
In each row, we have given $\kappa$ by displaying its value at $\sum_{i=1}^3 \ell_i\widetilde{\sigma_i}$, and we have
given $\kappa_\ttp$ and $\kappa_\ttb$ by displaying their values on the pairs
$\left(\sum_{i=1}^3 \ell_i\widetilde{\sigma_i},\sum_{j=1}^3 \ell'_j\widetilde{\sigma_j}\right)$.
The list of associative $3$-tori with involution up to bi-isomorphism is the corresponding
list of tori  $\ATI_\Llat(\kappa,\kappa_\ttb)$.
\end{example}

We finish this section by recording a lemma that will be needed  in the next section.

\begin{proposition} \label{prop:induced} Suppose that
 $(\cA,-)$ is an associative  $\Llat$-torus with mod-2 quadratic form $\kappa$. Then, the map
$\ph \mapsto \widetilde{\ph_\gr}$ is an epimorphism
of the group
of all isograded automorphisms%
\footnote{It is known \cite[Lemma 1.2]{Y2} that any automorphism of $(\cA,-)$ is isograded,
but we will not use that here.}
of $(\cA,-)$
onto the orthogonal group  $\Orth(\kappa)$.
\end{proposition}
\begin{proof}  The proof of \cite[Prop.~1.14]{AY}~(iii) can be easily
adapted to prove this fact.
\end{proof}

\subsection{Centreless Lie $n$-tori of type $\BCr$}\
\label{subsec:Lientori}

In this section, \emph{we assume that $\characteristic(\F) = 0$
and  $(\F^\times)^2 = \F^\times$}.  Note that this includes the important case when $\F$ is algebraically closed of
characteristic 0.

\begin{pgraph}
\label{con:uLientorus}
To construct centreless Lie $n$-tori as unitary Lie tori,  we suppose that

\begin{itemize}
\item[(I1)]    $\rk$ is a positive integer.
\item[(I2)] $\Llat$ is a free abelian group of rank $n$ and
$(\cA,-)$ is an associative $\Llat$-torus with involution.
\item[(I3)]  $\tM$ is a subset of $\iso(\kappa)$ with $0\in \tM$,
where  $\kappa : \tL \to \bbZ_2$  is the mod-2 quadratic form of $(\cA,-)$.
\end{itemize}

Beginning with these ingredients,  we construct a graded hermitian form $\xi : X \times X \to \cA$ over $(\cA,-)$
and hence a unitary Lie $n$-torus $\fsu(X,\xi)$.

First we  build an abelian group $\G$ containing $\Llat$ that will be our grading group for $\xi$.
To do this, we regard   $\Llat$ as a subgroup of the vector space
$\QL:=  \bbQ \otimes_\bbZ \Llat$  over $\bbQ$
by means of the map $\sg \mapsto 1\otimes \sg$.
Let $\Mlat$ denote the inverse image under $\tildemap$ of $\tM$ in $\Llat$. Then,
by \eqref{eq:L+}, we have
\[2\Llat \subseteq \Mlat \subseteq \Llat_+ \subseteq \Llat,\]
where $\Llat_+ = \Llat_+(\cA,-)$. We let
\[\G  = \textstyle  \langle  \frac 12 \Mlat \rangle \]
in $\QL$.
Then $\Llat$ is a subgroup of $\G$, and, since
$\Llat \subseteq \G \subseteq \frac 12 \Llat$, $\G$
is free abelian of rank $n$.

Now let $m$ be the order of the set $\tM$ and $\ell = \twor + m$.
We select  $\tau_1, \dots, \tau_\ell \in  \Mlat \subseteq \Llat_+$ such that $\tau_i = 0$ for $1\le i \le \twor + 1$ and
\[\widetilde{\tau_{\twor + 1}},\dots,  \widetilde{\tau_{\ell}} \ \text{ are the distinct elements of $\tM$}.\]
We also choose nonzero elements $\gm_1,\dots,\gm_\ell\in \cA_+$  with $\gm_i = 1$ for $1\le i \le \twor + 1$ and
$\gm_i \in \cA_+^{\tau_i}$ for $1\le i \le \ell$.
We call the sets $\set{\tau_i}_{i=1}^\ell$
and $\set{\gm_i}_{i=1}^\ell$  \emph{compatible sets of parameters} for $\tM$.

Next let $X = \cA^\ell$, a right $\cA$-module with standard basis $e_1,\dots,e_\ell$, and
define a $\G$-grading on $X$ so that $X$ is a $\G$-graded
$\cA$-module with $\deg(e_i) = \frac 12\tau_i$ for $1 \le i \le \ell$.  Observe that
\begin{equation}
\label{eq:suppX}
\supp_\G(X) = \bigcup_{i=1}^{\ell}\left(\half  \tau_i+\Llat\right) = \half \Mlat.
\end{equation}
Further define a $\G$-graded hermitian form
$\xi : X \times X \to \cA$ by
\[\xi(e_i,e_j) = \delta_{i \prmj}\gm_i\]
for $1\le i,j \le \ell$,  where $\prmi = \twor + 1 - i$ for $1\le i \le \twor$ and
$\prmi = i$ for $\twor + 1 \le i \le \ell$.

Then   Assumptions \ref{asspts:uLietorus}    hold with
\[X_\hyp = \textstyle \bigoplus_{i=1}^{\twor} e_i.\cA \andd X_\an = \bigoplus_{i=\twor+1}^{\ell} e_i.\cA.\]
Indeed, $X_\an$ is finely graded since
the elements $\frac 12 \tau_{\twor + 1},\dots,\frac 12 \tau_\ell$ are distinct modulo $\Llat$;
and it then follows that  $X_\an$ is anisotropic by Proposition \ref{prop:anisotropic}.
The other conditions are easily checked.  So we can construct the
$(Q\times \G)$-graded  Lie algebras $\cF = \ffu(X,\xi)$ and
$\cS = \fsu(X,\xi)$ as in \pref{con:uLietorus}.  In particular,
$\cS$ is a unitary Lie $n$-torus of type $\BCr$ that we say is   the \emph{unitary Lie $n$-torus
constructed from $\rk$, $\Llat$, $(\cA,-)$ and~$\tM$}.
\end{pgraph}

\begin{remark}
It follows from Corollary \ref{cor:biisomorphism}
that $\cF$ and $\cS$ are independent, up to bi-isomorphism, of the
choice of compatible sets of parameters $\set{\tau_i}_{i=1}^\ell$
and $\set{\gm_i}_{i=1}^\ell$ for~$\tM$.
\end{remark}

\begin{pgraph}
\label{pgraph:consimple}
The $(Q\times \G)$-graded  Lie algebras $\cF$ and  $\cS$ constructed in
\pref{con:uLientorus} have simple descriptions as  matrix algebras using the results of
Section \ref{subsec:unitary}.  For this we use matrix notation relative to the
homogeneous $\cA$-basis $\set{e_i}_{i=1}^\ell$ for $X$.   Then, by
\eqref{eq:Esum}, we see that  $\cE = \cE(X,\xi)$ is equal to the algebra $\Mat_\ell(\cA)$ of $\ell\times\ell$
matrices over $\cA$, and hence, since $X$ is free of finite rank over $\cA$, we have
$\cE = \End_\cA(\cA)$.
Thus, by \eqref{eq:Fchar},  we have
\[\cF =  \textstyle \sum_{i,j=1}^\ell u_{ij}(\cA) = \cU,\]
where $u_{ij}(\al) = e_{ij}(\al) - e_{\prmj\,\prmi}(\gm_j^{-1}\overline\al\gm_i)$
and $\cU = \fu(X,-)$ is the unitary Lie algebra of $\xi$.
Alternatively, if we let $J_{\twor} = (\delta_{i\prmj})\in \Mat_{\twor}(\F)$
and
\[G = \begin{bmatrix} J_{2r} & 0 \\ 0 & \diag(\gm_{\twor + 1},\dots, \gm_\ell) \end{bmatrix}\]
(the matrix of $\xi$ relative to the $\cA$-basis $\set{e_i}_{i=1}^\ell$),  then by
\eqref{eq:Fchar1} and \eqref{eq:eij*},
\[\cF = \set{T\in \gl_\ell(\cA) \ \suchthat \ T^* = -T}, \quad\text{where}\quad T^* = G^{-1}\overline T^t G \quad
\text{for $T\in \cE$},\]
$t$ denotes the transpose, and $\gl_\ell(\cA)$ is $\Mat_\ell(\cA)$ under the commutator product.
Moreover,  by
Proposition \ref{prop:propSU}\,(a),
\[\cS = \set{\ T\in \cF \suchthat \tr(T) \equiv 0\ \pmd}.\]
Finally, the $Q$-gradings on $\cF$ and $\cS$ are the root space decompositions
with respect to the adjoint action of $\hd = \bigoplus_{i=1}^\rk \F u_{ii}(1)$,
and the $\G$-gradings on $\cF$ and $\cS$ are obtained by restricting the $\G$-grading on $\cE$
which is given by
\begin{equation}
\label{eq:gradee}
\deg_\G(e_{ij}(\al)) = \textstyle \frac 12 \tau_i - \frac 12 \tau_j + \deg_\G(\al)
\end{equation}
for $1\le i,j\le \ell$ and
$\al$ homogeneous in $\cA$
(see Proposition \ref{prop:gradingmatrix}).
\end{pgraph}

We can now combine our results to prove the following:

\begin{theorem} [\textbf{Classification of centreless  Lie $\boldsymbol{n}$-tori of type $\boldsymbol{\BCr}$, $r\ge 3$}]\
\label{thm:classn} Assume $\F$ is a field  of characteristic 0 with
${\F^\times}^2 = \F^\times$. Suppose that  $n\ge 1$ and $\rk \ge 3$.
Let $\Llat$ be a free abelian group of rank $n$, and let
$\kappa_1,\dots,\kappa_s$ be a list of the distinct quadratic forms
over $\Ztwo$ on $\tL = \Llat/2\Llat$ up to isometry.  For $1\le i \le s$,   let
$(\cA_i,-)$ be an associative $\Llat$-torus with involution whose
mod-2 quadratic form is $\kappa_i$, and let
$\widetilde{M_{i1}},\dots,\widetilde{M_{i q_i}}$ be representatives of the orbits of the orthogonal group
$\Orth(\kappa_i)$ acting on the set of subsets of $\iso(\kappa_i)$
containing 0.  Then, the distinct
centreless Lie $n$-tori of type $\BCr$ are, up to bi-isomorphism,   the
unitary Lie $n$-tori $\cS_{ij}$ constructed as in \pref{con:uLientorus} from
$\rk$, $\Llat$, $(\cA_i,-)$ and $\widetilde{M_{ij}}$, $1\le i \le s$, $1\le j \le q_i$.
\end{theorem}

\begin{proof}  Suppose first that two of the unitary Lie $n$-tori  are bi-isomorphic, say
$\cS$ constructed from $\rk$, $\Llat$, $(\cA_i,-)$ and  $\widetilde{M_{ij}}$;
and  $\cS'$ constructed from $\rk$, $\Llat$, $(\cA_{i'},-)$ and $\widetilde{M_{i'j'}}$.
By Corollary \ref{cor:biisomorphism}~(b) and \eqref{eq:suppX}, we have an isograded isomorphism
$\ph$ from $(\cA_i,-)$ onto $(\cA_{i'},-)$ such that
$\widetilde{\ph_\gr}(\widetilde{M_{ij}}) = \widetilde{M_{i'j'}}$. So by
Proposition \ref{prop:assocclass}, $i= i'$. Also, by
Proposition \ref{prop:induced}, $\widetilde{M_{ij}}$ and $\widetilde{M_{ij'}}$ are in the same orbit
under $\Orth(\kappa_i)$, and hence  $j=j'$.

Next  suppose that $\cL$ is a centreless Lie $n$-torus of type $\BCr$.
By Theorem \ref{thm:structure}\,(b), we can assume that $\cL = \cS$, where
$\cS = \fsu(X,\xi)$ is the unitary Lie $\G$-torus constructed
using $ \rk$, $\G$ (a free abelian group of rank $n$),
$\Llat'$, $(\cA,-)$ and $\xi : X\times X\to \cA$ as in
\pref{con:uLietorus}.  Also, by Lemma \ref{lem:suppcompare}\,(d),   $\Llat'$ is free of
rank $n$, so we can identify it with the given group $\Llat$.
Let $\Slat = \supp_{\G}(X)$, and set
\begin{equation}
\label{eq:class1}
\Mlat:= 2\Slat \subseteq \Llat_+\subseteq \Llat,
\end{equation}
where $\Llat_+= \Llat_+(\cA,-)$ (see Lemma  \ref{lem:suppcompare}\,(a) and (b)).
Note that $\Slat + \Llat \subseteq \Slat$, so $\Mlat + 2\Llat \subseteq \Mlat$, hence
\begin{equation}
\label{eq:class2}
\text{$\Mlat$  is the inverse image of $\tM$ under $\tildemap : \Llat  \to \tL$}.
\end{equation}

By Proposition \ref{prop:assocclass}, there exists an isograded isomorphism $\ph$ of the
associative $\Llat$-torus with involution $(\cA,-)$ onto the
associative  $\Llat$-torus with involution $(\cA_i,-)$ for some
$1\le i \le s$.   In addition, by \eqref{eq:class1}, $\ph_\gr(\Mlat) \subseteq \ph_\gr(\Llat_+) = \Llat_+(\cA_i,-)$,
and hence
$\widetilde{\ph_\gr}(\tM)  \subseteq \widetilde{\Llat_+(\cA_i,-)} = \iso(\kappa_i)$.
Since $0\in \widetilde{\ph_\gr}(\tM)$,
Proposition \ref{prop:induced} tells us that
there exists an isograded automorphism $\psi$ of $(\cA_i,-)$ such that
$\widetilde{\psi_\gr}(\widetilde{\ph_\gr}(\tM)) = \widetilde{M_{ij}}$ for some $1\le j \le q_i$.  Replacing $\ph$
by $\psi\circ \ph$, we can assume that $\widetilde{\ph_\gr}(\tM) =  \widetilde{M_{ij}}$.
Thus, letting $M_{ij}$ be the inverse image of $\widetilde{M_{ij}}$ under
$\tildemap : \Llat \to \tL$, we have, by \eqref{eq:class2}, that
$\ph_\gr(\Mlat) = M_{ij}$.   Applying  Corollary \ref{cor:biisomorphism}\,(a),  \eqref{eq:suppX} for $\cS_{ij}$, and \eqref{eq:class1}, we see that
$\cL$ is bi-isomorphic to  the unitary Lie $n$-torus  $\cS_{ij}$.
\end{proof}

Theorem \ref{thm:classn} reduces the  classification of centreless Lie $n$-tori of type $\BCr$, $r \geq 3$,  up to bi-isomorphism to two problems:
the classification of $n$-dimensional quadratic forms over $\Ztwo$ up to isometry;
and, given such a quadratic form $\kappa : \tL \to  \Ztwo$, the determination of the orbits of
the orthogonal group $\Orth(\kappa)$ acting on the set of
subsets containing 0 of $\iso(\kappa)$.
The first of these problems has a well-known  solution as we have  mentioned, while  the second can, at least in some cases
(in particular, for small $n$),  be worked out directly.
We will examine the case $n=3$
when we consider EALAs in the next section (see Example \ref{ex:EALAn=3}).

\subsection{Construction of extended affine Lie algebras of type $\BCr$}\
\label{subsec:EALA}

In this section, we construct some EALAs  of type $\BCr$.
We do not assume here that the reader is familiar with EALAs -- it suffices to know
that an EALA  is a triple $(\rE,\rH,\form)$ consisting of a Lie algebra $\rE$
with a finite-dimensional  distinguished abelian Cartan subalgebra $\rH$  and a nondegenerate invariant
symmetric bilinear  form $\form$ satisfying a natural class of axioms \cite{N2} which model
properties of affine Kac-Moody Lie algebras.  The \emph{nullity}
of $\rE$ is the rank of the finitely generated free abelian group $\G$
generated by the isotropic roots of $\rE$, and the \emph{type} of $\rE$
is the type of the finite irreducible root system obtained by identifying two roots of $\rE$
if they
differ by an element of $\G$.   Besides \cite{N2}, the reader can consult \cite{BGK}, \cite{AABGP}, \cite{N3}
and \cite{AF} and the references therein for more information about this topic.

Generalizing earlier work in type $\text{A}_\rk$ \cite{BGK,BGKN},
Neher \cite{N2} gave a construction of  a family of EALAs of nullity $n$
from a centreless Lie torus $\cL$ of nullity $n$.
Roughly speaking  one constructs an EALA $\rE = \cL \oplus \rC \oplus \rD$,
where $\rD$ is a graded subalgebra of  the Lie algebra $\SCDer(\cL)$ of skew-centroidal
derivations of $\cL$ and $\rC$ is the graded dual of $\rD$,
and where the product on $\rE$ involves a 2-cocycle on $\rD$ with values in $\rC$.
Varying the algebra $\rD$ and the 2-cocycle, one obtains a family of EALAs determined by $\cL$.
It was
announced
in \cite{N2}  that any EALA occurs in the family constructed from some $\cL$ in this way.
Moreover, it was shown in   \cite[Cor.~6.3]{AF} that
if two centreless Lie $n$-tori $\cL$ and $\cL'$ are bi-isomorphic,  then the EALAs in the corresponding families
are pairwise isomorphic.

If $\cL$ is a centreless Lie $n$-torus of type $\Dl$,  where $\Dl$ is reduced, then
the corresponding EALAs have type $\Dl$.
On the other hand, if $\Dl$ has type $\BCr$ and $\supp_Q(\cL) = \Dl\cup \set{0}$,
then the corresponding EALAs have type $\BCr$.
So, it follows from Remark \ref{rem:Btype} that
all EALAs of nullity $n$ and type $\Br$ and $\BCr$ are obtained from
a centreless Lie $n$-torus of type $\BCr$.  Moreover,  if $\rk \ge 3$, it follows from
Theorem \ref{thm:structure} that there is no loss of generality
in starting with a unitary Lie $n$-torus $\cS$.
For the sake of readers interested in working concretely with EALAs, in this section
we review Neher's construction, without proofs, in this case.    Instead of constructing a family of EALAs from $\cS$,
we focus on just one EALA,   which is \emph{maximal} in the sense that
$\rD = \SCDer(\cS)$  (the full set of skew-centroidal derivations)   is used in the construction; and we use the
trivial 2-cocycle on $\rD$ to define the multiplication.
The reader familiar with \cite{N2} will have no trouble obtaining the whole family
in the same way.

Throughout this section, \emph{we will assume that $\F$ is
a field of characteristic 0 with  $(\F^\times)^2 = \F^\times$, $\rk \ge 3$, and that
$\cF = \ffu(X,\xi)$ and $\cS = \fsu(X,\xi)$ are the $(Q\times \G)$-graded  Lie algebras  constructed from $\rk$, $\Llat$, $(\cA,-)$ and $\tM$ as in \pref{con:uLientorus}}
using compatible sets of parameters  $\set{\tau_i}_{i=1}^\ell$
and $\set{\gm_i}_{i=1}^\ell$, where $\ell = \twor + m$, and $m$ is the order of $\tM$.
We let $\cE = \fe(X,\xi)$, and we
view $\cE$, $\cF$ and $\cS$ as graded matrix algebras as in \pref{pgraph:consimple}.

\begin{pgraph}
\label{pgraph:conEALA}
As in Theorem \ref{thm:SU}\,(d), we
define a $\G$-graded nondegenerate associative symmetric bilinear form $\form$ on $\cE$
by
\[(T_1\mid T_2) = \varpi(\tr(T_1,T_2))\]
for $T_1,T_2\in \cS$, where  $\varpi : \cA \to \F$ is the $\Llat$-graded projection
of $\cA$ onto $\F$ (see Theorem \ref{thm:SU}\,(d)). Then,
$\form$ restricts to a $\G$-graded nondegenerate  invariant form on $\cF$ and $\cS$.

Let $Z = Z(\cA,-)$; \,  $\kappa : \Llat \to \Ztwo$ be the mod-2 quadratic form of $(\cA,-)$;  and
\[\Gm := \Gm(\cA,-) = \supp_\Llat(Z) \andd \Llat_+ := \Llat_+(\cA,-).\]
By \eqref{eq:L+}, we have
\[\Gm :=  \set{\sg\in \Llat \suchthat \tsg \in \rad(\kappa)} \andd
\Llat_+ =  \set{\sg\in \Llat \suchthat \tsg \in \iso(\kappa)}.\]

Recall next that $\cE$ is a left $Z$-module (see \pref{pgraph:Zaction}), and we
have $\fz e_{ij}(\al) = e_{ij}(\fz\al)$ and $\fz u_{ij}(\al) = u_{ij}(\fz\al)$ for $\fz\in Z$, $1\le i,j\le \ell$
and $\al\in \cA$.  Furthermore, by Proposition \ref{prop:centroid} and Proposition \ref{prop:propSU}\,(b),
the map $\fz\to \Lmult_\fz|_\cS$
is an isomorphism of
$Z$ onto the centroid $\Cent(\cS)$ of $\cS$.  This map is $\G$-graded
so $\supp_\G(\Cent(\cS)) = \Gm$; that is,  $\Gm$ is the centroidal grading group of $\cS$
(see \textbf{\ref{rem:LTsupport}}~(c)).

Let
$\Hom(\G,\F)$ be the group of group homomorphisms of $\G$ into $\F$, and let
$\Der_\F(\cS)$ be the algebra of derivations of $\cS$. For
$\theta\in \Hom(\G,\F)$, we define the \emph{degree derivation}
$\partial_\theta\in \Der_\F(\cS)$
by $\partial_\theta|_{\cS_\sg} = \theta(\sg)\id_{\cS_\sg}$ for $\sg\in \G$.
Setting
\[\cD = \partial_{\Hom(\G,\F)},\]
we see that
$\cD$ is an $n$-dimensional abelian subalgebra of $\Der_\F(\cS)$.

Now $\Der_\F(\cS)$ is a left $Z$-module under the
action
$\fz d = \Lmult_\fz\circ d$ for $\fz\in Z$ and
$d\in \Der_\F(\cS)$.
Moreover, the space
\[\CDer(\cS) = Z\cD\]
is a subalgebra of the Lie algebra $\Der(\cS)$,   and
$\CDer(\cS)$ is a free $Z$-module of rank $n$
which is
$\Gm$-graded with
$\CDer(\cS)^\sg= Z^\sg \cD$
for $\sg\in \Gm$.
Let
\[\rD := \SCDer(\cS), \]
the $\Gm$-graded subalgebra of  $\CDer(\cS)$ consisting
of the derivations in $\CDer(\cS)$ that are skew relative to the form $\form$.
Then $\rD$ is called the
\emph{Lie algebra of skew-centroidal derivations of $\cS$}.
Since $\rD^0 = \cD$, $\dim(\rD^0) = n$, while  $\dim(\rD^{\sg})= n-1$ if $\sg\in \Gm\setminus{0}$.

Next let
\begin{equation*}
\label{eq:SCDer}
\rC := \rD^\grd  = \bigoplus_{\sg\in\Gm} (\rD^\sg)^* \subseteq \rD^*,
\end{equation*}
be the graded-dual space of $\rD$, where
$(\rD^\sg)^*$ is embedded in
$\rD^*$ by letting its elements act trivially on $\rD^\tau$ for $\tau\ne \sg$.
We give the vector space $\rC$ a $\Gm$-grading by setting
\[\rC^\sg = (\rD^{-\sg})^*,\]
in which case, $\rC$ is a $\Gm$-graded $\rD$-module
by means of the contragradient action~``$\ast$''  given by
\[(d \staraction c)(e) = - c([d,e]),\]
for $d, e \in \rD$, $c\in \rC$.
Now define $\varsigma: \cS \times \cS \to \rC$ by
\[\varsigma(T_1,T_2)(d) = (dT_1|T_2).\]
Then $\varsigma$ is a $\Gm$-graded $2$-cocycle on $\cS$  with values in
the trivial $\cS$-module  $\rC$.

With these ingredients, we are ready to build an EALA.    Set
\[\rE = \cS\oplus \rC \oplus\rD,\]
where $\rC = \rD^\grd$.
We  identify $\cS$, $\rC$  and $\rD$
with  subspaces of $\rE$,
and define a product $[\, ,\, ]_\rE$ on $\rE$ by
\begin{multline}
\label{eq:EALAprod}
[T_1+c_1+d_1,T_2+c_2+d_2]_\rE \\
=  \left([T_1,T_2]+d_1(T_{2})-d_{2}(T_1)\right)
+ \left(d_1 \staraction c_{2}-d_2 \staraction c_1+\varsigma(T_1,T_{2})\right) + [d_1,d_2]
\end{multline}
for $T_i\in \cS$, $c_i \in \rC$, $d_i\in \rD$.
Then,  $\rE$ is a $\G$-graded Lie  algebra with the direct sum grading,
and
\[\rH :=  \hd\oplus \rC^0\oplus \rD^0\]
 is an abelian subalgebra of $\rE$, where  $\hd = \bigoplus_{i=1}^\rk \F u_{ii}(1)$.
Finally we extend the bilinear form $\form$ on $\cS$ to a  graded bilinear form
$\form$ on $\rE$ by defining
\begin{equation}
\label{eq:EALAform}
(T_1+c_1+d_1\mid T_2+c_2+d_2)=(T_1\mid T_{2})+c_1(d_{2})+c_{2}(d_1).
\end{equation}
Then, Neher's theorem \cite[Thm.~6]{N2} tells us that
$(\rE,\form,\rH)$ is an  EALA
of nullity $n$ and type $\BCr$ or $\Br$.  Moreover, by that same theorem
and our Theorem \ref{thm:classn},
we have constructed a maximal EALA of nullity $n$ in each family
of EALAs of nullity $n$ and type $\BCr$ or $\Br$.

Finally, we note that  the EALA $\rE$  is of type $\Br$ if and only if $\kappa = 0$.
\end{pgraph}

\begin{pgraph} \ptitle{Computations using bases}
\label{pgraph:conEALA2} To facilitate working with
the EALA $\rE$ constructed in \pref{pgraph:conEALA},  we now show how to select
convenient bases for $\Llat$ and $\G$ and use coordinates to obtain
expressions for the product and form on $\rE$. This follows the original
approach
in \cite{BGK} and \cite{BGKN} which treated  type $\text{A}_\rk$, $\rk \ge 2$.

We may assume that  $(\cA,-) = \ATI_\Llat(\kappa,\kappa_\ttb)$,
where $\kappa_\ttb$ is chosen as in Remark \ref{rem:quadprop}, in which case we have
\begin{equation}
\label{eq:quadprop2again}
t^\sg t^\tau = t^\tau t^\sg = t^{\sg + \tau} \qquad \text{for } \sg\in \Gamma, \tau
\in \Llat.
\end{equation}
Now choose a basis $\set{\sg_1,\dots,\sg_n}$ for $\Llat$ such that
\begin{equation}
\label{eq:nicebasis}
\begin{gathered}
\set{\sg_1,\dots,\sg_{n_1}, 2\sg_{n_1+1},\dots, 2\sg_n} \text{ is a basis for $\langle \Mlat \rangle$},\\
\end{gathered}
\end{equation}
where $\sg_1,\dots,\sg_{n_1}\in \Mlat$ and $0\le n_1\le n$.  (To obtain such a basis,
chose a basis for $\tL$ with the appropriate properties
and lift to a basis for  $\Llat$.)
Then, $\set{\lm_1,\dots,\lm_n}$  is a basis for $\G = \langle \frac 12 \Mlat\rangle$, where
\[\lm_i = \frac 12 \sg_i \ \text{ for } \  1 \le i \le n_1, \andd \lm_i = \sg_i  \  \text{ for } \  n_1+1\le i \le n.
\]

Next we introduce  coordinates in $\cD$ and $\CDer(\cS)$.  First, for $\lm = \sum_{i=1}^n \ell_i\lm_i\in \G$,
we set
\[\lm^\# = (\ell_1,\dots,\ell_1) \in \F^n.\]
Further, let $\set{\theta_1,\dots,\theta_n}$
be the dual basis for $\set{\lm_1,\dots,\lm_n}$ in $\Hom(\G,\F)$, and for $\bolds = (s_1,\dots,s_n) \in \F^n$
write  $\theta_\bolds = \sum_{i=1}^n s_i\theta_i$.   Then
\begin{equation}
\label{eq:thetas}
\theta_\bolds(\lm) = \lm^\# \cdot \bolds
\end{equation}
for $\bolds\in \F^n$ and $\lm\in \G$, where $\cdot$ is the usual dot product on $\F^n$.
Also,  if $\bolds  \in \F^n$, let $\partial_\bolds = \partial_{\theta_\bolds}$. Then
\[\cD = \set{\partial_\bolds \suchthat \bolds \in \F^n};\]
$\CDer(\cS) = \sum_{\sg\in\Gm} \CDer(\cS)^\sg$ with
$\CDer(\cS)^\sg = t^\sg \cD$
for $\sg\in \Gm$; and one checks using  \eqref{eq:quadprop2again} that
the multiplication in $\CDer(\cS)$ is given by
\begin{equation}
\label{eq:EALAc1}
[t^\sg\partial_\bolds, t^\rho\partial_\boldr] =
t^{\sg+\rho}( (\rho^\# \cdot \bolds)\partial_\boldr
-  (\sg^\# \cdot \boldr)\partial_\bolds ).
\end{equation}
The homogeneous components of the algebra $\rD = \SCDer(\cS) = \bigoplus_{\sg\in\Gm} \rD^\sg$
of skew-centroidal derivations of $\cS$ are given by
\begin{equation*}
\rD^\sg =  \textstyle \bigoplus_{\bolds \in \F^n,\, \sg^\#\cdot \bolds = 0} \F t^\sg \partial_\bolds.  \end{equation*}
for $\sg\in \Gm$.

To discuss $\rC = \rD^\grd$
we define $\rc_\sg(\bolds)\in \rC^\sg = (\rD^{-\sg})^*$ for $\bolds \in \F^n$ and
$\sg\in \Gm$ by
\begin{equation}
\label{eq:EALAc2}
\rc_\sg(\bolds) (t^{-\sg} \partial_\boldr) = \bolds\cdot\boldr
\end{equation}
for $\boldr \in \F^n$ with $\sg^\# \cdot \boldr = 0$.
Then,  $\bolds \mapsto \rc_\sg(\bolds)$ determines a linear map
of $\F^n$ into $\rC^\sg$ with kernel $\F \sg^\#$.  So by counting dimensions, we see that
\[\rC^\sg = \set{\rc_\sg(\bolds) \suchthat \bolds\in \F^n}.\]
The action of $\rD$ on $\rC$ is given by
\begin{equation}
\label{eq:EALAc3}
(t^\rho\partial_\boldr) \staraction (\rc_\sg(\bolds)) =
\rc_{\sg + \rho}((\sg^\#\cdot \boldr)\bolds + (\bolds\cdot\boldr)\rho^\#).
\end{equation}
for $\sg\in \Gm$, $\bolds\in \F^n$, and $\rho\in\Gm$, $\boldr\in \F^n$ with $\rho^\#\cdot \boldr = 0$.

To calculate the action of $\CDer(\cS)$ on $\cS$ using coordinates, it is  natural to
extend the action  of $\CDer(\cS)$ to $\cE = \Mat_\ell(\cA)$; and then restrict to $\cS$.
Indeed, since $\cE$ is $\G$-graded we can define degree derivations
$\set{\partial_\theta^\cE}_{\theta\in \Hom(\G,\F)}$ of $\cE$ just as for $\cS$.
Moreover, $\Der_\F(\cE)$ is naturally  a left $Z$-module and we have the
$\G$-graded Lie algebra $Z \partial^\cE_{\Hom(\G,\F)}$ in $\Der_\F(\cE)$.
Letting $\partial_\bolds^\cE = \sum_{i=1}^n s_i \partial_{\theta_i}^\cE$
for $\bolds = (s_1,\dots,s_n) \in \F^n$, we see  that the restriction map
$t^\sg \partial^\cE_\bolds \mapsto t^\sg \partial_\bolds$  is an algebra isomorphism
sending
$Z \partial^\cE_{\Hom(\G,\F)}$ onto $\CDer(\cS)$.   Henceforth, we treat
this map as an identification.  This gives an action of $\CDer(\cS)$ on $\cE$, and one checks
directly  using \eqref{eq:gradee} and \eqref{eq:thetas} that
\begin{equation}
\label{eq:CDeraction}
(t^\sg \partial^\cE_\bolds) (e_{ij}(t^\tau))
= \textstyle \bolds\cdot(\tau + \frac 12 \tau_i - \frac 12 \tau_j)^\# e_{ij}(t^{\sg+\lm})
\end{equation}
for $\sg\in\Gm$, $\bolds\in \F^n$ with $\sg^\#\cdot \bolds = 0$ and $\tau\in \Llat$.
By restriction, this determines  the action of $\CDer(\cS)$ on $\cS$,
and hence the action of $\rD$ on $\cS$.

To calculate the 2-cocycle $\varsigma: \cS \times \cS \to \rC$, it is natural to proceed
in a similar fashion  and  extend
this map to a bilinear map
$\varsigma : \cE \times\cE \to \rC$ defined by
\[\varsigma(T_1,T_2)(d) = (dT_1, T_2)\]
for $T_1,T_2\in \cE$, $d\in \rD$;
and then restrict to $\cS\times\cS$.  A
direct calculation using \eqref{eq:ATIdef},  \eqref{eq:quadprop2again} and \eqref{eq:CDeraction}
shows  that
\begin{multline}
\label{eq:EALAc4}
\varsigma (e_{ij}(t^\sg), e_{pq}(t^\tau))
\\=
\left\{
  \begin{array}{ll}
    \textstyle
\delta_{iq}\delta_{jp} (-1)^{\kappa_\ttb(\ttau,\ttau)}
\rc_{\sg+\tau}(( \sg + \frac 12 \tau_i - \frac 12 \tau_j)^\#), & \hbox{if $\sg + \tau \in \Gm$} \\
    0, & \hbox{otherwise}
  \end{array}
\right.
\end{multline}
for $\sg,\tau\in \Llat$
and  $1\le i,j,p,q\le n$.
By restriction, we obtain the 2-cocycle $\varsigma$ on $\cS$.

Equations  \eqref{eq:EALAc1},  \eqref{eq:EALAc2},   \eqref{eq:EALAc3}
and \eqref{eq:EALAc4} now allow us to calculate
products in $\cE$ explicitly using  \eqref{eq:EALAprod} and the form on $\cE$ using  \eqref{eq:EALAform}.
\end{pgraph}

\begin{example}
\label{ex:EALAn=1} Suppose that $n=1$.  Then the EALA $\rE$ constructed  in \pref{pgraph:conEALA} and
\pref{pgraph:conEALA2} is an affine Kac-Moody algebra.  If the mod-2 quadratic form $\kappa$
of $(\cA,-)$ is not 0, then $\widetilde{\Llat_+} = 0$, so $\tM = 0$.  In this case
$\rE$ has  affine type $\text{A}_\twor^{(2)}$ in the notation of \cite{K}.
On the other hand if $\kappa = 0$, then
$\widetilde{\Llat_+} = \tL$, so $\tM = 0$ or $\tL$.  The EALA $\rE$ is then of
affine type $\Br^{(1)}$ or $\text{D}_{\rk+1}^{(2)}$ respectively.
(In the notation of \cite{MP}, which is more natural in this context, the three affine types occurring here are
in order: $\BCr^{(2)}$, $\Br^{(1)}$ and $\Br^{(2)}$.)
\end{example}

\begin{example}
\label{ex:EALAn=3}
Suppose that $\rk \ge 3$ and $n=3$.  Let
$\Llat$ be a free abelian group  with basis $\set{\sg_1,\sg_2,\sg_3}$.  The  five $3$-dimensional
quadratic forms over $\Ztwo$ were listed in Table \ref{table:quad}. We consider the case when
$(\cA,-) = \ATI_\Llat(\kappa,\kappa_\ttb)$, where
\[\textstyle \kappa(\sum_{i=1}^n \ell_i\widetilde{\sg_i}) =  \ell_3 + \ell_1\ell_2
\andd \kappa_\ttb(\sum_{i=1}^n \ell_i\widetilde{\sg_i},\sum_{i=1}^n \ell'_i\widetilde{\sg_i}) =  \ell_1\ell'_2.\]
We  have $\rad{\kappa} = 0$ and $\iso(\kappa)
= \set{0, \widetilde{\sg_1},\widetilde{\sg_2}, \widetilde{\sg_1} +\widetilde{\sg_2}+\widetilde{\sg_3}}$.
One can check (for example using facts about $\Orth(\kappa)$ from
\cite[\S 1.16]{D}) that, for $1\le m \le 4$, $\Orth(\kappa)$ acts transitively
on the set of order $m$ subsets of  $\iso(\kappa)$ containing~0.
So, by Theorem \ref{thm:classn}, for $1\le m \le 4$ there is up
to bi-isomorphism exactly one centreless Lie $3$-torus $\cS$ of type $\BCr$ with coordinate torus
$(\cA,-)$ and anisotropic rank $m$.

We consider the case when $m=3$  and we select
\[\tM = \set{0, \widetilde{\sg_1}, \widetilde{\sg_2}}.\]
Then   $\langle \Mlat \rangle  = \bbZ \sg_1 \oplus \bbZ \sg_2 \oplus 2\bbZ \sg_3$,
so the given basis $\set{\sg_1,\sg_2,\sg_3}$ satisfies \eqref{eq:nicebasis} with $n_1 = 2$.
Now $\ell = \twor + 3$, and we let
\[\tau_{\twor + 2} = \sg_1,\quad  \tau_{\twor + 3} = \sg_2,\quad \tau_{i} = 0 \  \text{ for } \  1\le i \le \twor+1,\]
and
\[\gm_{\twor + 2} = t_1,\quad  \gm_{\twor + 3} = t_2, \quad \gm_{i} = 1\  \text{ for } \  1\le i \le \twor+1,\]
in which case
$\set{\tau_i}_{i=1}^\ell$
and $\set{\gm_i}_{i=1}^\ell$ are compatible sets of parameters for $\tM$.
We construct graded matrix algebras $\cE$, $\cF$ and $\cS$ from
$\rk$, $\Llat$, $(\cA,-)$ and
$\tM$ (using the above parameter sets) as in \pref{pgraph:consimple}.
Then
\[\cS = \set{T\in \gl_{\twor + 3}(\cA) \ \suchthat \ G^{-1}\overline T^t G = -T \text{ and } \tr(T) \equiv 0\ \pmd },\]
where
\[G = \begin{bmatrix} J_{2r} & 0 \\ 0 & \diag(1,t_1,t_2)\end{bmatrix}.\]
The $(Q\times\G)$-gradings  on  $\cF$ and $\cS$ are as described in
\pref{pgraph:consimple} with $\G =\textstyle \langle \frac 12 \Mlat \rangle = \langle \frac 12 \sg_1,\frac 12 \sg_2, \sg_3\rangle$.
By Theorem \ref{thm:classn}, $\cS$ is  the unique (up to bi-isomorphism)
centreless Lie $3$-torus of type $\BCr$ with coordinate torus $(\cA,-)$ and anisotropic rank $3$.

Finally, we let $\rE$ be the EALA constructed as in \pref{pgraph:conEALA}.  Since
$\widetilde{\Gm} = \iso(\kappa) = 0$, we have $\Gm = 2\Llat$ and
$Z = \F[\Gm] = \F[2\Llat]$.
We can now use
coordinates to calculate the product and form in $\rE$ as in \pref{pgraph:conEALA2}.
\end{example}

\section[Conclusions]{\cm CONCLUSIONS}
\label{sec:conclude}

We finish with some remarks on special cases and possible generalizations of our results.
Suppose that $\G$ is arbitrary (unless mentioned otherwise) and $\characteristic(\F) = 0$.

\subsection{Remarks}

\begin{remark} \ptitle{Lie tori of type $\Br$, $\rk \ge 3$}
If we  specialize to the case when
the involution $-$
on the associative torus $\cA$ is the identity,
Theorems \ref{thm:structure} and \ref{thm:biisomorphism} yield
structure and bi-isomorphism theorems for centreless $\G$-Lie tori
of type $\Br$, $\rk \ge 3$. The structure theorem for type $\Br$
was proved (in a different form) when $\G = \bbZ^n$ and $\F = \mathbb C$
in \cite{AG}, and it is a special case of a more general structure theorem
for division graded Lie algebras of type $\Br$ due to Yoshii \cite{Y3}.
However, the bi-isomorphism theorem is new even for  type~$\Br$.

Also,  if we specialize to the case when
the quadratic form is trivial in the classification result,  Theorem \ref{thm:classn},
we obtain a classification of the centreless Lie $n$-tori of type $\Br$, $\rk\ge 3$, up to bi-isomorphism.
\end{remark}

\begin{remark}  As we have noted earlier,  in the construction
of a family of EALAs from a centreless Lie $n$-torus,
two centreless Lie $n$-tori that are bi-isomorphic give isomorphic families of EALAs \cite[Cor.~6.3]{AF}.
However, the converse is not true.    To obtain a one-to-one correspondence
between centreless Lie $n$-tori
and families of EALAs of nullity $n$
one needs a coarser equivalence relation on centreless Lie $n$-tori
called \emph{isotopy} [ibid].
Therefore, it would be beneficial  to have a version of
the classification result, Theorem \ref{thm:classn},   with isotopy replacing bi-isomorphism.  With the results
for the
other types of root systems in \cite{AF}  as a model, we expect that Theorem \ref{thm:classn} will be the first main step in obtaining such
a result for type $\BCr$.
\end{remark}

\begin{remark}
\label{rem:IARA}
A generalization of Lie $\G$-tori of type $\Dl$
which arose out of Yoshii's work   is the class of
\emph{pre-division $(\Dl,\G)$-graded Lie algebras} (see \cite[\S 2]{Y3}, \cite[\S 5.1]{N3}).
Such a graded Lie algebra is said to be \emph{invariant} if it possesses a suitable
invariant bilinear form.    Correspondingly,  Neher
has introduced an interesting class of algebras generalizing
EALAs called \emph{invariant affine reflection algebras} (IARAs) \cite[\S 6.7]{N3}.
In \cite[Thm.~6.10]{N3},  Neher has announced
that a family of IARAs can be constructed
from an invariant  pre-division $(\Dl,\G)$-graded Lie algebra whenever $\G$ is torsion free.  In particular,
since a unitary Lie $\G$-torus has an invariant form by Theorem \ref{thm:SU},
it can be used to construct a family
of IARAs   when $\G$ is torsion free.

In light of Neher's work discussed above,  it would be useful to prove a generalization of our structure theorem and
bi-isomorphism theorem for invariant pre-division $(\BCr,\G)$-graded Lie algebras, $\rk \ge 3$.
Such results would be especially interesting in the
particular case of an invariant \emph{division} $(\BCr,\G)$-graded Lie algebra,  since they would naturally
include our main theorems as well as the
finite-dimensional result mentioned at the beginning of this paper (when $\G = 0$).
We expect that the results and arguments
in the division case would  closely follow the ones   in the present  work.
\end{remark}

\bibliographystyle{amsalpha}

\end{document}